\documentclass[11pt]{amsart}
\usepackage{amsmath,amssymb,amsfonts,url}
\usepackage{mathptmx}
\usepackage{color}
\usepackage{appendix}
\numberwithin{equation}{section}
\newtheorem{Def}{Definition}[section]

\newtheorem{thm}{Theorem}[section]
\newtheorem{lem}{Lemma}[section]
\newtheorem{rem}{Remark}[section]
\newtheorem{prop}{Proposition}[section]
\newtheorem{cor}{Corollary}[section]

\usepackage{palatino,mathpazo}

\newcommand{\e}{\varepsilon}

\usepackage{amsthm,a0size,bm,indentfirst,wasysym}

\usepackage{amsmath}

\allowdisplaybreaks[4]

\begin{document}
	\title[Non degeneracy of blow-up solution]{Non degeneracy of blow-up solutions of non-quantized singular Liouville-type equations and the convexity of the mean field entropy of the Onsager vortex model with singular sources}
	\keywords{Liouville equation, blow up, non quantized singular source, non degeneracy}
	
	
	\author[D. Bartolucci]{Daniele Bartolucci}
	\address{Daniele Bartolucci, Department of Mathematics, University of Rome "Tor Vergata", Via della ricerca scientifica 1, 00133 Roma,
		Italy}
	\email{bartoluc@mat.uniroma2.it}
	
	\author[W. Yang]{Wen Yang}
	\address{Wen Yang, Department of Mathematics, Faculty of Science and Technology, University of Macau, Taipa, Macau}
	\email{wenyang@um.edu.mo}
	
	\author[L. Zhang]{Lei Zhang}
	\address{Lei Zhang, Department of Mathematics, University of Florida, 1400 Stadium Rd, Gainesville, FL 32611, USA}
	\email{leizhang@ufl.edu}

	\date{\today}
	\begin{abstract}
		
		We establish the non-degeneracy of bubbling solutions for singular mean field equations when the blow-up points are either regular or involve non-quantized singular sources. This extends the results from Bartolucci-Jevnikar-Lee-Yang \cite{bart-5}, which focused on regular blow-up points. As a consequence, we establish the strict convexity of the Entropy in the large energy limit for a specific class of two-dimensional domains in the Onsager mean field vortex model with singular sources.
	\end{abstract}
	
	
	
	\maketitle
	
	\section{Introduction}
	In this article, we shall investigate the non degeneracy property of the following mean field equation with singular source defined on the Riemann surface $M$
	\begin{equation}\label{m-equ}
		\Delta_g \nu+\rho\bigg(\frac{he^\nu}{\int_M h e^\nu{\rm d}\mu}-\frac{1}{vol_g(M)}\bigg)=\sum_{j=1}^N 4\pi \alpha_j \left(\delta_{q_j}-\frac{1}{vol_g(M)}\right)~{\rm in}~M,
	\end{equation}
	along with the corresponding Dirichlet problem
	\begin{equation}\label{equ-flat}
		\begin{cases}
			\Delta \nu+\rho \dfrac{he^\nu}{\int_{\Omega} h e^\nu{\rm d}x}=\sum\limits_{j=1}^N 4\pi \alpha_j \delta_{q_j}  &{\rm in} \;\ \Omega,
			\\
			\\
			\nu=0  &{\rm on} \;\ \partial\Omega,
		\end{cases}
	\end{equation}
	where $(M,g)$ in \eqref{m-equ} represents a Riemann surface with metric $g$, and $\Omega$ in \eqref{equ-flat} is an open, bounded domain in $\mathbb{R}^2$ with a $C^2$ boundary $\partial\Omega$, $\Delta_g$ denotes the Laplace-Beltrami operator on $M$, while $\Delta$ stands for the classical Laplacian operator in Euclidean space, the potential function $h$ is a positive $C^5$ function and the parameter $\rho$ is strictly positive in both equations, the points $q_1, \cdots, q_N$ are distinct in the respective ambient space, the coefficients $\alpha_j > -1$, and $\delta_{q_j}$ represents the Dirac measure at $q_j$. The volume of $(M,g)$ is denoted by ${vol}_g(M)$, and we assume throughout the article that ${vol}_g(M) = 1$.
	
	
	Both equations have a rich background in Differential Geometry and Physics, including topics such as conformal geometry \cite{KW,troy}, Electroweak and Self-Dual Chern-Simons vortices \cite{ambjorn, spruck-yang,taran-1,taran-2,y-yang}, the statistical mechanics approach to turbulent Euler flows, plasmas, and self-gravitating systems \cite{bart-5,caglioti-2, kiessling, CCKW,SO,TuY,wolan}, Cosmic Strings \cite{BT-2,pot}, the theory of hyperelliptic curves and modular forms \cite{chai}, and CMC immersions in hyperbolic 3-manifolds \cite{taran-4}. These factors motivated the significant efforts dedicated to studying (\ref{m-equ}), focusing on existence results \cite{BdMM,BLin2,BMal,luca-b,dAWZ,chan-fu-lin,chen-lin-sharp,chen-lin-deg,chen-lin,chen-lin-deg-2,DKM,DJ,EGP,lin-yan-cs,lin-yan-cs,Mal1,Mal2,mal-ruiz,nolasco-taran}, uniqueness \cite{bart-4,bart-4-2,bj-lin,BLin3,BLT,chang-chen-lin,GM1,lee-lin-jfa,Lin1,lin-yan-uniq,suzuki,wu-zhang-ccm}, and concentration-compactness and bubbling behavior \cite{BM3,BT,BM,gu-zhang-1,gu-zhang-2,gluck,kuo-lin,li-cmp,ls,T3,taran-3,wei-zhang-19,wei-zhang-22,wei-zhang-jems, zhang1,zhang2}. Notably, groundbreaking work initiated in \cite{chen-lin-wang,lin-Lwang} has recently advanced the understanding of the flat two-torus with singular sources, see \cite{Zchen-lin-1,lin22} and related references.

	
	By using the classical Green function $G_M(z,p)$, which satisfies
	\begin{equation*}
		\begin{cases}
			-\Delta_g G_M(x,p)=\delta_p-1\quad {\rm in}\ \; M
			\\
			\\
			\int_{M}G_M(x,p){\rm d}\mu=0,
		\end{cases}
	\end{equation*}
	we can rewrite \eqref{m-equ} as follows,
	\begin{equation}\label{r-equ}
		\Delta_g  {\rm w}+\rho\bigg(\frac{He^{\rm w}}{\int_M H e^{\rm w}{\rm d}\mu}-1\bigg)=0 \quad {\rm in}\ \; M,
	\end{equation}
	where
	\begin{equation}\label{r-sol}
		{\rm w}(x)=\nu(x)+4\pi \sum_{j=1}^N \alpha_j G_M(x,q_j)
		\quad\mbox{and}\quad
		H(x)=h(x)\prod_{j=1}^N e^{-4\pi\alpha_j G_M(x,q_j)}.
	\end{equation}
	Using the local isothermal coordinates $z=z(x)$ centered at $p,~0=z(p)$, $G(z,0)=G_M(x(z),p)$ can be decomposed by
	$$G(z,0)=-\frac 1{2\pi }\log |z|+R(z,0),$$
	where $R(z,0)$ denotes the regular part of $G(z,0)$, and $R_M(x(z),p)=R(z,0)$ stands for the regular part of $G_M(x(z),p)$.  As a consequence, near each singular point $q_j,~j=1,\cdots,N$, we can write
	\begin{equation}\label{H2}
		H(z)=h_j(z)|z|^{2\alpha_j},\quad |z|\ll 1,\quad 1\leq j\leq N,\quad \mbox{where}~h_j(z)>0~\mbox{near}~0.
	\end{equation}
	
	The equation \eqref{r-equ} possesses a variational structure, as it can be derived as the Euler-Lagrange equation of the functional:
	$$I_{\rho}({\rm w})=\frac 12 \int_M |\nabla {\rm w}|^2+\rho\int_M {\rm w}-\rho \log \int_M He^{\rm w},\quad {\rm w}\in H^1(M).$$
	Since adding a constant to any solution of (\ref{r-equ}) still yields a solution, we can, without loss of generality, define $I_\rho$ on the  following subspace $$\mathring{H}^1(M):=\left\{f\in H^1(M)|\int_Mfd\mu=0\right\}.$$ A detailed analysis of the variational structure of (\ref{r-equ}) can be found in \cite{mal-ruiz}.  Typically, we say that ${\rm w}$ in \eqref{r-equ} (or $\nu$ in \eqref{m-equ}) is {\it non-degenerate} if the linearized equation corresponding to the solution ${\rm w}$ of \eqref{r-equ} admits only the trivial solution. Specifically, the following equation 
	\begin{equation}\label{r-lin-equ} 
		\Delta_g {\phi} + \rho \frac{H e^{\rm w}}{\int_M H e^{\rm w} {\rm d}\mu} \left( \phi - \int_M \frac{H e^{\rm w} \phi}{\int_M H e^{\rm w}{\rm d}\mu} \right) = 0 \quad {\rm in}\ M 
	\end{equation} 
	has only the solution $\phi \equiv 0$.
	
	Similarly, for the Dirichlet problem \eqref{equ-flat}, we recall the standard Green's function 
	\begin{equation*}
		\begin{cases}
			-\Delta G_{\Omega}(x,p)=\delta_{p}  &{\rm in} \;\ \Omega,
			\\\\
			G_{\Omega}(x,p)=0  &{\rm on} \;\ \partial\Omega.
		\end{cases}
	\end{equation*}
	Using $G_{\Omega}(x,p)$, equation \eqref{equ-flat} can be rewritten as  follows
	\begin{equation}\label{r-equ-flat}
		\begin{cases}
			\Delta  {\rm w}+\rho \dfrac{He^ {\rm w}}{\int_{\Omega} H e^ {\rm w}{\rm d}x}=0  &{\rm in} \;\ \Omega,
			\\\\
			{\rm w}=0  &{\rm on} \;\ \partial\Omega,
		\end{cases}
	\end{equation}
	with
	\begin{equation*}
		{\rm w}(x)=\nu(x)+4\pi \sum_{j=1}^N \alpha_j G_{\Omega}(x,q_j)\quad\mbox{and}\quad
		H(x)=h(x)\prod_{j=1}^N e^{-4\pi\alpha_j G_{\Omega}(x,q_j)}.
	\end{equation*}
	It is well known that $G_{\Omega}(x,p)$ can be decomposed by
	$$G_{\Omega}(x,p)=-\frac{1}{2\pi}\log |x-p|+R_{\Omega}(x,p),$$ 
	where $R_{\Omega}(x,p)$ denotes the regular part. Then near each singular point $q_j$, we have
	\begin{equation*}
		H(z)=h_j(x)|x-q_j|^{2\alpha_j},\quad x\neq q_j
	\end{equation*}
	with $h_j(x)>0$ in a small neighbourhood of $q_j$. The linearized equation associated with a solution ${\rm w}$ of \eqref{r-equ-flat} takes the form
	\begin{equation}\label{r-lin-equ-flat}
		\Delta {\phi}+\rho\dfrac{He^{\rm w}}{\int_{\Omega} H e^{\rm w}{\rm d}x}
		\bigg(\phi-\int_{\Omega}\frac{He^{\rm w}\phi}{\int_{\Omega} H e^{\rm w}{\rm d}x}\bigg)=0 \quad {\rm in}\ \; {\Omega}.
	\end{equation}
	As the equation \eqref{m-equ} defined on the Riemann surface, we say that a solution $\nu$ of \eqref{equ-flat} is {\it non degenerate} if \eqref{r-lin-equ-flat} admits only the trivial solution.

	In two-dimensional space, the loss of compactness phenomenon occurs for both equations \eqref{m-equ} (or \eqref{r-equ}) and \eqref{equ-flat} (or \eqref{r-equ-flat}), primarily due to the exponential nonlinearity is critical in the sense of Sobolev embedding. Consequently, studying the qualitative behavior of blow-up solutions, such as their uniqueness and non-degenerate properties, becomes a natural and important question. For example, consider equation \eqref{m-equ}, we say that $\nu_k$ (or ${\rm w}_k$) is a sequence of bubbling solutions of \eqref{m-equ} (or \eqref{r-equ}) if the $L^{\infty}$ norm of the corresponding ${\rm w}_k$, defined by \eqref{r-sol}, tends to infinity as $k$ approaches infinity. Through a series of seminal works (see \cite{BM3, BT, li-cmp}), it is well-known that once blow-up phenomenon occurs, the function ${\rm w}_k$ blows up at a finite number of points. To say that the set of blow-up points is ${p_1, \cdots, p_m}$, means that there exist $m$ sequences of points $p_{k,1}, \cdots, p_{k,m}$ such that, possibly along a subsequence, ${\rm w}_k(p_{k,j}) \to +\infty$ and $p_{k,j} \to p_j$ as $k \to +\infty$ for $j = 1, \cdots, m$. If none of the blow-up points are singular sources, the local uniqueness of the bubbling solutions has been established in \cite{bart-4}. Here, the local uniqueness means that once the blow-up points of the bubbling solution are determined, the solution is unique. While the non-degeneracy property was proved in \cite{bart-5}. A natural question arises: are the results on uniqueness and non-degeneracy still true when the set of blow-up points contains singular points? In a recent paper, we have proved uniqueness in the case where at least one blow-up point is a singular source \cite{byz-1}; see also \cite{bart-4-2, wu-zhang-ccm} for related results. By assumption, the strength of the singular source at a point $p$ is $4\pi \alpha_p$, where $\alpha_p = 0$ if $p$ is a regular point. We use $\alpha_1, \cdots, \alpha_m$ to denote the strengths at the points $p_1, \cdots, p_m$, respectively, and assume that the first $m_1 \geq 1$ of these points are singular blow-up points:
	\begin{equation}\label{largest-s}
		\alpha_i>-1,\quad \alpha_i\not \in \mathbb N, \quad 1\le i\le m_1,\quad \alpha_{m_1+1}=\cdots=\alpha_m=0.
	\end{equation}
	Here, $\mathbb{N}$ represents the set of natural numbers including $0$, and $\alpha_M$ denotes the maximum of ${\alpha_1, \cdots, \alpha_m}$. We say that the singular source at $p$ is non-quantized if $\alpha_p$ is not a positive integer. Throughout this work, we assume all singular sources, when they occur at blow-up points, are non-quantized. In the first part of this work, we shall prove that the non-degeneracy result holds true in this more general setting.
	
	In order to state our results, we need the following definitions. First of all, let us set: 
	\begin{equation}\label{G_j*}
		G_j^*(x)=8\pi (1+\alpha_j)R_M(x,p_j)+8\pi \sum_{l,l\neq j}(1+\alpha_l)G_M(x,p_l),
	\end{equation}
	where $R_M(\cdot, p_j)$ is the regular part of $G_M(\cdot, p_j)$,
	\begin{equation}\label{L-p}
		L(\mathbf{p})=\sum_{j\in I_1}[\Delta \log h(p_j)+\rho_*-N^*-2K(p_j)]h(p_j)^{\frac{1}{1+\alpha_M}}e^{\frac{G_j^*(p_j)}{1+\alpha_M}},
	\end{equation}
	and
	\begin{equation}
		\label{a-note}
		\begin{cases}
			\alpha_M=\max_{i}\alpha_i,\quad
			I_1=\{i\in \{1,\cdots,m\};\quad \alpha_i=\alpha_M\}, \\ 
			\\
			\rho_*=8\pi\sum_{j=1}^m(1+\alpha_j),\quad  N^*=4\pi\sum_{j=1}^m\alpha_j.
		\end{cases}
	\end{equation}
	As mentioned above, some blow-up points being non-quantized singular sources while others being regular points, we can assume, without loss of generality, that $p_1, \cdots, p_{m_1}$ are singular sources and $p_{m_1+1}, \cdots, p_m$ are regular points. For $(x_{m_1+1}, \cdots, x_m) \in M \times \cdots \times M$, we define: 
	\begin{equation*}
		\begin{aligned}
			f^*(x_{m_1+1},\cdots,x_m)
			=&\sum_{j=m_1+1}^m\big[\log h(x_j)+4\pi R_M(x_j,x_j)\big]+4\pi \sum_{l\neq j}^{m_1+1,\cdots,m}G_M(x_l,x_j)\\
			&+\sum_{j=m_1+1}^m\sum_{i=1}^{m_1}8\pi(1+\alpha_i)G_M(x_j,p_i).
		\end{aligned}
	\end{equation*}
	We remark that for the set of regular blow-up points $(p_{m+1},\cdots,p_m)$, Chen-Lin in \cite{chen-lin-sharp} proved that it is a critical point of $f^*.$ In this article, we will examine all possible combinations of blow-up points. Dividing them into two categories:
	$$\mbox{\bf Class One:}\quad \alpha_M>0,\quad \mbox{\bf Class Two:}\quad \alpha_M\le 0.$$
	In \emph{Class One}, the set of blow-up points includes at least one positive singular source. While in \emph{Class Two}, the blow-up set either consists of regular blow-up points and negative singular sources ($\alpha_M = 0$) or contains only negative singular sources ($\alpha_M < 0$). Our first result addresses \emph{Class One}:

	\begin{thm}\label{main-theorem-2}
		Let $\nu_k$ be a sequence of bubbling solutions of {\rm (\ref{m-equ})} and assume that the blow-up set $\{p_1,\cdots,p_m\}$ satisfies
		$\{p_1,\cdots,p_m\}\cap \{q_1,\cdots,q_N\}\neq \emptyset$. Suppose $(\alpha_1,\cdots,\alpha_N)$ satisfies {\rm (\ref{largest-s})},
		$\alpha_M>0$, $L(\mathbf{p})\neq 0$ and, as far as $m_1<m$,  $$\det \big(D^2f^*(p_{m_1+1},\cdots,p_m)\big)\neq 0.$$ Then there exists $n_0>1$ such that
		$\nu_k$ is non degenerate for all $k\ge n_0$.
	\end{thm}
	Here $D^2f^*$ denotes the Hessian tensor field on $M$. At each regular blow-up point $p_j$ for $j=m_1+1,\cdots,m$, using conformal normal coordinates where $\chi_j$ is the conformal factor and satisfies the conditions
	$$\Delta_g=e^{-\chi_j}\Delta,\quad \chi_j(0)=|\nabla \chi_j(0)|=0,\quad e^{-\chi_j}\Delta \chi_j=-2K,$$
	we can compute $\det(D^2f^*)$ at these points. Here, $h_j$ is interpreted as $h_j e^{\chi_j}$ in a small neighborhood of each $p_j$ ($j=m_1+1,\cdots,m$). If $m_1=m$, meaning all blow-up points are singular sources, then as long as $\alpha_M>0$ we see that the unique relevant assumption is $L(\mathbf{p})\neq 0$.
	
	To present the result concerning \emph{Class two}, we begin by noting that the set of blow-up points consists only of regular points and negative sources. We introduce some new quantities. Let $B(q,r)$ represent the geodesic ball of radius $r$ centered at $q \in M$, and let $\Omega(q,r)$ denote the pre-image of the Euclidean ball of radius $r$, $B_r(q) \subset \mathbb{R}^2$, in a suitably defined isothermal coordinate system. If $m\geq 2$, we choose a collection of open, mutually disjoint sets $M_j$, with the closure of their union covering the entire manifold $M$.
	
	In this case let us remark that if $q_j$ is a regular blow-up point, then it is  automatically a critical point of 
	\[G_j^*(x)-G_j^*(q_j)+\log h_j(x)-\log h_j(0)\]
	and we make this assumption for  any blow-up point (remark that if $q_s$ is singular blow-up point, since $\alpha_M\le 0$, then the corresponding $\alpha_s\in (-1,0)$). So for all blowup points we assume that,
	\begin{equation}\label{neg-crit}
		\nabla \bigg ( G_j^*(x)
		+\log h_j(x)\bigg )|_{x=q_j}=0.
	\end{equation}

	Then we define, 
	\begin{equation}\label{Dp}
		D(\mathbf{p})=\lim_{r\to 0}\sum_{j=1}^mh_j(0)e^{G_j^*(q_j)}\bigg (
		\int_{M_j\setminus \Omega(q_j,r_j)}
		e^{\Phi_j(x,\mathbf{q})}d\mu(x)-\frac{\pi}{1+\alpha_j}r_j^{-2-2\alpha_j}\bigg ),
	\end{equation}
	where $M_1=M$ if $m=1$, $h_j$ is defined in (\ref{H2}), 
	\[r_j=\begin{cases}
		r\left(8h_j(0)e^{G_j^*(q_j)}\right)^{1/2},\quad  &m_1<j\le m, \\
		\\
		r,\quad &1\le j\le m_1,
	\end{cases}\]
	and
	\begin{align}\label{Phi_j}\Phi_j(x,\mathbf{q})=&\sum_{l=1}^m8\pi(1+\alpha_l)G_M(x,q_l)-G_j^*(q_j)\\ 
		&+\log h_j(x)-\log h_j(0)+4\pi\alpha_j(R_M (x,q_j)-G_M(x,q_j)). \nonumber
	\end{align}	
	\begin{rem}
		In local isothermal coordinates centered at $q_j$, $0=z(q_j)$, we have $4\pi\alpha_j(R(z,0)-G(z,0))=2\alpha_j\log|z|$ and, as far as $\alpha_j\neq 0$, we have $\alpha_j\in(-1,0)$. Although the expansion term is not absolutely integrable, we can handle the leading-order term after integration by using the term $1/r_j^{2+2\alpha_j}$. The remaining non-integrable term in the definition of $D(\mathbf{p})$ cancels out due to \eqref{neg-crit}, ensuring that the limit in \eqref{Dp} is well-defined. 
	\end{rem}
	
	Our second main result concerns \emph{Class Two}, i.e. $\alpha_M\le 0$.

	\begin{thm}\label{mainly-case-2}
		Under the same hypothesis of Theorem \ref{main-theorem-2} but assuming  $\alpha_M\le 0$ and (\ref{neg-crit}), then in any one of the following situations:
		\begin{enumerate}
			\item $\alpha_M=0$,\, $L(\mathbf{p})\neq 0$,\, $\det \big(D^2f^*(p_{m_1+1},\cdots,p_m)\big)\neq 0,$
			\item $\alpha_M=0$, $L(\mathbf{p})=0$,  $D(\mathbf{p})\neq 0$, $\det \big(D^2f^*(p_{m_1+1},\cdots,p_m)\big)\neq 0,$
			\item $\alpha_M<0$, $D(\mathbf{p})\neq 0$,
		\end{enumerate} the same non degeneracy property as in Theorem \ref{main-theorem-2} holds true.
	\end{thm}
	The proof of the above two theorems is highly dependent on the refined estimates recently obtained in \cite{byz-1}, which notably allow us for \emph{Class One} problems to avoid the requirement that singular blow-up points must be critical points of specific Kirchoff-Routh type functionals. 
	In the second part of this paper we apply the results about \emph{Class Two} problems to the Onsager mean field vortex model with singular sources. Since long-lived vortex structures are generally expected to concentrate at the critical points of Kirchoff-Routh type functionals (\cite{caglioti-2},\cite{ma-wei}), we will focus on the physically more interesting situation where a negative source (co-rotating vortex) satisfies this property.
	However we need the counterpart conclusions of the above results (in particular Theorem  \ref{mainly-case-2}) for the Dirichlet problem \eqref{r-equ-flat} (or \eqref{equ-flat}). To present these results, we first introduce some local and global quantities that are appropriate for describing the combinations of blow-up points. We define the following:
	\begin{equation*}
		G_{j,\Omega}^*(x)=8\pi (1+\alpha_j)R_{\Omega}(x,p_j)+8\pi \sum_{l\neq j}^{1,\cdots,m}(1+\alpha_l)G_{\Omega}(x,p_l),
	\end{equation*}
	and, similar to notations for the first part, assume (\ref{largest-s}) for $(\alpha_1,\cdots,\alpha_m)$, where we keep the same conventions
	about $I_1$ and $\alpha_M$. Next let us define,
	$$L_{\Omega}(\mathbf{p})=\sum_{j\in I_1}\Delta \log h(p_j)h(p_j)^{\frac{1}{1+\alpha_M}}e^{\frac{G_{j,\Omega}^*(p_j)}{1+\alpha_M}},$$
	\begin{align*}
		f_{\Omega}^*(x_{m_1+1},\cdots,x_m)&=\sum_{j=m_1+1}^m\big[\log h(x_j)+4\pi R(x_j,x_j)\big]+4\pi \sum_{l\neq j}^{m_1 +1,\cdots,m}G(x_l,x_j),\\
		&\quad+\sum_{j=m_1+1}^m\sum_{i=1}^{m_1}8\pi(1+\alpha_i)G(x_j,p_i),
	\end{align*}
	and let $D^2f_{\Omega}^*$ be the Hessian on $\Omega$. Of course, in this case $(p_{m_1+1},\cdots,p_m)$ is a critical point of $f_{\Omega}^*$ (\cite{ma-wei}).
	Concerning {\em Class one} we have,
	
	\begin{thm}\label{main-theorem-4}
		Let $\nu_k$ be a sequence of bubbling solutions of {\rm (\ref{equ-flat})} and assume that the blow-up set $\{p_1,\cdots,p_m\}$ satisfies
		$\{p_1,\cdots,p_m\}\cap \{q_1,\cdots,q_N\}\neq \emptyset$. Suppose $(\alpha_1,\cdots,\alpha_N)$ satisfies {\rm (\ref{largest-s})},
		$\alpha_M>0$, $L_{\Omega}(\mathbf{p})\neq 0$ and, as far as $m_1<m$,  $\det \big(D^2f^*(p_{m_1+1},\cdots,p_m)\big)\neq 0$. Then there exists $n_0>1$ such that
		$\nu_k$ is non degenerate for all $k\ge n_0$.
	\end{thm}
	
	Consider the case of {\em Class two} we set
	\begin{align*}\Phi_{j,\Omega}(\mathbf{q}):=~&\sum_{l=1}^m 8\pi(1+\alpha_l) G_{\Omega}(x,q_l)-G_{j,\Omega}^*(q_j)+\log h_j(x)-\log h_j(q_j)\\
		&+2\alpha_j\log |x-q_j|,
	\end{align*}
and
$$
D_{\Omega}(\mathbf{p})
=\lim_{r\to 0}\sum_{j=1}^m h(p_j)e^{G_{j,\Omega}^*(p_j)}\bigg (\int_{\Omega_j\setminus B_{r_j}(p_j)}e^{\Phi_{j,\Omega}(\mathbf{q})}dx-\frac{\pi}{1+\alpha_j}r_j^{-2\alpha_j-2}\bigg ),
$$
where $\displaystyle{r_j=r\left( 8 h(p_j)e^{G_{j,\Omega}^*(p_j)}\right)^{1/2}}$ if $j>m_1$ and $r_j=r$ for $1\le j\le m_1$; $\Omega_1=\Omega$ if $m=1$, otherwise we have $\Omega_l\cap \Omega_s=\emptyset$ for $l\neq s$ and $\cup_{j=m_1+1}^m\overline\Omega_j=\overline\Omega$. The counterpart result of Theorem \ref{main-theorem-4} in the case of Class two reads as follows

Similar to the previous case we require 
\begin{equation}\label{neg-crit-2}
	\nabla \bigg (G_{j,\Omega}^*(x)+\log h_j(x)\bigg)\bigg |_{x=q_j}=0, \quad 1\le j\le m.
\end{equation}
Then we have,
\begin{thm}\label{main-theorem-1}
	Under the same hypothesis of Theorem \ref{main-theorem-4} but assuming $\alpha_M\le 0$ and (\ref{neg-crit-2}), then for any one of the following situations:
	\begin{enumerate}
		\item $\alpha_M=0$, $L_{\Omega}(\mathbf{p})\neq 0$, $\det \big(D^2f_{\Omega}^*\big )(p_{m_1+1},\cdots,p_m)\neq 0$,
		\item $\alpha_M=0$, $L_{\Omega}(\mathbf{p})=0$, $D_{\Omega}(\mathbf{p})\neq 0$, $\det \big(D^2f_{\Omega}^*\big )(p_{m_1+1},\cdots,p_m)\neq 0$,
		\item $\alpha_M<0$, $D_{\Omega}(\mathbf{p})\neq 0$,
	\end{enumerate} the same non degeneracy property as in Theorem \ref{main-theorem-4} holds true.
\end{thm}
\medskip

\medskip

In the second part of this paper, we apply our non-degeneracy results to the analysis of mean field theory within the Onsager (\cite{On} and \cite{caglioti-1},\cite{caglioti-2})  statistical mechanics description of the vortex model with singular sources. We consider the case where the mean field vorticity $\omega$ interacts with one fixed co-rotating vortex, whose total vorticity (proportional to $|\beta|$) has the same sign as that of $\omega$, alongside $N$ counter-rotating vortices. These singularities act as a fixed external potential in the model (see Remark \ref{rem.1} for further details). Compared to earlier findings in \cite{caglioti-2}, \cite{chang-chen-lin}, \cite{BLin2} and \cite{BLin3}, a key nuance is that the critical threshold for the existence of solutions in the canonical mean field model is no longer $8\pi$, but rather $8\pi(1+\beta)$. Physically, this suggests that entropy maximizers (i.e., thermodynamic equilibrium states satisfying the microcanonical variational principle, as discussed in section \ref{sec5.2}) are expected to concentrate at the vortex centered at the $\beta$-sink in the large energy limit. Consequently, we need to extend first the results from \cite{caglioti-2} and \cite{chang-chen-lin} regarding the existence of negative temperature thermodynamic equilibrium states and the equivalence of statistical ensembles for a specific class of domains (domains of the first kind) see Definition \ref{def:first} and sections \ref{sec:meansink} and 
\ref{section:firstsec}.
This extension relies on adaptations of known (\cite{wolan2}) dual variational methods, a sharp singular version (\cite{ads}) of the Moser-Trudinger inequality \cite{moser}, recent findings in \cite{bj-lin} and \cite{Bons}, as well as arguments from \cite{caglioti-1} and \cite{caglioti-2}, which 
is why we will be rather sketchy about some of these proof.\\
However, as in \cite{caglioti-2} and \cite{bart-5}, the situation for domains of second kind (again see Definition \ref{def:first}) is much more subtle as it requires, among other things, the understanding of the existence/non existence of solutions of the mean field equation (which is \eqref{r-equ-flat-lm} below) exactly at the critical (with respect to the singular Moser-Trudinger inequality \cite{ads}) threshold $8\pi(1+\beta)$ and a full description of the set of solutions in a small enough supercritical region behind $8\pi(1+\beta)$, where uniqueness of solutions fails. It is worth to remark that a full description of the thermodynamic equilibrium states is very involved and still open in general for domains of second kind also in the regular (without $\beta$-sources) case,  see \cite{Bons} for partial results concerning this point and the very interesting counterexamples recently found in \cite{BCN}. Here  we are able to generalize various results in \cite{caglioti-2}, \cite{chang-chen-lin} and \cite{bart-5}, \cite{BLin2}, including the existence and asymptotic behavior of entropy maximizers for large energies (see Lemma \ref{lem:entropyasymp}) and  the characterization of domains of second kind in terms of $D_\Omega({\bf p})$ (see Theorem \ref{thm:5.2}) which is in turn crucial to the existence of a full unbounded interval of strict convexity of the Entropy (negative specific heat states, see Theorem \ref{convex:entropy}). Few  rigorous proof of the existence of negative specific heat states are at hand for these models (\cite{caglioti-2},\cite{bart-5}) which in our case require the description of the monotonicity of the energy as a function of the inverse temperature in a  supercritical (behind $8\pi(1+\beta)$) interval. This is not at all trivial and we succeed here
by using the local uniqueness result in \cite{byz-1}, Theorem \ref{main-theorem-1} above and the local uniqueness and nondegeneracy results in \cite{wyang} about the naturally associated Gel'fand-type problem. These results are neither enough on their own which is why we purse a generalization (for $\beta$-sources) of independent interest of former estimates obtained in the regular case in \cite{chen-lin-sharp}, see Theorem \ref{negative-b}, Proposition \ref{est-muk} and \eqref{im-ck} below. 

The proof of Theorem \ref{thm:5.2} requires careful consideration, especially in cases where $D_{\Omega}({\bf p})$ might vanish. Theorem \ref{thm:5.2} itself is of independent interest as it generalizes results previously established for semilinear elliptic equations with critical nonlinearity in three dimensions (\cite{druet}) and for mean field type equations when $\beta=0$ (\cite{chang-chen-lin}, \cite{BLin2}). Additionally, we derive an intriguing sufficient condition for the non-existence of solutions at the critical parameter, which extends a similar result found in \cite{BLin2}, as detailed in Corollary \ref{last:thm.1}.


At last, interestingly enough, as an application of these refined analysis, we come up with the exact counting of the number of solutions for the mean field equation \eqref{r-equ-flat-lm} in a small supercritical region behind $8\pi(1+\beta)$, see Theorem \ref{thm:count1}. 
\medskip


Throughout the article, $B_\tau = B_\tau(0)$ will always represent a ball centered at the origin in some local isothermal coordinates $y \in B_\tau$. Whenever $B_\tau$ refers to such a ball centered at $0 = y(p),~p \in M$, we will denote by $\Omega(p,\tau) \subset M$ the pre-image of $B_\tau$. On the other hand, $B(p,\tau) \subset M$ will always refer to a geodesic ball. Additionally, many estimates will involve a small positive number $\epsilon_0 > 0$ and generic constant $C$, which may vary from line to line. Note that while the letter $\rho$ is commonly used in physics to denote density, we will use $\omega$ to represent the density starting from section \ref{sec:meansink}, where we discuss the statistical mechanics of the Onsager model with singular sources. This choice is made because $\rho$ is already used as a parameter in equations \eqref{r-equ} and \eqref{r-equ-flat}.
\medskip

This paper is organized as follows: Section 2 presents preliminary results essential for proving non-degeneracy, including a precise estimate for the difference between the parameter $\rho_k$ and the critical value $\rho_*$, as well as a uniqueness lemma. Sections 3 and 4 revisit the asymptotic analysis from \cite{byz-1} and provide a proof of the non-degeneracy result. The final two sections apply this non-degeneracy result to demonstrate the strict convexity of the entropy in the large energy limit of the Onsager mean field vortex model with sinks.


\medskip

\noindent{\bf Acknowledgement} L. Zhang acknowledges support from Simon's foundation grant 584918. W. Yang acknowledges support from National
Key R\&D Program of China 2022YFA1006800, NSFC No.
12171456 and No. 12271369, FDCT No. 0070/2024/RIA1, University of Macau Development Foundation
No. TISF/2025/006/FST, No. MYRG-GRG2024-00082-FST and Startup
Research Grant No. SRG2023-00067-FST.
D. Bartolucci's Research is partially supported by the MIUR Excellence Department Project MatMod@TOV awarded to the
Department of Mathematics, University of Rome Tor Vergata and by PRIN project 2022, ERC PE1\_11,
"{\em Variational and Analytical aspects of Geometric PDEs}". D. Bartolucci is member of the INDAM Research Group  "Gruppo Nazionale per l'Analisi Matematica,
la Probabilit\`a e le loro Applicazioni".
\medskip

\section{Preliminary Estimates}\label{preliminary}

Since the proof of the main theorems requires delicate analysis,
in this section we list some estimates established in \cite{BCLT,BM3,BT,BT-2,chen-lin-sharp,chen-lin,li-cmp,zhang1,zhang2}.

Let ${\rm w}_k$ be a sequence of solutions of (\ref{r-equ}) with $\rho =\rho_k$
and assume that ${\rm w}_k$ blows up at $m$ points $\{p_1, \cdots,p_m\}$. Since \eqref{r-equ} is invariant after adding a constant to the solution. Without loss of generality, we may assume that 
\begin{equation}\label{norm}
	\int_{M}He^{{\rm w}_k}{\rm d}\mu=1.
\end{equation}
Then we can rewrite the equation for ${\rm w}_k$ as,
\begin{equation}\label{equ-uk}
	\Delta_g {\rm w}_k+\rho_k(He^{{\rm w}_k}-1)=0\quad {\rm in} \ \; M.
\end{equation}
From well known results about Liouville-type equations (\cite{BT-2,li-cmp}),
\begin{equation*}
	{\rm w}_k-\overline{{\rm w}_k} \ \to \sum_{j=1}^m 8\pi(1+\alpha_j)G(x,p_j) \quad {\rm in} \quad {\rm C}_{\rm loc}^2(M\backslash \{p_1,\cdots,p_m\}),
\end{equation*}
where $\overline{{\rm w}_k}$ is the average of ${\rm w}_k$ on $M$:
$\overline{{\rm w}_k}=\int_{M}{\rm w}_k{\rm d}\mu$. For later convenience we fix $r_0>0$ small enough and $M_j\subset M, 1\leq j\leq m$ such that
\begin{equation*}
	M=\bigcup_{j=1}^m \overline{M}_j;\quad M_j\cap M_l=\varnothing,\  {\rm if}\ j\neq l;\quad B(p_j,3r_0)\subset M_j, \quad j=1,\cdots,m.
\end{equation*}
According to this definition $M_1=M$, if $m=1$.

The notation about local maximum is of particular relevance in this context. If $p_j$ is a regular blow-up point (i.e. $\alpha_j=0$) we define $p_{k,j}$ and $\lambda_{k,j}$  as follows,
\begin{equation*}
	\lambda_{k,j}=u_k(p_{k,j}) \mathrel{\mathop:}=\max_{B(p_j,r_0)}u_k,
\end{equation*}
while if  $p_j$ is a singular blow-up point (i.e. $\alpha_j\neq 0$), then we define,
\begin{equation*}
	p_{k,j}:=p_j\quad \mbox{ and }\quad \lambda_{k,j}:=u_k(p_{k,j}).
\end{equation*}
Next, let us define the so called "standard bubble" $U_{k,j}$ to be the solution of
\begin{equation}\nonumber
	\Delta U_{k,j}+\rho_kh_j(p_{k,j})|x-p_{k,j}|^{2\alpha_j}e^{U_{k,j}}=0 \quad  {\rm in} \ \; \mathbb{R}^2
\end{equation}
which takes the form (\cite{CL1,CL2,Parjapat-Tarantello}),
\begin{equation}\nonumber
	U_{k,j}(x)=\lambda_{k,j}-2\log\Big(1+\frac{\rho_k h_j(p_{k,j})}{8(1+\alpha_j)^2}e^{\lambda_{k,j}}|x-p_{k,j}|^{2(1+\alpha_j)}\Big).
\end{equation}

It is well-known (\cite{BCLT,BT-2,li-cmp}) that $u_k$ can be approximated by the standard bubbles $U_{k,j}$ near $p_j$ up to a uniformly bounded error term:
\begin{equation*}
	\big|u_k(x)-U_{k,j}(x)\big| \leq C, \quad x\in B_r(p_j,r_0).
\end{equation*}	
As a consequence, in particular we have,
\begin{equation}\nonumber
	|\lambda_{k,i}-\lambda_{k,j}|\leq C, \quad 1\leq i,j \leq m,
\end{equation}
for some $C$ independent of $k$.

\medskip
In case $m_1<m$, it has been shown in \cite{chen-lin-sharp} that,
\begin{equation}\label{first-deriv-est}
	\nabla(\log h+G_j^*)(p_{k,j})=O(\lambda_{k,j}e^{-\lambda_{k,j}}),\quad m_1+1\leq j\leq m,
\end{equation}
which, in view of the non-degeneracy condition
$$\det\big(D^2f^*(p_{m_1+1},\cdots,p_m)\big)\neq0,$$
readily implies that,
\begin{equation}\label{p_kj-location}
	|p_{k,j}-p_j|=O(\lambda_{k,j}e^{-\lambda_{k,j}}),\quad m_1+1\leq j\leq m.
\end{equation}
Later, sharper estimates were obtained in \cite{chen-lin,zhang2} for $1\leq j \leq m_1$ and
in \cite{chen-lin-sharp,gluck,zhang1} for $m_1+1\leq j\leq m$.

By using $\lambda_i^k$ to denote the maximum of $u_i^k$, $i=1,2$ which share the same value of the parameter $\rho_k$, then it is a simple consequence of $L(\mathbf p)\neq 0$ (see \cite{bart-4,wu-zhang-ccm}) that
\begin{equation}\label{initial-small}
	|\lambda_1^k-\lambda_2^k|\leq Ce^{-\epsilon_0 \lambda_1^k}\quad \mbox{ \rm for some }\epsilon_0>0.
\end{equation}

Let us also recall that it has been established in \cite{BM3,BT} that (see \eqref{a-note})  $\rho_*=\lim\limits_{k\to +\infty}\rho_k$.
Concerning the difference between $\rho_k$ and $\rho_*$,
we set
$$\rho_{k,j}=\rho_k\int_{\Omega(p_{j},\tau)}He^{{\rm w}_{k}}.$$ The following estimates hold (see  \cite{chen-lin-sharp}, \cite{chen-lin},\cite{byz-1}):
\begin{thm}\label{thm:2.1}
	There exists $\epsilon_0>0$ and $d_j>0$ such that,
	\begin{align*}
		\begin{cases}
			\rho_{k,j}-8\pi(1+\alpha_j)=2\pi d_je^{-\frac{\lambda_{k,j}}{1+\alpha_j}}+
			O(e^{-\frac{1+\epsilon_0}{1+\alpha_j}\lambda_{k,j}}),  &\alpha_j >0, \\
			\\
			\rho_{k,j}-8\pi(1+\alpha_j)=O(e^{-\lambda_{k,j}}) &\alpha_j <0, \\
			\\
			\rho_{k,j}-8\pi=O\big(\lambda_{k,j}e^{-\lambda_{k,j}}\big), &\tau+1\leq j\leq m.
		\end{cases}
	\end{align*} 
	
\end{thm}
\medskip

It has been proved in \cite{chen-lin} that
\begin{equation}\label{small-rho-k}
	\rho_k-\rho_*= L(\mathbf{p})e^{-\frac{\lambda_1^k}{1+\alpha_M}}+O(e^{-\frac{1+\epsilon_0}{1+\alpha_M}\lambda_{k,j}})
\end{equation}
for some $\epsilon_0>0$ and $L(\mathbf{p})$ as defined in (\ref{L-p}).

For later study, we need a general theorem describing $\rho_k-\rho_*$ when the set of blow-up points is a mixture of regular points and singular sources with negative strength. Here we let $p_1,\cdots,p_{m_1}$ be singular sources with negative strength: $\alpha_l\in (-1,0)$ ($l=1,\cdots,m_1$) and let $\alpha_l=0$ for $l=m_1+1,\cdots,m$. That is, $p_{m_1+1}^k,\cdots,p_m^k$ are regular blow-up points. If $m>m_1$, we recall that 
$$L(\mathbf{p})=\sum_{j=m_1+1}^m[\Delta \log h(p_j)+\rho_*-N^*-2K(p_j)]h(p_j)e^{G_j^*(p_j)},$$
where $$\rho_*=8\pi\sum_{j=1}^{m_1}(1+\alpha_j)+8\pi (m-m_1),\quad N^*=\sum_{j=1}^{m_1}4\pi \alpha_j.$$

Now we define two terms which involves global integration. 
Let 
\[D_R:=\lim_{r\to 0}\sum_{j=m_1+1}^mh_j(0)e^{G_j^*(q_j)}\bigg (\int_{M_j\setminus \Omega(q_j,r_k)}e^{\Phi_j(x,q)}d\mu-\pi r_j^{-2}\bigg ),\]
where 
\[r_j=r\left(8h_j(0)e^{G_j^*(q_j)}\right)^{\frac 12},\quad j=m_1+1,\cdots,m,\]
and
\[D_s:=\lim_{r\to 0}\sum_{j=1}^{m_1}h_j(0)e^{G_j^*(p_j)}\left(\int_{M_j\setminus \Omega (q_j,r)}e^{\Phi_j(x,q)}d\mu-\frac{\pi}{1+\alpha_j}r^{-2-2\alpha_j}\right).\]

\begin{thm}\label{negative-b}
	Suppose $\alpha_l\in (-1,0)$ for $l=1,\cdots,m_1$, $\alpha_l=0$ for $l=m_1+1,\cdots,m$ and that \eqref{neg-crit} holds true. Then we have, 
	\begin{align*}
		\rho_k-\rho_*=~&\frac{16\pi e^{-\lambda_m^k}}{\rho_*h_m^2(0)e^{G_m^*(p_m)}}L(\mathbf{p})[\lambda_m^k+\log (h_m^2(0)e^{G_m^*(p_m)}r^2\rho_*)-2]\\
		&+\frac{64 e^{-\lambda_m^k}}{\rho_*h_m^2(0)e^{G_m^*(p_m)}}(D_R+O(r^{\epsilon_0}))\\
		&+e^{-\lambda_1^k}\rho_*e^{-G_1^*(p_1)}\left(\frac{8(1+\alpha_1)^2}{\rho_* h_1(0)}\right)^2(D_s+O(r^{\epsilon_0}))
		+O(e^{-(1+\epsilon_0)\lambda_1^k}),
	\end{align*}
	where $\epsilon_0$ is a positive small number. If both regular and singular blow-up points exist, then $\lambda_1^k$ and $\lambda_m^k$ satisfy,
	\[-\lambda_m^k=-\lambda_1^k+2\log\frac{h_m(0)(1+\alpha_1)^2}{h_1(0)}+G_m^*(p_m)-G_1^*(p_1)+O(e^{-\epsilon_0\lambda_1^k}).\]
\end{thm}
\begin{rem}
	The sum of the last two terms in the expansion of $\rho_k-\rho_*$ can be written either as follows,
	\[e^{-\lambda_1^k}\rho_*e^{-G_1^*(p_1)}\left(\frac{8(1+\alpha_1)^2}{\rho_*h_1(0)}\right)^2(D(\mathbf{p})+O(r^{\epsilon_0})),\]
	whenever there exists at least one singular source or  as follows,
	\[\frac{64 e^{-\lambda_m^k}}{\rho_*h_m^2(0)e^{G_m^*(p_m)}}(D(\mathbf{p})+O(r^{\epsilon_0})),\]
	if at least one regular blow-up point exists.
\end{rem}

\begin{proof}[Proof of Theorem \ref{negative-b}]
	Here for convenience we define 
	\begin{equation}\label{big-phi}
		\Phi(x)=\sum_{j=1}^m 8\pi(1+\alpha_j)G(x,p_j),
	\end{equation}
	and recall the equation for ${\rm w}_k$ in (\ref{equ-uk}) and the normalization (\ref{norm}). We carry out the analysis around $p_l$. Let $f_k$ be defined in a neighborhood of $p_l$ by 
	\begin{equation}\label{eq-fk}
		\Delta_g f_k=\rho_k~\mbox{in} \,\, B(p_l,\tau),\quad f_k(p_l)=0,\quad f_k=\mbox{constant on } \partial B(p_l,\tau).
	\end{equation}
	Then we set
	\[u_k={\rm w}_k-f_k\]
	and have
	\[\Delta_gu_k+\rho_kHe^{f_k}e^{u_k}=0,\quad \mbox{in}\quad B(p_l,{\tau}).\]
	Using the conformal factor function $\chi$, which in local coordinates centered at $p_l$ satisfies, 
	\begin{equation}\label{eq-phi-k}
		\chi(0)=|\nabla \chi(0)|=0,\quad \Delta_g=e^{-\chi}\Delta,\quad \Delta \chi(0)=-2K(p_l),
	\end{equation}
	we have
	\begin{equation}\label{ufk-e}
		\Delta u_k+\rho_kHe^{\chi+f_k}e^{u_k}=0 \quad \mbox{in}\quad B_{\tau}.\footnote{Here we use the same notation for $u_k$ and $f_k$ in both $B(p_l,\tau)$ (the geodesic ball on the Riemann manifold) and $B_\tau(0)$ (in the local isothermal coordinates).}
	\end{equation}
	To avoid cumbersome notations, for the time being we use the same notation for $p_l$ in local and global coordinates. Recall the definition \eqref{H2} of $h_l$, then in local coordinates around $p_l$ we have, if $1\le l\le m_1$, 
	\begin{equation}\label{hl-small}
		h_l(x)=h(p_l+x)\exp \bigg (\sum_{j=1,j\neq l}^{m_1}(-4\pi\alpha_j)G(x,p_j)-4\pi\alpha_l R(x,p_l)\bigg ),
	\end{equation}
	and 
	\[
	h_l(0)=h(p_l)\exp \bigg (\sum_{j=1,j\neq l}^{m_1}(-4\pi\alpha_jG(p_l,p_j))-4\pi\alpha_l R(p_l,p_l)\bigg ),
	\]
	while if $l>m_1$ we have,
	\begin{equation}\label{hl-large}
		h_l(x)=h(p_l+x)\exp \bigg (\sum_{j=1}^{m_1}(-4\pi \alpha_j G(x,p_j)\bigg ),\quad m_1<l\le m,
	\end{equation}
	and 
	\[
	h_l(0)=h(p_l)\exp \bigg (\sum_{j=1}^{m_1}(-4\pi\alpha_jG(p_l,p_j))\bigg ),\quad m_1<l\le m.
	\]
	Here we also note that around $p_l$ for $l>m_1$, by (\ref{a-note}),
	\begin{equation}\label{In-Lp}
		\Delta (\log H+\chi+f_k)(0)=\Delta \log h(p_l)+\rho_k-N^*-2K(p_l). 
	\end{equation}
	Let $\psi_k$ be the harmonic function that eliminates the oscillation of $u_k$ on $\partial B_\tau$, that is $\Delta \psi_k=0$ in $B_\tau$ 
	and 
	\[
	\psi_k=u_k-\frac{1}{2\pi \tau}\int_{\partial B_{\tau}}u_kdS
	\quad \hbox {on } \partial B_{\tau}. 
	\]
	Then, in local coordinates centered at $p_l$, a rough estimate for $u_k$ in $B_\tau$ is,
	$$u_k(x)=U_k+\psi_k+O(\epsilon_{l,k}^{\delta}),\quad x\in B_{\tau}$$
	where 
	\[U_k(x)=\lambda_l^k-2\log\left(1+\frac{\rho_k h_l(0)}{8(1+\alpha_l)^2}e^{\lambda_l^k}|x|^{2+2\alpha_l}\right),\]
	$\epsilon_{l,k}=e^{-\frac{\lambda_l^k}{2(1+\alpha_l)}}$. In particular for $x\in B_{\tau}\setminus B_{\tau/2}$, we have
	\begin{equation}\label{uk-neg-d}
		\begin{aligned}
			u_k(x)=&-\lambda_l^k-2\log \frac{\rho_kh_l(0)}{8(1+\alpha_l)^2}-4(1+\alpha_l)\log |x|\\
			&+\psi_k(x)+O(\epsilon_{l,k}^{\delta}),  
		\end{aligned}
	\end{equation}
	for some $\delta>0$. On the other hand, the Green
	representation of ${\rm w}_k$ gives
	\begin{equation}\label{green-u-away}
		{\rm w}_k(x)=\overline{{\rm w}_k}+\Phi(x)+O(\epsilon_{l,k}^{\delta}),
	\end{equation}
	for $x$ away from singular sources, where $\Phi$ is defined in (\ref{big-phi}). By using the notation $G_l^*(x)$ in (\ref{G_j*}), this expression can be written as follows, 
	\[{\rm w}_k(x)=\overline{{\rm w}_k}-4(1+\alpha_l)\log |x-p_l|+G_l^*(x)+O(\epsilon_{l,k}^{\delta}).\]
	By the definition of $f_k$, in local coordinates around $p_l$ we have,
	\begin{equation}\label{uk-det}
		\begin{aligned}
			u_k(x)=~&\overline{{\rm w}_k}-4(1+\alpha_l)\log |x|+G_l^*(p_l)\\
			&+G_l^*(x)-G_l^*(p_l)-f_k+O(\epsilon_{l,k}^{\delta}).  
		\end{aligned}
	\end{equation}
	Comparing (\ref{uk-det}) and (\ref{uk-neg-d}), by the definition of $G_l^*$, and recalling that $\psi_k(0)=0$, we have that, 
	\[\psi_k(x)=G_l^*(x)-G_l^*(p_l)-f_k+O(\epsilon_{l,k}^{\delta})\]
	and 
	\begin{equation}\label{uk-bar-l}
		\overline{{\rm w}_k}=-\lambda_l^k-2\log \frac{\rho_k h_l(0)}{8(1+\alpha_l)^2}-G_l^*(p_l)+O(e^{-\epsilon_0\lambda_1^k}),\quad l=1,\cdots,m.
	\end{equation}
	In particular 
	\begin{equation}\label{uk-bar-1}
		\overline{{\rm w}_k}=-\lambda_1^k-2\log \frac{\rho_kh_1(0)}{8(1+\alpha_1)^2}-G_1^*(p_1)+O(e^{-\epsilon_0\lambda_1^k}).
	\end{equation}
	By using (\ref{uk-bar-1}) and (\ref{uk-bar-l}),  for $l=1,\cdots,m$, we have that
	\[
	-\lambda_l^k=-\lambda_1^k+2\log\frac{\rho_kh_l(0)}{8(1+\alpha_l)^2}-2\log \frac{\rho_kh_1(0)}{8(1+\alpha_1)^2}+G_l^*(p_l)-G_1^*(p_1)+O(e^{-\epsilon_0\lambda_1^k}),
	\]
	and consequently,
	\begin{equation}\label{height-d}
		e^{-\lambda_l^k}=e^{-\lambda_1^k}\frac{(1+\alpha_1)^4h_l^2(0)}{(1+\alpha_l)^4h_1^2(0)}e^{G_l^*(p_l)-G_1^*(p_1)}+O(e^{-(1+\epsilon_0)\lambda_1^k}),~ l=1,\cdots,m,
	\end{equation}
	for some $\epsilon_0>0$. Remark that the estimate (\ref{height-d}) holds for regular points as well. Next we evaluate 
	$\rho_k=\int_M\rho_kHe^{{\rm w}_k}d\mu$ as follows,
	\begin{align*}
		\rho_k=\sum_{l=1}^m\int_{B(p_l,\tau_l)}\rho_kHe^{{\rm w}_k}d\mu+\int_E\rho_kHe^{{\rm w}_k}d\mu=\sum_{l=1}^m\rho_{k,l}+\int_E\rho_kHe^{{\rm w}_k}d\mu,
	\end{align*}
	where $E=M\setminus (\bigcup_l B(p_l,\tau_l))$. First we evaluate $\rho_{k,j}$ for $j=1,\cdots,m_1$. For each $j$ let $\Omega_j$ be a neighborhood of $p_j$ such that the $\Omega_j$ are mutually disjoint and their union is $M$. In each $B(p_l,\tau_l)$, $l=1,\cdots,m_1$, we use Theorem 3.2 of \cite{byz-1} as follows, 
	\begin{equation}
		\label{inter-ne-1}
		\begin{aligned}
			\int_{B(p_l,\tau)}\rho_kHe^{{\rm w}_k}d\mu
			&=\int_{B_{\tau}}\rho_k|x|^{2\alpha_l}h_le^{f_k+\chi}e^{u_k}d\mu \\
			&=8\pi(1+\alpha_l)\left(1-\frac{8(1+\alpha_l)^2}{\rho_kh_l(0)}\tau^{-2-2\alpha_l}e^{-\lambda_l^k}\right)+l.o.t.
		\end{aligned}
	\end{equation}
	Remark that the lower order terms in the r.h.s. of \eqref{inter-ne-1} include factors of order $e^{-\lambda_l^k}$ which are $o(1)$ in $\tau$, as $\tau\to 0$. On the other side, by (\ref{green-u-away}) and (\ref{uk-bar-l}) we have that,
	\begin{align*}
		\int_{\Omega_l\setminus B(p_l,\tau)}\rho_kHe^{{\rm w}_k}d\mu
		&=\int_{\Omega_l\setminus B(p_l,\tau)}\rho_kHe^{\overline{{\rm w}_k}}e^{\Phi}d\mu+O(e^{-(1+\epsilon_0)\lambda_1^k})\\
		&= e^{-\lambda_l^k}\left(\frac{8(1+\alpha_l)^2}{\rho_kh_l(0)}\right)^2e^{-G_l^*(p_l)}\int_{\Omega_l\setminus B(p_l,\tau)}\rho_kHe^{\Phi}d\mu\\
		&\quad+O(e^{-(1+\epsilon_0)\lambda_1^k}).
	\end{align*}
	By using local coordinates around $p_l$ to evaluate the last integral, we see that, \begin{align*}\int_{\Omega_l\setminus B(p_l,\tau)}\rho_kHe^{\Phi}d\mu 
		=\int_{\Omega_l\setminus B(p_l,\delta)}\rho_kH e^{\Phi}d\mu+\int_{B(p_l,\delta)\setminus B(p_l,\tau)}\rho_kh_l|x|^{-4-2\alpha_l}e^{G_l^*(x)}dx\\
		=\int_{\Omega_l\setminus B(p_l,\delta)}\rho_kH e^{\Phi}d\mu+
		\rho_kh_l(0)e^{G_l^*(p_l)}\left(\frac{\pi\tau^{-2-2\alpha_l}}{1+\alpha_l}+O(\delta^{-2-2\alpha_l})\right),
	\end{align*}
	where, in view of \eqref{neg-crit}, the remaining non integrable term in the expansion vanishes, while the rest can be safely included in the $O(\delta^{-2-2\alpha_l})$.   Therefore we have, 
	\begin{equation*}
		\int_{\Omega_l}\rho_kHe^{{\rm w}_k}d\mu
		=8\pi(1+\alpha_l)+e^{-\lambda_l^k}\rho^*e^{-G_l^*(p_l)}\left(\frac{8(1+\alpha_l)^2}{\rho_*h_l(0)}\right)^2\left(D_l+O(\delta^{-2\alpha_l})\right)
	\end{equation*}
	where we used once more \eqref{neg-crit} and 
	$$D_l=\lim_{\tau\to 0+}\int_{\Omega_l \setminus B(p_l,\tau)} He^{\Phi}d\mu-\frac{h_l(0)e^{G_l^*(p_l)}\pi}{1+\alpha_l}\tau^{-2-2\alpha_l},\quad l=1,\cdots,m_1,$$	
	and we can replace $\lambda_l^k$ by $\lambda_1^k$ to deduce that,
	\begin{equation}\label{ball-l-1}
		\int_{\Omega_l}\rho_kHe^{{\rm w}_k}d\mu
		=8\pi(1+\alpha_l)+e^{-\lambda_1^k}\rho^*e^{-G_1^*(p_1)}\left(\frac{8(1+\alpha_1)^2}{\rho_*h_1(0)}\right)^2\left(D_l+O(\delta^{-2\alpha_l})\right).\nonumber
	\end{equation}
	
	At this point, for $l=m_1+1,\cdots,m$,  we set $\tau_l=\sqrt{8 e^{G_l^*(p_l)}h_l(0)}\tau,$ and invoke Theorem 3.3 and Remark 3.2 of \cite{byz-1} as follows, 
	\begin{align*}
		\int_{B(p_l,\tau_l)}\rho_kHe^{{\rm w}_k}d\mu&=\int_{B_{\tau_l}}\rho_kh_le^{f_k+\chi}e^{u_k}dx\\
		&=8\pi-8\pi \frac{e^{-\lambda_l^k}}{e^{-\lambda_l^k}+a_k\tau_l^2}
		+O(e^{-(2-\epsilon_0)\lambda_l^k})\\
		&\quad-\frac{\pi e^{-\lambda_l^k}}{2}(\Delta (\log h_l)(0)+\rho_k-2K(p_l))\rho_kh_l(0)b_{0,k},
	\end{align*}
	where $\epsilon_0>0$ is a small constant and 
	\[b_{0,k}=\int_0^{\tau_l \bar \epsilon_k^{-1}}\frac{r^3(1-a_kr^2)}{(1+a_kr^2)^3}dr,\quad a_k=\frac{\rho_k h_l(0)}8.\]
	Elementary arguments show that,
	\[b_{0,k}=\frac{1}{2a_k^2}(-\lambda_l^k-\log (a_k\tau_l^{2})+2)+O(e^{-\lambda_l^k}).\]
	For $m_1+1\le l\le m$, using this expression of $b_{0,k}$ we have, 
	\begin{equation}
		\label{ener-reg-2}
		\begin{aligned}
			&\int_{B(p_l,\tau_l)}\rho_kHe^{{\rm w}_k}d\mu\\
			&=8\pi -\frac{16\pi}{\rho_kh_l(0)}(\Delta \log h_l(0)+\rho_k-2K(p_l))e^{-\lambda_l^k}\\
			&\quad\times\left(-\lambda_l^k-\log\frac{\rho_kh_l(0)\tau_l^2}{8}+2\right)-\frac{64\pi}{\rho_kh_l(0)}\tau_l^{-2}e^{-\lambda_l^k}
			+O(e^{-(2-\epsilon_0)\lambda_l^k})\\
			&=8\pi-\frac{16\pi}{\rho_kh_m^2(0)e^{G_m^*(p_m)}}e^{-\lambda_m^k}h_l(0)e^{G_l^*(p_l)}(\Delta \log h_l(0)+\rho_k-2K(p_l)) \\
			&\quad\times \left(-\lambda_m^k-2\log h_m(0)-G_m^*(p_m)+\log \frac{8 h_l(0) e^{G_l^*(p_l)}}{\rho_k \tau_l^2}+2\right) \\
			&\quad-\frac{64\pi }{\rho_k h_m^2(0)e^{G_m^*(p_m)}}e^{-\lambda_m^k}h_l(0)e^{G_l^*(p_l)}\tau_l^{-2}
			+O(e^{-(2-\epsilon_0)\lambda_m^k}). 
		\end{aligned}
	\end{equation}
	It is well known that (\cite{chen-lin-sharp}) that the non-degeneracy assumption ${\rm det}(D^2f^*)\neq 0$ implies \eqref{p_kj-location}, that is,\[|p_l^k-p_l|=O(\lambda_l^ke^{-\lambda_l^k})\] so $p_l^k$ can be replaced by $p_l$. Next we evaluate $\int_{\Omega_l\setminus B(p_l,\tau_l)}\rho_kHe^{{\rm w}_k}d\mu$ as follows,
	\begin{equation}
		\label{outside-int}
		\begin{aligned}
			\int_{\Omega_l\setminus B(p_l,\tau_l)}\rho_kHe^{{\rm w}_k}d\mu&=\int_{\Omega_l\setminus B(p_l,\tau_l)}\rho_kh_le^{\overline{{\rm w}_k}+\Phi}d\mu\\
			&= \rho_k\int_{\Omega_l\setminus B(p_l,\tau_l)}h_le^{-\lambda_l^k}\frac{64}{(\rho_kh_l(0))^2}e^{-G_l^*(p_l)+\Phi}d\mu\\
			&=\frac{64e^{-\lambda_l^k}}{\rho_kh_l(0)}\int_{\Omega_l\setminus B(p_l,\tau_l)}e^{\Phi_l(x,p_l)}d\mu ,
		\end{aligned}
	\end{equation}
	where $\Phi_l(x,p_l)$ has been defined in \eqref{Phi_j}. Concerning the last integral we have, 
	\begin{align*}
		&\int_{\Omega_l\setminus B(p_l,\tau_l)}e^{\Phi_l(x,p_l)}d\mu\\
		&=\int_{\Omega_l\setminus B(p_l,\delta)}e^{\Phi_l(x,p_l)}d\mu+\int_{B(p_l,\delta)\setminus B(p_l,\tau_l)}e^{\Phi_l(x,p_l)}d\mu\\
		&=\int_{\Omega_l\setminus B(p_l,\delta)}e^{\Phi_l(x,p_l)}d\mu+
		\int_{B(p_l,\delta)\setminus B(p_l,\tau_l)}\frac{h_l}{h_l(0)}|x|^{-4}e^{G_l^*(x)-G_l^*(p_l)}dx\\
		&=\int_{\Omega_l\setminus B(p_l,\delta)}e^{\Phi_l(x,p_l)}d\mu+\pi(\tau_l^{-2}-\delta^{-2})\\
		&\quad+\frac{\pi}2(\Delta \log h_l(0)+\rho_k-2K(p_l))\log \frac{\delta}{\tau_l}+O(\tau_l).
	\end{align*}
	Then we see that 
	\begin{equation*}
		\label{local-energy}
		\begin{aligned}
			\int_{\Omega_l}\rho_kHe^{{\rm w}_k}d\mu
			&=8\pi-\frac{16\pi}{\rho_kh_m^2(0)e^{G_m^*(p_m)}}e^{-\lambda_m^k}h_l(0)e^{G_l^*(p_l)}(\Delta \log h_l(0)+\rho_k-2K(p_l)) \\
			&\quad\times \left(-\lambda_m^k-2\log h_m(0)-G_m^*(p_m)+\log \left(\frac{\tau_l^{-2}}{\rho_k}\right)+2\right) \\
			&\quad+\frac{64e^{-\lambda_m^k}h_l(0)e^{G_l^*(p_l)}}{\rho_kh_m^2(0)e^{G_m^*(p_m)}}\left (\int_{\Omega_l\setminus B(p_l,\tau_l)}e^{\Phi_l(x,p_l)}d\mu-\pi \tau_l^{-2}+O(\tau_l)\right). 
		\end{aligned}
	\end{equation*}
	Putting the estimates on regular points and singular sources together, we have
	\begin{align*}
		&\rho_k-\rho_*\\
		&=\frac{16\pi}{\rho_*h_m^2(0)e^{G_m^*(p_m)}}e^{-\lambda_m^k}L(\mathbf{p})
		\cdot (\lambda_m^k+2\log h_m(0)+G_m^*(p_m)+\log (\tau^2 \rho_k)-2)\\
		&\quad+\frac{64 e^{-\lambda_m^k}}{\rho_*h_m^2(0)e^{G_*(p_m)}}\sum_{l=m_1+1}^mh_l(0)e^{G_l^*(p_l)}\left(\int_{\Omega_l\setminus B(p_l,\tau_l)}e^{\Phi_l(x,p_l)}d\mu-\frac{\pi}{\tau_l^2}+O(\tau_l)\right )  \\
		&\quad+64e^{-\lambda_1^k}\rho^*e^{-G_1^*(p_1)}\frac{(1+\alpha_1)^4}{\rho^2_*h_1^2(0)}\bigg (\sum_{s=1}^{m_1}\left(\int_{\Omega_s\setminus B(p_s,\tau_s)}He^{\Phi}d\mu-\frac{\pi h_s(0)e^{G_s^*(p_s)}}{(1+\alpha_s)\tau_s^{2+2\alpha_s}}\right )\\
		&\quad+O(\tau^{\sigma})e^{-\lambda_1^k}), 
	\end{align*}
	where $\sigma>0$ is a small positive number. By using the relation between $\lambda_1^k$ and $\lambda_m^k$, we can rewrite the above as follows,
	\begin{equation*}\label{combine-2}
		\begin{aligned}
			&\rho_k-\rho_*\\
			&=\frac{16\pi}{\rho_*h_m^2(0)e^{G_m^*(p_m)}}e^{-\lambda_m^k}L(\mathbf{p})
			\cdot (\lambda_m^k+\log (\rho_k h_m^2(0)e^{G_m^*(p_m)}\tau^2)-2)\\
			&\quad+\frac{64 e^{-\lambda_m^k}}{\rho_*h_m^2(0)e^{G_*(p_m)}}\sum_{l=1}^mh_l(0)e^{G_l^*(p_l)}\left (\int_{\Omega_l\setminus B(p_l,\tau_l)}e^{\Phi_l(x,p_l)}d\mu-\frac{\pi}{(1+\alpha_l)\tau_l^{2+2\alpha_l}}\right ) \\
			&\quad+o(\tau^{\sigma})e^{-\lambda_m^k}+O(e^{-(1+\epsilon_0)\lambda_1^k}). 
		\end{aligned}
	\end{equation*}
	Thus, Theorem \ref{negative-b} is established.
\end{proof}

\subsection{An estimate about a Dirichlet problem}

In this subsection we estimate an integral in the following Dirichlet problem:
\begin{equation}\label{dirichlet-lin}
	\begin{cases}
		\Delta {\rm w}_k+\rho_k\frac{h(x)e^{{\rm w}_k(x)}}{\int_{\Omega}h(x)e^{{\rm w}_k(x)}}=0,\quad &\mbox{in}\quad \Omega,\\
		\\
		{\rm w}_k(x)=0,\quad &\mbox{on}\quad \partial \Omega.
	\end{cases}
\end{equation}
where $\Omega$ is a bounded smooth domain in $\mathbb R^2$ and $h(x)$ is a positive $C^5$ function in $\Omega$, $\rho_k$ tends to a constant $\rho$ and $\max\limits_{\overline\Omega}{\rm w}_k(x)\to \infty$.  Set 
\[I_{{\rm w}_k}=\log \int_{\Omega_k}h(x)e^{{\rm w}_k(x)}dx,\]
and we assume that $h(x)\ge 0$ may be zero at some blow-up points. Suppose $p_1,\cdots,p_m$ are blow-up points, $p_1,\cdots,p_{m_1}$ are singular sources and around $p_l$ ($l=1,\cdots,m_1$), $h(x)=|x-p_l|^{2\alpha_l}h_l$, while the remaining blow-up points, $p_{m_1+1},\cdots,p_m$, are regular blow-up points in $\Omega$. So if we use $\alpha_l$ to denote the strength of singularity at each blow-up point, we have 
\[\alpha_l\neq 0,\quad 1\le l\le m_1, \quad \alpha_l=0,\quad m_1+1\le l\le m. \]
We set 
\[v_k={\rm w}_k-I_{{\rm w}_k},\] and $\lambda_l^k$ to be the maximum of $v_k$ around $p_l$, and $\alpha_M$ to be the largest index. Then we have 
\begin{prop}\label{est-muk}
	\[I_{{\rm w}_k}=\lambda_l^k+2\log \frac{\rho_k h_l(0)}{8(1+\alpha_l)^2}+G_l^*(p_l)+O(e^{-\epsilon_0\lambda_1^k}),\quad l=1,\cdots,m\]
	for some small $\epsilon_0>0$. 
\end{prop}

\begin{proof} 
	Clearly $v_k=w_k-I_{{\rm w}_k}$
	satisfies,
	\[\Delta v_k+\rho_kh(x)e^{v_k}=0,\quad \mbox{in}\quad \Omega,\]
	and $v_k=-I_{{\rm w}_k}$ on $\partial \Omega_k$. 
	By the Green representation formula for $v_k$,
	for $x$ away from the singular sources we have, 
	$$v_k(x)=\sum_{j=1}^m8\pi(1+\alpha_j)G_\Omega(x,p_j)-I_{{\rm w}_k}+O(e^{-\epsilon_0\lambda_1^k}),\quad 
	x\in \Omega\setminus \bigcup_{j=1}^m B_\epsilon(p_j)$$
	for some small $\epsilon>0$. Around each $p_j$, the harmonic function that eliminates the oscillation of $v_k$ on $\partial B_\epsilon(p_j)$ takes the form, 
	\[\phi_j^k(x)=G_j^*(x)-G_j^*(p_j)+O(e^{-\epsilon_0\lambda_1^k}).\]
	On the other hand the expansion around $p_j$ reads as follows, 
	\[v_k(x)=\log \frac{e^{\lambda_j^k}}{\left(1+e^{\lambda_j^k}\frac{\rho_k h_j(0)}{8(1+\alpha_j)^2}|x-p_j|^{2\alpha_j+2}\right)^2}+\phi_j^k(x)+\Pi_k\]
	where (see either section \ref{sec3:byz2} below or \cite{gluck,zhang1,byz-1}),
	\[\Pi_k=\begin{cases}O(e^{-\lambda_l^k/{(1+\alpha_M)}}),\quad &\mbox{if}\quad \alpha_M>0,\\
		O(e^{-\lambda_l^k}\lambda_k^2),\quad &\mbox{if}\quad \alpha_M=0,\\
		O(e^{-\lambda_l^k}),\quad &\mbox{if}\quad \alpha_M<0.
	\end{cases}\]
	The comparison between the two expressions shows that,
	\begin{equation}\label{ck-id}
		I_{{\rm w}_k}=\lambda_l^k+2\log \frac{\rho_k h_l(0)}{8(1+\alpha_l)^2}+G_l^*(p_l)+\Pi_k+O(e^{-\epsilon_0\lambda_1^k}),\quad l=1,\cdots,m.
	\end{equation}
	Proposition \ref{est-muk} follows from (\ref{ck-id}). \end{proof}

As particularly relevant case for applications to be discussed in sections \ref{sec:meansink} and \ref{section:firstsec} is when we have just one blow-up point placed on a negative source $p$ with strength $\beta\in (-1,0)$. Around the blow-up point $p$ we write $h(x)=|x-p|^{2\beta}\bar h(x)$. We shall need an accurate estimate about $\rho_k-8\pi(1+\beta)$ under the assumption \eqref{neg-crit-2}. Let $$c_k=I_{{\rm w}_k},$$ we recall that, for $x$ far away from $p$, we  have that,
\[v_k(x)=\int_{\Omega}G(x,\eta)\rho_khe^{v_k}d\eta-c_k,\]
and a rough evaluation shows that,
\begin{equation}\label{rough-ck}
	v_k(x)=8\pi(1+\beta)G_\Omega(x,p)-c_k+O(e^{-\lambda_k}),\quad x\in \Omega\setminus B_\tau(p)
\end{equation}
On the other hand, since $\int_{\Omega}he^{v_k}=1$ and setting $$
G^*(x)=8\pi(1+\beta)R_\Omega(x,p),
$$ 
we can use Theorem 3.2 of \cite{byz-1} as follows,  
\begin{equation}\label{about-ck}
	\begin{aligned}
		\rho_k=&\int_{B_\tau(p)}\rho_k|x-p|^{2\beta}\bar h(x)e^{v_k}dx+\int_{\Omega\setminus B_\tau(p)}\rho_k|x-p|^{2\beta}\bar h(x)e^{v_k}\\
		=&~8\pi(1+\beta)-\frac{64\pi(1+\beta)^3}{\rho_k \bar h(0)}\tau^{-2-2\beta}e^{-\lambda_k}\\
		&+ e^{-c_k}\int_{\Omega\setminus B_\tau(p)}|x-p|^{-4-2\beta}e^{G^*(x)}\rho_k \bar h(x)dx+o_\tau(1)(e^{-\lambda_k})\\
		=&~8\pi(1+\beta)-\frac{\pi\rho_k\bar h(0)}{1+\beta}e^{G^*(p)}\tau^{-2-2\beta}e^{-c_k}\\
		&+e^{-c_k}\int_{\Omega\setminus B_\tau(p)}|x-p|^{-4-2\beta}e^{G^*(x)}\rho_k\bar h(x)dx+o_\tau(1)e^{-\lambda_k},
	\end{aligned}
\end{equation}
where we have used (\ref{ck-id})-(\ref{rough-ck}) and 
$o_\tau(1)$ is a uniformly bounded quantity such that $o_\tau(1)\to 0$ as $\tau\to 0$.
Thus, we get
\begin{equation}\label{im-ck}
	\rho_k-8\pi(1+\beta)=e^{-c_k}\rho_k(D_{\beta}+o_\tau(1)),
\end{equation}
where $D_{\beta}$ is defined as follows,
\begin{equation}\label{d-beta}
	D_{\beta}:=\lim_{\tau\to 0^+}\bar h(0)e^{G^*(p)}\left(\int_{\Omega\setminus B_\tau(p)}e^{\Phi_{\Omega,p}}dx-\frac{\pi}{1+\beta}\tau^{-2-2\beta}\right)
\end{equation}
and 
\begin{equation}\label{Phi-beta}\Phi_{\Omega,p}(x)=8\pi(1+\beta)G_\Omega(x,p)-G^*(p)+\log \frac{\bar h(x)}{\bar h(0)}+2\beta \log |x-p|.
\end{equation}
Remark that the limit defining $D_\beta$ is well defined due to 
\eqref{neg-crit-2}, see also Lemma \ref{Iexp} in section \ref{sec:meansink}. 


\subsection{A uniqueness lemma}
The proof of the following Lemma can be found in \cite{chen-lin,wu-zhang-ccm}, for $\alpha>0$ and in \cite{byz-1} for
$\alpha<0$.

\begin{lem}
	\label{lem1}
	Let $\alpha>-1$, $\alpha\not \in \mathbb N\cup \{0\}$ and $\phi$ be a $C^2$ solution of
	$$\begin{cases}
		\Delta \phi+8(1+\alpha )^2|x|^{2\alpha}e^{U_{\alpha}}\phi=0\;\;
		\hbox{ {\rm in} }\;\; \mathbb R^2,\\
		\\
		|\phi(x)|\le C(1+|x|)^{\tau},
		\quad x\in \mathbb R^2,
	\end{cases}
	$$
	where
	$$
	U_{\alpha}(x)=\log \frac{1}{(1+|x|^{2\alpha+2})^2}
	\quad \hbox{ and } \quad
	\tau\in [0,1).
	$$
	Then there exists some constant $b_0$ such that
	\begin{equation*}
		\phi(x)= b_0\frac{1-|x|^{2(1+\alpha)}}{1+|x|^{2(1+\alpha)}}.
	\end{equation*}
\end{lem}

\bigskip

The following lemma has been proved in \cite{chen-lin-sharp}.
\begin{lem}
	Let  $ \varphi $ be a $ C^2 $ solution of
	\begin{equation*}
		\begin{cases}
			\Delta \varphi+e^U\varphi=0\quad & {\rm in} \ \;\mathbb{R}^2,
			\\
			\\
			|\varphi| \leq c\big(1+|x|\big)^{\kappa} \quad & {\rm in} \ \;\mathbb{R}^2,
		\end{cases}
	\end{equation*}
	where $ U(x)=\log\frac{8}{(1+|x|^2)^2} $ and $ \kappa \in[0,1) $. Then there exist constants $b_0$, $b_1$, $b_2$ such that
	\begin{equation*}
		\varphi= b_0\varphi_0+b_1\varphi_1+b_2\varphi_2,
	\end{equation*}
	where
	\begin{equation*}
		\varphi_0(x)= \frac{1-|x|^2}{1+|x|^2},\quad \varphi_1(x)= \frac{x_1}{1+|x|^2},\quad \varphi_2(x)= \frac{x_2}{1+|x|^2}.
	\end{equation*}
	
\end{lem}

\medskip

\section{Some results in \cite{byz-1} about the local asymptotic analysis near a blow-up point}
\label{sec3:byz2}

In this section we summarize the asymptotic expansions of a sequence of blow-up solutions near a blow-up point as recently
derived in \cite{byz-1}. We assume that $u_k$ denotes a sequence of blow-up solutions of
\begin{equation*}
	\Delta u_k+|x|^{2\alpha}h_k(x)e^{u_k}=0 \quad \mbox{in}\quad B_{2\tau},
\end{equation*}
where $B_{\tau}\subset \mathbb{R}^2$ for $\tau>0$ is the ball centered at the origin with radius $\tau>0$,
\begin{equation}\label{assum-H}
	\frac{1}{C}\le h_k(x)\le C, \quad \|D^{m}h_k\|_{L^{\infty}(B_{2\tau})}\le C
	\quad \forall x\in B_{2\tau}, \quad \forall m\le 5,
\end{equation}
and
\begin{equation}\label{assum-H.1}
	\int_{B_{2\tau}}|x|^{2\alpha}h_k(x)e^{u_k}<C,
\end{equation}
for some uniform $C>0$. The sequence $\{u_k\}$ has its only blow-up point at the origin:
\begin{equation}\label{only-bu}
	\mbox{there exists}\quad  z_k\to 0, \mbox{ such that} \quad \lim_{k\to \infty} u_k(z_k)\to \infty
\end{equation}
and for any fixed $K\Subset B_{2\tau}\setminus \{0\}$, there exists $C(K)>0$ such that
\begin{equation}\label{only-bu-2}
	u_k(x)\le C(K),\quad x\in K.
\end{equation}
It is also standard to assume that $u_k$ has bounded oscillation on $\partial B_{\tau}$
\begin{equation}\label{u-k-boc}
	|u_k(x)-u_k(y)|\le C,\quad \forall x,y\in \partial B_{\tau},
\end{equation}
for some uniform  $C>0$. As before, we set $\psi_k$ to be the harmonic function which  encodes the boundary oscillation of $u_k$:
\begin{equation}\label{psik.1}
	\begin{cases}
		\Delta \psi_k=0  \quad &\hbox{in }\quad B_{\tau},\\
		\\
		\psi_k=u_k-\frac{1}{2\pi\tau}\int_{\partial B_{\tau}}u_kdS
		\quad &\hbox {on }\quad \partial B_{\tau}.
	\end{cases}
\end{equation}
Since $u_k$ has a bounded oscillation on $\partial B_{\tau}$ (see \eqref{u-k-boc}), $u_k-\psi_k$ is constant and uniformly bounded on $\partial B_1$ and $\|D^m\psi_k\|_{B_{\tau/2}}\le C(m)$ for any $m\in \mathbb N$.
The mean value property of harmonic functions also gives $\psi_k(0)=0$.
For the sake of simplicity let us define,
$$
V_k(x)=h_k(x)e^{\psi_k(x)}, \quad \varepsilon_k=e^{-\frac{u_k(0)}{2(1+\alpha)}}
$$
and
\begin{equation} \label{defvk-1}
	v_k(y)=u_k(\varepsilon_ky)+2(1+\alpha)\log \varepsilon_k-\psi_k(\varepsilon_ky),\quad y\in \Omega_k:=B_{\tau/\varepsilon_k}.
\end{equation}
Obviously we have $V_k(0)=h_k(0)$, $\Delta \log V_k(0)=\Delta \log h_k(0)$.
Let
\begin{equation}\label{st-bub}
	U_k(y)=
	-2\log
	\left(1+\frac{h_k(0)}{8(\alpha+1)^2}|y|^{2\alpha+2} \right),
\end{equation}
be the standard bubble which satisfies
\begin{equation}\nonumber \Delta U_k+|y|^{2\alpha}h_k(0)e^{U_k(y)}=0 \quad
	\hbox{ in } \quad \mathbb R^2.
\end{equation}

In the remaining part of this section we split the results into three subsections
according to the sign of $\alpha$ ($\alpha>0$, $\alpha<0$ and $\alpha=0$).

\subsection{Asymptotic analysis for $\alpha>0$}\label{sec:alfap}

We say that a function is separable if it is the product of a polynomial of two variables $P(x_1,x_2)$ times a radial function with certain decay at infinity,
where $P$ is the sum of monomials $x_1^nx_2^m$ where $n,m$ are non-negative integers and at least one among them is odd.
Let us define:
\begin{equation}
	\label{nov10e6}
	c_{1,k}(y)=\varepsilon_kg_k(r)\nabla (\log V_k)(0)\cdot \theta, \quad \theta=(\theta_1,\theta_2),\; \theta_j=y_j/r,~ j=1,2,
\end{equation}
where
\begin{equation}
	\label{mar27e1} g_k(r)=-\frac{2(1+\alpha)}{\alpha }
	\frac{r}{1+\frac{V_k(0)}{8(1+\alpha)^2}r^{2\alpha+2}},
\end{equation}
$c_{0,k}$ to be the unique solution of
\begin{equation*}
	\label{c0k}
	\begin{cases}
		\frac{d^2}{dr^2}c_{0,k}+\frac{1}{r}\frac{d}{dr}c_{0,k}+V_k(0) r^{2\alpha}e^{U_k}c_{0,k}\\
		=-\frac{\varepsilon_k^2}4V_k(0)r^{2\alpha}e^{U_k}\bigg ((g_k+r)^2|\nabla \log V_k(0)|^2+\Delta \log V_k(0)r^2\bigg ),\\
		\\
		c_{0,k}(0)=c'_{0,k}(0)=0, \quad 0<r<\tau/\varepsilon_k,
	\end{cases}
\end{equation*}
$c_{2,k}$ a separable function which satisfies,
\begin{equation*}\label{c2k-e}\Delta c_{2,k}+V_k(0) r^{2\alpha}e^{U_k}c_{2,k}=
	-\varepsilon_k^2r^{2\alpha}e^{U_k}\hat \Theta_2,
\end{equation*}
\noindent
where
\begin{equation*}
	\label{E212}
	\begin{aligned}
		\varepsilon_k^2\hat \Theta_2&=\varepsilon_k^2( E_{2,1}(r)\cos 2\theta+E_{2,2}(r)\sin 2\theta)\\
		&=\varepsilon_k^2\bigg (\bigg(\frac 14\left(\partial_{11}(\log V_k)(0)-\partial_{22}(\log V_k)(0)\right)r^2 \\
		&\quad+\frac 14\left((\partial_1\log V_k(0))^2-(\partial_2\log V_k(0))^2\right)(g_k+r)^2\bigg)\cos 2\theta \\
		&\quad+\left(\partial_{12}(\log V_k)(0)\frac{r^2}{2}+\frac{\partial_1V_k(0)}{V_k(0)}\frac{\partial_2V_k(0)}{V_k(0)}\frac{(g_k+r)^2}{2}\right)\sin 2\theta \bigg ).
	\end{aligned}
\end{equation*}
and
\begin{equation}
	\label{c2k}
	c_{2,k}=\varepsilon_k^{2}f_{2a}(r)\cos (2\theta)+\varepsilon_k^{2}f_{2b}(r)\sin (2\theta),
\end{equation}
where
\begin{equation}
	\label{4.two-fun1}
	|f_{2a}(r)|+|f_{2b}(r)|=O(1+r)^{-2\alpha},\quad 0<r<\varepsilon_k^{-1}.
\end{equation}

In the statement below $\epsilon_0>0$ is a suitable positive number allowed to change even line to line.
\begin{thm}{(\cite{byz-1})}\label{int-pos}
	\begin{align*}
		\int_{B_{\tau}}h_k(x)|x|^{2\alpha}e^{u_k}
		=~&8\pi(1+\alpha)\left(1-\frac{e^{-u_k(0)}}{e^{-u_k(0)}+\frac{h_k(0)}{8(1+\alpha)^2}\tau^{2\alpha+2}}\right)\\
		&+d_{1,k} \Delta \log h_k(0)\varepsilon_k^2
		+T_1(V_k)\varepsilon_k^4+O(\varepsilon_k^{4+\epsilon_0}),
	\end{align*}
	for some $\epsilon_0>0$, where,
	\begin{align*}
		d_{1,k}=&-\frac{2\pi^2}{(1+\alpha ) \sin \frac{\pi}{1+\alpha}}\left(\frac{8(1+\alpha)^2}{V_k(0)}\right)^{\frac 1{1+\alpha}}\\
		&+\left(\frac{16\pi(1+\alpha)^2}{\alpha }\frac{|\nabla \log V_k(0)|^2}{V_k(0)\Delta\log V_k(0)}+\frac{16\pi(1+\alpha)^4}{\alpha V_k(0)}\right)\tau^{-2\alpha}\varepsilon_k^{2\alpha},
	\end{align*}
	and $T_1(V_k)$ is a bounded function of $D^{\beta}V_k(0)$ for $|\beta |=0,1,2,3,4$.\\
	Moreover for $|y|\leq \tau/\varepsilon_k$, we have,
	\begin{align*}
		v_k(y)&=U_k(y)+c_{1,k}(y)+O(\varepsilon_k^{2}(1+|y|)^{\epsilon_0}),\\
		v_k(y)&=U_k(y)+c_{1,k}(y)+c_{0,k}(y)+c_{2,k}(y)+O(\varepsilon_k^{2+\epsilon_0}(1+|y|)^{\epsilon_0}),\\
		v_k(y)&=U_k(y)+c_{1,k}(y)+c_{0,k}(y)+c_{2,k}(y)+c_{*,k}(y)+O(\varepsilon_k^{4+\epsilon_0}(1+|y|)^{\epsilon_0}),
	\end{align*}
	for some $\epsilon_0>0$.
\end{thm}
We refer to \cite{byz-1} for the definition of $c_{*,k}$.

\subsection{Asymptotic analysis around a negative pole}
In this case we use $\beta\in (-1,0)$ to denote the strength of a negative pole and define,
$$\displaystyle\tilde \varepsilon_k=e^{-\frac{\lambda^k}{2(1+\beta)}}.$$
We use the same notations to approximate $v_k$, where
$c_{1,k}$ is defined as in \eqref{nov10e6} with $\alpha$ replaced by $\beta$ in the definition of $g_k$ in \eqref{mar27e1}.
In the statement below $\delta>0$ is a suitable small positive number allowed to change even line to line.
\begin{thm}{\mbox{\rm (\cite{byz-1})}}\label{integral-neg}
	\begin{align*}
		\int_{B_{\tau}}V_k(x)|x|^{2\beta}e^{u_k}
		=~&8\pi(1+\beta)\left(1-\frac{e^{-u_k(0)}}{e^{-u_k(0)}+\frac{h_k(0)}{8(1+\beta)^2}\tau^{2\beta+2}}\right)\\
		&+b_k(\tau,\beta)+\ell_k(\tau,\beta)+O(e^{(-2-\epsilon_0)u_k(0)})
	\end{align*}
	for some $\epsilon_0>0$, where both $b_k(\tau,\beta)$ and $\ell_k(\tau,\beta)$ depend on $|\nabla \log V_k(0)|$ and $\Delta \log V_k(0)$, $b_k(\tau,\beta)\sim e^{-u_k(0)}$, $\ell_k(\tau,\beta)\sim e^{-2u_k(0)}$ for $\tau>0$, and
	\begin{equation*}\label{neg-tau}\lim_{\tau\to 0}\frac{b_k(\tau,\beta)}{e^{-u_k(0)}}=0,\quad
		\lim_{\tau\to 0}\frac{\ell_k(\tau,\beta)}{e^{-2u_k(0)}}=0.
	\end{equation*}
	Moreover for $|y|\leq \tau/\varepsilon_k$, we have,
	\begin{align*}
		v_k(y)&=U_k(y)+c_{1,k}(y)+O(\tilde \varepsilon_k^{2+2\beta-\delta}(1+|y|)^{\delta}),\\
		v_k(y)&=U_k(y)+c_{1,k}(y)+c_{0,k}(y)+c_{2,k}(y)+c_{*,k}(y)+O(\tilde \varepsilon_k^{4+4\beta-\delta}(1+|y|)^{\delta}),
	\end{align*}
	for some small $\delta>0$.
\end{thm}

\subsection{Asymptotic analysis around a regular blow-up point}
In this subsection we analyze the solution near a regular blow-up point and consider the equation,
$$\Delta u_k+\bar h_k(x)e^{u_k}=0\quad \mbox{in}\quad B_{2\tau}. $$
We use $\bar \lambda_k=\max_{B_{2\tau}}u_k$ to denote the height of the bubble
and set  $\bar \varepsilon_k=e^{-\bar \lambda_k/2}$.
We assume in particular \eqref{only-bu}, \eqref{only-bu-2},
(i.e. $0$ is the only blow-up point in $B_{2\tau}$), that the standard uniform bound
$\int_{B_{2\tau}}\bar h_ke^{u_k}\le C$ holds and finally that \eqref{u-k-boc} holds, implying
by well known results (\cite{li-cmp}) that the blow-up point is simple.
We assume that \eqref{assum-H}, \eqref{only-bu}, \eqref{only-bu-2}, \eqref{u-k-boc} and $\int_{B_{2\tau}}\bar h_k(x)e^{u_k}\leq C$ are satisfied and define
$\psi_k$ as in \eqref{psik.1}. Let $V_k(x)=h_k(x)e^{\psi_k(x)}$ and we have the following estimates, see \cite[Theorem 3.3 and Remarks 3.2, 3.5]{byz-1}

\begin{thm}\label{reg-int}
	\begin{equation}\label{total-q-1}
		\int_{B_{\tau}}\bar h_k(x)e^{u_k}dx=8\pi-\frac{8\pi\bar \varepsilon_k^2}{\bar\varepsilon_k^2+a_k\tau^2}-\frac{\pi\varepsilon_k^2}{2}\Delta (\log \bar h_k)(q_k)\bar h_k(q_k)b_{0,k}+O(\bar\varepsilon_k^{4-2\epsilon_0}).
	\end{equation}
	for some $\epsilon_0>0$, where
	\begin{equation*}
		b_{0,k}=\int_0^{\tau/\bar\varepsilon_k}\frac{r^3(1-a_kr^2)}{(1+a_kr^2)^3}dr,\quad a_k=\bar h_k(q_k)/8.
	\end{equation*}
	Moreover we have,
	\begin{equation*}\label{grad-v2-new-0}
		|\bar v_k(y)-(U_k+c_{0,k}+c_{2,k})(y)|\le C\bar \varepsilon_k^{2+\epsilon_0}(1+|y|)^{2\epsilon_0}\quad \mbox{{\rm in}}\quad \Omega_k,
	\end{equation*}
	and
	\begin{equation*}\label{grad-v2-new}
		|\nabla (\bar v_k-U_k-c_{0,k}-c_{2,k})(y)|\le C\tau\bar \varepsilon_k^3,\quad y\in \Omega_k, \quad |y|\sim \bar\varepsilon_k^{-1}.
	\end{equation*}
\end{thm}
\medskip

\section{Proof of the non degeneracy result}

In this section, we prove Theorems \ref{main-theorem-2}-\ref{main-theorem-1}. Since the proof of these theorems are almost the same, we shall only focus on Theorem \ref{main-theorem-2}. To simplify our discussion, we assume there are only three blow-up points: 
$q$, $p_1$, $p_2$, where $q$ is a regular blow-up point,$p_1$ is a positive singular source with $\alpha\equiv \alpha_M>0$, $p_2$ is a
negative singular source with $\beta\in (-1,0)$. Let $\nu_k$ be a blow-up sequence of \eqref{m-equ},
$$
{\rm w}_k=\nu_k+4\pi\alpha G(x,p_1)+4\pi \beta G(x,p_2),
$$
which satisfies
\begin{equation*}\label{uik-tem}
	\Delta_g {\rm w}_k+\rho_k\left(\frac{He^{{\rm w}_k}}{\int_MHe^{{\rm w}_k}}-1\right)=0,
\end{equation*}
where $\alpha=\alpha_M>0$ and
$$H(x)=h(x) e^{-4\pi\alpha G(x,p_1)-4\pi\beta G(x,p_2)}.$$
Since by adding any constant to ${\rm w}_k$ in \eqref{uik-tem} we still come up with a solution, without loss of generality we may assume that
${\rm w}_k$ satisfy
\begin{equation*}
	\int_M He^{{\rm w}_k}dV_g=1.
\end{equation*}
We argue by contradiction and assume that $\phi^k$ is a non trivial sequence of solutions to
\begin{equation*}
	\Delta_g  {\phi}^k+\rho{He^{{\rm w}_k}}\bigg(\phi^k-\int_M{He^{{\rm w}_k}\phi^k}\bigg)=0 \quad {\rm in}\ \; M.
\end{equation*}

Let us set $\lambda^k={\rm w}_k(p_1)$, $\hat \lambda^k={\rm w}_k(p_2)$ and $\bar \lambda^k=\max\limits_{B(q,\tau)}{\rm w}_k$, then
it is well known that
\begin{equation}\label{same-ord}
	\hat \lambda^k-\lambda^k=O(1),\quad \bar \lambda^k-\lambda^k=O(1).
\end{equation}

We work in local isothermal coordinates around $p_1,p_2,q$, where the metric locally takes the form,
$$ds^2=e^{\chi}((dx_1)^2+(dx_2)^2),$$ for some $\chi$ that satisfies,
\begin{equation}\label{def-phi}
	\Delta \chi+2Ke^{\chi}=0 \quad{\rm in }\quad B_{\tau}, \quad \chi(0)=|\nabla \chi(0)|=0.
\end{equation}
Then we define $f_k$ to be any solution of,
\begin{equation}
	\label{def-fk}
	\Delta f_k=e^\chi \rho_k\quad {\rm in}\quad B_\tau,\qquad f_k(0)=0.
\end{equation}
In particular we choose local coordinates so that $p_i$, $q$ are locally the origin of a ball $x\in B_\tau=B_\tau(0)$.
Remark that in our setting we have,
\begin{equation}\label{rho-3}
	\begin{cases}
		\rho_*=24\pi+8\pi\alpha+8\pi\beta,\quad N^*=4\pi(\alpha+\beta),\\
		\\
		G^*_1(x)=8\pi(1+\alpha)R(x,p_1)+8\pi(1+\beta)G(x,p_2)+8\pi G(x,q).
	\end{cases}
\end{equation}

\begin{rem}
	Throughout this section, $B_\tau=B_\tau(0)$ (or $B_\tau(\bar q)$) will represent the ball centered at the origin (or $\bar q$) of some local isothermal coordinates $x\in B_\tau$, while $\Omega(p,\tau)\subset M$
	denotes a geodesic ball. Additionally, if $B_\tau$ is the ball in local isothermal coordinates centered at some point $0=x(p), p\in M$,
	we will denote by $\Omega(p,\tau)\subset M$
	the pre-image of $B_\tau$. We can always choose $\tau$ small enough to ensure that
	$\Omega(p_i,\tau)$, $i=1,2$ and $\Omega(q,\tau)$ are simply connected and at
	positive distance one from each other. After scaling,
	$x=\varepsilon y$ for some $\varepsilon>0$, we will denote
	$B_{\tau/\varepsilon}=B_{\tau/\varepsilon}(0)$. By a slightly abuse of notations, we will use the same symbols,
	say ${\rm w}_k$, $\xi_k$, to denote functions when expressed in different local coordinate  systems. However, the context will clarify the meaning of the symbols as needed.
\end{rem}
\noindent
Thus, working in these local coordinates centered at $p_1$, $p_2$, $q$ respectively,
we have the local variables $x\in B_\tau$ and we define
\begin{equation}
	\label{around-p1-h}
	\tilde h_k=\rho_k h(x)e^{-4\pi \alpha R(x,p_1)-4\pi \beta G(x,p_2)+\chi(x)+f_k},
\end{equation}
\begin{equation}
	\label{around-p2-h}
	\hat h_k=\rho_k h(x)e^{-4\pi \alpha G(x,p_1)-4\pi \beta R(x,p_2)+\chi(x)+f_k},
\end{equation}
\begin{equation}
	\label{around-q}
	\bar h_k=\rho_k h(x)e^{-4\pi \alpha G(x,p_1)-4\pi \beta G(x,p_2)+\chi(x)+f_k}.
\end{equation}

Obviously $\tilde h_k,\hat h_k, \bar h_k$ satisfy \eqref{assum-H} and clearly (\ref{small-rho-k}) holds as well. Indeed from
$\rho^k=\rho^k\int_MHe^{u_k}$ we first evaluate the integrals around blow-up points and then check that the integrals outside
the bubbling disks are of order $O(e^{-\lambda^k})$.

\begin{rem}\label{center-q}
	Concerning \eqref{around-q}, for technical reasons, it will be useful to define $\bar q^k\to 0$ to be the maximum points of $u_k$ in $B_\tau$ and work
	in the local coordinate system centered at $\bar q^k$. In particular $q^k\in M$ will denote the pre-images of these points
	via the local isothermal map and we will work sometime, possibly taking a smaller $\tau$, with \eqref{around-q} where $0=x(q^k)$.
\end{rem}

By using also \eqref{def-phi} we see that the assumption $L(\mathbf{p})\neq 0$ takes the form
$$\Delta_g \log h(p_1)+\rho_*-N^*-2K(p_1)\neq 0, $$
which, for any $k$ large enough, is equivalent to (see (\ref{rho-3}))
$$\Delta \log \tilde h_k(0)\neq 0.$$
At this point let us define,
$$<\phi^k>_k=\int_M\frac{He^{{\rm w}_k}\phi^k}{\int_M H e^{{\rm w}_k}{\rm d}\mu},\quad\sigma_k=\|\phi^k-<\phi^k>_k\|_{L^\infty(M)},$$
and
\begin{equation*}
	\xi_k:= (\phi^k-<\phi^k>_k)/\sigma_k.
\end{equation*}
Then in local coordinates around $p_1$, so that $0=x(p_1)$, $x\in B_\tau=B_\tau(0)$, $\xi_k$ satisfies,
\begin{equation*}
	\Delta\xi_k(x)+|x|^{2\alpha}\tilde h_k(x) \tilde c_k(x)\xi_k(x)=0\quad {{\rm in}}\quad B_{\tau},
\end{equation*}
where $\tilde {h}_k$ is defined in \eqref{around-p1-h} and $\tilde c_k(x)$ denotes the local coordinates expression of
$e^{u_k(x)}$, that is,
$$
\tilde c_k(x)= e^{u_k(x)}.
$$
In local coordinates around $p_2$, so that $0=x(p_2)$, $x\in B_\tau=B_\tau(0)$, $\xi_k$ satisfies,
\begin{equation*}
	\Delta\xi_k(x)+|x|^{2\beta}\hat h_k(x) \hat  c_k(x)\xi_k(x)=0\quad {{\rm in}}\quad B_{\tau},
\end{equation*}
where $\hat {h}_k$ is defined in \eqref{around-p2-h} and $\hat c_k(x)$ denotes the local coordinates expression of
$e^{u_k(x)}$, that is,
$$
\hat c_k(x)= e^{u_k(x)}.
$$
In local coordinates around $q$, so that $0=x(q)$, $x\in B_\tau=B_\tau(0)$, $\xi_k$ satisfies,
\begin{equation*}
	\Delta\xi_k(x)+\bar h_k(x) \bar  c_k(x)\xi_k(x)=0\quad {{\rm in}}\quad B_{\tau},
\end{equation*}
where $\bar {h}_k$ is defined in \eqref{around-p2-h} and $\bar c_k(x)$ denotes the local coordinates expression of
$e^{u_k(x)}$, that is,
$$
\bar c_k(x)= e^{u_k(x)}.
$$
As in \cite{byz-1}, after a suitable scaling in local coordinates, we see that the limits of $\xi_k$ takes the form
\begin{equation}\label{linim-1}
	b_0\frac{1-A|y|^{2\alpha+2}}{1+A|y|^{2\alpha+2}}~\mbox{\rm near } p_1,\quad  A=\lim_{k\to \infty}\frac{\tilde h_k(0)}{8(1+\alpha)^2};
\end{equation}
\begin{equation}\label{linim-2}
	b_0\frac{1-B|y|^{2\beta+2}}{1+B|y|^{2\beta+2}}~ \mbox{\rm near } p_2,\quad B=\lim_{k\to \infty}\frac{\hat h_k(0)}{8(1+\beta)^2};
\end{equation}
\begin{equation}\label{linim-3}
	b_0\frac{1-C|y|^2}{1+C|y|^2}+b_1\frac{y_1}{1+C|y|^2}+b_2\frac{y_2}{1+C|y|^2}~ \mbox{\rm near } q,  \quad C=\lim_{k\to \infty}\frac{\bar h_k(0)}8.
\end{equation}
In the following two subsections, we shall show that all these coefficients are zero.
\subsection{Proof of $b_0=0$.}
First of all, locally around $p_1$ the rescaled function $\xi_k(\varepsilon_ky)$ converges on compact subsets of $\mathbb{R}^2$  to a solution of the following linearized equation
$$\Delta \xi+|y|^{2\alpha}\tilde h(0)e^{U}\xi=0,\quad \tilde h(0)=\lim_{k\to \infty}\tilde h_k(0),$$
where  $U$ is the limit of the standard bubble $U_k$, then we have that
\begin{equation}\label{lim-1.0}
	\xi(y)=b_0\frac{1-c|y|^{2\alpha+2}}{1+c|y|^{2\alpha+2}},\quad c=\frac{\tilde h(0)}{8(1+\alpha)^2}.
\end{equation}
As a first approximation, we choose
\begin{equation}
	\label{w0kxik}
	w_{0,\xi}^k=b_0^k\frac{1-a_k|y|^{2\alpha+2}}{1+a_k|y|^{2\alpha+2}},\quad a_k=\tilde h_k(0)/8(1+\alpha)^2.
\end{equation}
It is well known that $w_{0,\xi}^k$ is a solution of
\begin{equation}\label{e-w0kxi}
	\Delta w_{0,\xi}^k +|y|^{2\alpha}\tilde h_k(0)e^{U_k}w_{0,\xi}^k=0,
\end{equation}
where $b_0^k$ is chosen such that $w_{0,\xi}^k(0)=\xi_k(0)$. Let $\psi_\xi^k$ be the harmonic function which encodes the oscillation of $\xi_k$ on $\partial B_\tau$,
\begin{align*}
	\Delta \psi_{\xi}^k=0~ \mbox{in}~ B_\tau,\quad
	\psi_{\xi}^k(x)=\xi_k(x)-\frac{1}{2\pi\tau}\int_{\partial B_{\tau}}\xi_k(s)ds~\mathrm{for}~ x\in \partial B_\tau,
\end{align*}
then we see that $\xi_k-\psi_\xi^k$ satisfies,
\begin{equation}\label{comp-1}
	\Delta (\xi_k-\psi_{\xi}^k)+|x|^{2\alpha}\tilde h_k\tilde c_k(\xi_k-\psi_{\xi}^k)=-|x|^{2\alpha}\tilde h_k\tilde c_k\psi_{\xi}^k.
\end{equation}
In order to show that the leading term of $(\xi_k-\psi_{\xi}^k)(\varepsilon_ky)$ is $w_{0,\xi}^k$, we write \eqref{e-w0kxi} as follows, 
\begin{equation}
	\label{e-hat-xi}
	\begin{aligned}
		&\Delta w_{0,\xi}^k+|y|^{2\alpha}\tilde h_k(\varepsilon_ky)\varepsilon_k^{2+2\alpha}\tilde c_k(\varepsilon_ky)w_{0,\xi}^k\\
		&=|y|^{2\alpha}\tilde h_k(0)\bigg (\frac{\tilde h_k(\varepsilon_ky)}{\tilde h_k(0)}\varepsilon_k^{2+2\alpha}\tilde c_k(\varepsilon_ky)-e^{U_k}\bigg )w_{0,\xi}^k(y).
	\end{aligned}
\end{equation}
Next, we set
\begin{equation}\label{tildewk}
	\tilde w_k(y):=\xi_k(\varepsilon_ky)-\psi_{\xi}^k(\varepsilon_ky)-w_{0,\xi}^k(y).
\end{equation}
From (\ref{comp-1}) and (\ref{e-hat-xi}) we have
\begin{equation}
	\label{crucial-1}
	\begin{aligned}
		&\Delta \tilde w_k+|y|^{2\alpha}\tilde h_k(\varepsilon_k y)\varepsilon_k^{2+2\alpha}
		\tilde c_k(\varepsilon_ky)\tilde w_k\\
		&=\tilde h_k(0) |y|^{2\alpha}
		\bigg (e^{U_k}-\frac{\tilde h_k(\varepsilon_ky)}{\tilde h_k(0)}\varepsilon_k^{2+2\alpha}
		\tilde c_k(\varepsilon_ky)\bigg)w_{0,\xi}^k(y)\\
		&\quad-|y|^{2\alpha}\tilde h_k(\varepsilon_k y)\varepsilon_k^{2+2\alpha}
		\tilde c_k(\varepsilon_ky)\psi_{\xi}^k(\varepsilon_ky).
	\end{aligned}
\end{equation}
By the construction of $\tilde w_k$, it is readily seen that,
$$\tilde w_k(0)=0,\quad \tilde w_k=\mbox{constant} \quad \mbox{on}\quad \partial B_{\tau/\varepsilon_k}.$$
Remark that the last term of (\ref{crucial-1}) is of order $O(\varepsilon_k)(1+r)^{-3-2\alpha}$ whence we obtain a first estimate about $\tilde w_k$:
\begin{equation}\label{first-e}
	|\tilde w_k(y)|\le C(\delta)\varepsilon_k (1+|y|)^{\delta}\quad \mbox{in}\quad B_{\tau/\varepsilon_k}.
\end{equation}
The analysis around $p_2$ is similar, we define (recall that now we have local coordinates where $0=x(p_2)$),
\begin{equation}\label{hat-wk}
	\hat w_k(y)=\xi_k(\hat \varepsilon_ky)-\hat \psi^k_{\xi}(\hat \varepsilon_ky)-\hat w^k_{0,\xi}(y),
\end{equation}
where $\hat \varepsilon_k=e^{-\frac{\hat \lambda_k}{2(1+\beta)}}$, $\hat \psi^k_{\xi}$ is the harmonic
function that encodes the oscillation of $\xi_k$ on $\partial B_\tau$, $\hat w^k_{0,\xi}$ is the same as
$w^k_{0,\xi}$ just with $\beta$ replacing $\alpha$ and the local limit of
$\xi_k(\hat \varepsilon_ky)$ is now \eqref{linim-2} which replaces \eqref{lim-1.0}.
By using the same argument adopted above for $\tilde w_k$, we have that,
\begin{equation}\label{sec-e}
	|\hat w_k(y)|\le C(\delta)\hat \varepsilon_k^{1+\beta}(1+|y|)^{\delta}=C(\delta)e^{-\lambda^k/2}(1+|y|)^{\delta} \quad
	\mbox{in}\quad B_{\tau/\varepsilon_k}.
\end{equation}

Next, we describe the expansion of $\xi_k$ near $q$. In this case \eqref{lim-1.0} is replaced by \eqref{linim-3}. As before, we denote $\bar q^k$ the maximum points of $u^k$ in $B_\tau$ and work in the local coordinate system centered at $\bar q^k$. In particular $q^k\in M$ will denote the pre-image of these points via the local isothermal map and we will work with the equation with $0=x(q^k)$. As usual the kernel functions are the first terms in the approximation of
$\xi_k(\bar q^k+\bar \varepsilon_ky)$,
$$\bar w_{0,\xi}^k(y)=\bar b_0^k\frac{1-\frac{\bar h_2^k(0)}8|y|^2}
{1+\frac{\bar h_k(0)}8|y|^2}+
\bar b_1^k\frac{y_1}{1+\frac{\bar  h_k(0)}8|y|^2}+
\bar b_2^k\frac{y_2}{1+\frac{\bar  h_k(0)}8|y|^2}.$$
Then we have
\begin{equation*}
	\xi_k(\bar q^k)-\bar w_{0,\xi}^k(0)=0\quad\mbox{and}\quad \nabla(\xi_k(\bar q^k+\bar\varepsilon_ky)-\bar w_{0,\xi}^k(\bar\varepsilon_ky))|_{y=0}=0.
\end{equation*}
Next we set $\bar\psi_\xi^k$ to be the harmonic function which encodes the oscillation of $\xi_k(\bar q^k+\bar\varepsilon_ky)-\bar w_{0,\xi}^k(\bar\varepsilon_ky)$ on $\partial B_{\tau/\bar\varepsilon_k}$. Let us define,
\begin{equation}\label{bar-w2k}
	\bar w_k(y)=\xi_k(\bar q^k+\bar \varepsilon_ky)-\bar w_{0,\xi}^k(y)-\bar \psi_{\xi}^k(\bar\varepsilon_ky),
\end{equation}
then we have,
\begin{equation}\label{eqwbar1}
	\begin{cases}
		\Delta \bar w_k(y)+\bar \varepsilon_k^2(\bar h_k \bar c_k)(\bar q^k+\bar \varepsilon_ky)\bar w_k(y)
		=\frac{O(\bar \varepsilon_k)}{(1+|y|)^3} \quad\mbox{in}\quad B_{\tau /\bar\varepsilon_k},\\
		\\
		\bar w_k(0)=0\quad \mbox{and}\quad \bar w_k={\rm constant}~\mbox{on}~\partial B_{\tau/\bar \varepsilon_k}.
	\end{cases}
\end{equation}
By using the fact that $\bar \psi_\xi^k$ is a harmonic, it is not difficult to check that $\bar w_k(0)=0$ and $\nabla \bar w_k(0)=-\nabla\bar\psi_\xi^k(0)=O(\bar\varepsilon_k)$. Then, using the Green representation formula and standard potential estimates, we have that,
\begin{equation}\label{pre-xi}
	|\bar w_k(y)|\le C(\delta)\bar \varepsilon_k(1+|y|)^{\delta} \quad {\rm in}\quad B_{\tau/\bar \varepsilon_k}.
\end{equation}
Next we present the following improved estimate on the oscillation of $\xi_k$

\begin{lem}\label{osci-xi-better}
	\begin{equation}\label{osi-vars}
		\xi_k(x_1)-\xi_k(x_2)=O(\varepsilon_k^{2}+e^{-\lambda_1^k/2})=O(\varepsilon_k^2+\varepsilon_k^{1+\alpha}),
	\end{equation}
	$\forall\, x_1,x_2\in
	M\setminus \{\Omega({p_1},\tau)\cup \Omega({p_2},\tau)\cup \Omega({q},\tau)\}$.
\end{lem}

\begin{proof}
	{We shall divide our proof into two steps, in the first step we derive the following estimate
		\begin{equation}
			\label{4.est-step1}
			\xi_k(x_1)-\xi_k(x_2)=O(\varepsilon_k),
		\end{equation}
		for any $x_1,x_2\in
		M\setminus \left\{\Omega({p_1},\tau)\cup \Omega({p_2},\tau)\cup \Omega({q},\tau)\right\}.$ Then in the second step we shall improve \eqref{4.est-step1} and obtain \eqref{osi-vars}.}
	\smallskip
	
	\noindent {\bf Step 1.} Using the Green representation formula for
	$\xi_k$, we have, for $x\in M\setminus \{\Omega({p_1},\tau)\cup \Omega({p_2},\tau)\cup \Omega({q},\tau)\},$
	\begin{align*}
		\xi_k(x)&=\xi_{k,a}+\int_M G(x,\eta)\rho_kH(\eta) e^{{\rm w}_k(\eta)}\xi_k(\eta)d\eta\\
		&=\xi_{k,a}+\sum_{i=1}^2\int_{\Omega({p_i},\tau/2)}G(x,\eta)\rho_kH(\eta) e^{{\rm w}_k(\eta)}\xi_k(\eta)d\eta\\
		&\quad+\int_{\Omega({q},\tau/2)}G(x,\eta)\rho_kH(\eta) e^{{\rm w}_k(\eta)}\xi_k(\eta)d\eta\\
		&\quad+\int_{M\setminus \{\Omega({p_1},\tau/2)\cup \Omega({p_2},\tau/2)\cup \Omega(q,\tau/2)\}}G(x,\eta)\rho_kH(\eta) e^{{\rm w}_k(\eta)}\xi_k(\eta)d\eta\\
		&=\xi_{k,a}+I_1+I_2+I_3+I_4,
	\end{align*}
	with obvious meaning and 
	$$\xi_{k,a}=\int_M\left(\int_MG(x,\eta)\rho_kH(\eta) e^{{\rm w}_k(\eta)}\xi_k(\eta)d\eta\right)H(x)e^{{\rm w}(x)}dx.$$
	Clearly $I_4$ is of order $O(e^{-\lambda^k})$. For the other terms we write,
	\begin{align*} I_1&=G(x,p_1)\int_{\Omega({p_1},\tau/2)}\rho_kHe^{{\rm w}_k}\xi_k+
		\int_{\Omega({p_1},\tau/2)}(G(x,\eta)-G(x,p_1))\rho_kHe^{{\rm w}_k}\xi_k\\
		&=I_{11}+I_{12}.
	\end{align*}
	Concerning $I_{12}$, by the Mean Value Theorem, in local coordinates after scaling we have $$G(x,\eta)-G(x,p_1)=\varepsilon_k(a_0\cdot y)+O(\varepsilon_k^2|y|^2),$$
	then in view of  (\ref{first-e}), we see that,
	\begin{align*}
		I_{12}
		=~&\varepsilon_k\int_{B_{\frac{\tau}{2 \varepsilon_k}}}\tilde h_k(0)(a_0 \cdot y)|y|^{2\alpha}e^{U_k}
		(\psi_{\xi}^k(\varepsilon_ky)+\psi_k(\varepsilon_ky)+w_{0,\xi}^k(y)\\
		&+O(\varepsilon_k(1+|y|)^{\delta}))+O(\varepsilon_k^2)\\
		=~&\varepsilon_k\int_{B_{\frac{\tau}{2 \varepsilon_k}}}\tilde h_k(0)(a_0 \cdot y)|y|^{2\alpha}e^{U_k}w_{0,\xi}^k(y)+O(\varepsilon_k^2)=O(\varepsilon_k^2),
	\end{align*}
	where we used $\psi^k_\xi(\varepsilon_ky)=O(\varepsilon_k|y|)$ and $\psi_k(\varepsilon_ky)=O(\varepsilon_k|y|)$. We notice that the first term on the right vanishes due to the fact that $w^k_{0,\xi}$ is a 
	radial function and the integrand is separable. Consider $I_{11}$ we have
	\begin{align*}
		I_{11}=~&G(x,p_1)\int_{\Omega({p_1},\tau/2)}\rho_kHe^{{\rm w}_k}(\xi_k-\psi_{\xi}^k)+G(x,p_1)\int_{\Omega({p_1},\tau/2)}\rho_kHe^{{\rm w}_k}\psi_{\xi}^k\\
		=~&I_{1a}+I_{1b},
	\end{align*}
	where, writing $\psi_{\xi}^k(\varepsilon_ky)=\varepsilon_k a_k \cdot y+O(\varepsilon_k^2|y|^2)$, we have
	$$
	I_{1b}=O(1)\int_{B_{\frac{\tau}{2 \varepsilon_k}}} \tilde h_k(0)|y|^{2\alpha}e^{U_k} \psi_{\xi}^k(\varepsilon_ky)=O(\varepsilon_k^2),
	$$
	again because the term of order $O({\varepsilon_k})$ is the integral of a
	separable function. On the other side, $I_{1a}=O(\varepsilon_k)$ since we can write,
	\begin{equation}
		\label{i1a}
		\begin{aligned}
			\int_{\Omega({p_1},\tau/2)}\rho_kHe^{{\rm w}_k}(\xi_k-\psi_{\xi}^k)
			=&\int_{B_{\frac{\tau}{2 \varepsilon_k}}}|y|^{2\alpha}\tilde h_k(0)e^{U_k }w_{0,\xi}^kdy+O(\varepsilon_k)\\
			=&-\int_{\partial B_{\frac{\tau}{2 \varepsilon_k}}}\partial_{\nu}w_{0,\xi}^k+O(\varepsilon_k)\\
			=&~O(\varepsilon_k^{2+\epsilon_0})+O(\varepsilon_k)=O(\varepsilon_k),
		\end{aligned}
	\end{equation}
	where in the first equality we used \eqref{first-e} and at last
	the equation for $w_{0,\xi}^k$ and the fact that $\partial_{\nu}w_{0,\xi}^k=O(r^{-3-2\alpha})$
	on the boundary. It is important to remark that
	$I_{1a}$ is the only term of order $\varepsilon_k$,  and all the other terms are either of order $\varepsilon_k^2$ or $e^{-\lambda^k/2}=\varepsilon_k^{1+\alpha}.$ While for the integral $I_2$, by \eqref{sec-e} and a similar argument we deduce that $I_2=O(e^{-\lambda^k/2})$. Concerning $I_3$ we have,
	\begin{align*}
		I_3&=\int_{\Omega(q^k,\tau/2)}G(x,\eta)\rho_kH e^{{\rm w}_k}\xi_k+O(\e_k^2)\\
		&=\int_{\Omega(q^k,\tau/2)}(G(x,\eta)-G(x,q^k))\rho_kHe^{{\rm w}_k}\xi_k
		+G(x,q^k)\int_{\Omega(q^k,\tau/2)}\rho_kHe^{{\rm w}_k}\xi_k\\
		&\quad+O(\e_k^2)\\
		&=I_{31}+I_{32}+O(\varepsilon_k^{2}).
	\end{align*}
	where we also used \eqref{same-ord} and $\bar q^k-\bar q=\bar q^k=
	O(e^{-\bar \lambda^k}\bar \lambda^k)$, see \eqref{p_kj-location}. By arguing as for $I_{12}$ above we find that $I_{32}=O(e^{-\lambda^k/2})$. Also, by using  (\ref{pre-xi}) instead of (\ref{first-e}),
	the same argument adopted for $I_{11}$ shows that  $I_{31}=O(e^{-\lambda^k/2})$. We skip these details to avoid repetitions.
	As a consequence, one can show that
	$$|\xi_k(x_1)-\xi_k(x_2)|=O(\varepsilon_k)+O(e^{-\lambda^k/2})=O(\varepsilon_k+\varepsilon^{1+\alpha}_k),$$
	for any $x_1,x_2\in M\setminus \{\Omega({p_1},\tau)\cup \Omega({p_2},\tau)\cup \Omega({q},\tau)\},$ that is exactly the estimation \eqref{4.est-step1}. Therefore, we see immediately that $\psi^k_{\xi}$ satisfies,
	\begin{equation}\label{impxi:1}
		|\psi^k_{\xi}(\varepsilon_k y)|\le C\varepsilon_k^2(1+|y|)\quad {\rm in} \quad B_{\tau/\varepsilon_k}.
	\end{equation}
	We will use these facts in a sort of bootstrap argument to improve the estimates about the integration of $I_{1a}$ and the
	oscillation of $\xi_k$ far away from blow-up points later.
	\medskip
	
	\noindent {\bf Step 2.} We invoke the intermediate estimate in Theorem \ref{int-pos} about the expansion of $v_k$ near $p_1$:
	\begin{equation}\label{p-rough-2}
		v_k=U_k+c_{0,k}+c_{1,k}+c_{2,k}+O(\varepsilon_k^{2+\delta})(1+|y|)^{\delta}.
	\end{equation}
	Around $0=x(p_1)$ we make the following expansion
	\begin{equation}
		\label{fullexp:1}
		\begin{aligned}
			&\tilde h_k(\varepsilon_ky)\varepsilon_k^{2+2\alpha}\tilde c_k(\varepsilon_ky)\\
			&=\tilde h_k(0)\exp \left\{v_k+\log \frac{V_k(\varepsilon_ky)}{V_k(0)}\right\}(1+O(\varepsilon_k^{2+\epsilon_0}))\\
			&=\tilde h_k(0)e^{U_k}\left(1+c_{0,k}+\tilde c_{1,k}+\tilde c_{2,k}+\frac {\Delta (\log V_k)(0)|y|^2\varepsilon_k^2}{4}+\frac{\tilde c_{1,k}^2}{2}\right)\\
			&\quad+E_k,
		\end{aligned}
	\end{equation}
	where $c_{0,k}$ is a radial function corresponding to the $\varepsilon_k^2$ order term in the expansion of $v_k$ (see section \ref{sec:alfap}), $V_k=\tilde h_ke^{\psi_k}$ and $\psi_k$ is the harmonic function that encodes the boundary oscillation of $u_k$ on $\partial B_\tau$, $\tilde c_{1,k}$ and $\tilde c_{2,k}$ are given below:
	\begin{align*}
		&\tilde c_{1,k}=\varepsilon_k\nabla\log \tilde h_k(0)\cdot \theta\left(r-\frac{2(1+\alpha)}{\alpha}\frac{r}{1+a_kr^{2\alpha+2}}\right),\quad \tilde c_{2,k}=c_{2,k}+\Theta_2\varepsilon_k^2r^2,
	\end{align*}
	where $a_k=\frac{\tilde h_k(0)}{8(1+\alpha)^2}$ and $\Theta_2$ are the collection of projections of $\log\frac{V_k(\e_ky)}{V_k(0)}$ onto non-radial modes. The left error term $E_k$ in \eqref{fullexp:1} satisfies,
	$$|E_k|\le C\varepsilon_k^{2+\epsilon_0}(1+|y|)^{-2+\epsilon_0-2\alpha}.$$

	\noindent At this point, around $p_1$ we define $w_{1,\xi}^k$,
	the next term in the approximation of $\xi_k$, to be a suitable solution of,
	\begin{equation}\label{xi-2nd}
		\Delta w_{1,\xi}^k+|y|^{2\alpha}\tilde h_k(0)e^{U_k}w_{1,\xi}^k=-|y|^{2\alpha}\tilde h_k(0)e^{U_k}\tilde c_{1,k} w_{0,\xi}^k.
	\end{equation}
	It turns out that \eqref{xi-2nd} can be explicitly solved in terms of a function. Indeed, let $c_{1,k}$ be defined as in \eqref{nov10e6} and we set,
	$$\hat c_1^k(x)=c_{1,k}\left(x/{\varepsilon_k}\right),$$
	that is,
	$$\hat c_1^k(x)=-\frac{2(1+\alpha )}{\alpha}(\nabla \log \tilde h_k(0)\cdot \theta)\frac{e^{-\lambda^k}|x|}{\left(e^{-\lambda^k}+\frac{\tilde{h}_k(0)}{8(1+\alpha)^2}|x|^{2\alpha+2}\right)}.$$
	Putting $\lambda=\lambda^k$ and differentiating with respect to $\lambda$ we have that,
	$$\frac{d}{d\lambda}\hat c_1^k(x)
	=\frac{\tilde{h}_k(0)}{4\alpha(1+\alpha)}(\nabla \log \tilde h_k(0)\cdot \theta) \frac{e^{-\lambda^k}|x|^{2\alpha +3}}{\left(e^{-\lambda^k}+\frac{\tilde{h}_k(0)}{8(1+\alpha)^2}|x|^{2\alpha+2}\right)^2}.$$
	Setting
	$$w_{1,\xi}^k(y)=b_0^k\left(\frac{d}{d\lambda}\hat c_1^k\right)(\varepsilon_ky)\quad \mbox{with}
	\quad \varepsilon_k=e^{-\frac{\lambda_1^k}{2(1+\alpha )}}, $$ we have
	\begin{equation}\label{w1-explicit}
		w_{1,\xi}^k(y)=b_0^k\frac{\tilde{h}_k(0)}{4\alpha(1+\alpha)}(\nabla \log \tilde h_k(0)\cdot \theta)
		\frac{\varepsilon_k r^{2\alpha+3}}{\left(1+\frac{\tilde{h}_k(0)}{8(1+\alpha)^2}r^{2+2\alpha }\right)^2}.
	\end{equation}
	It is not difficult to check that $w_{1,\xi}^k$ satisfies \eqref{xi-2nd},
	$$w_{1,\xi}^k(0)=|\nabla w_{1,\xi}^k(0)|=0,\quad\mbox{and}\quad
	|w_{1,\xi}^k(y)|\le C\varepsilon_k (1+|y|)^{-1}~{\rm in}~B_{\tau/\varepsilon_k}.$$
	Next, let us we write the equation for $w_{1,\xi}^k$ in the following form:
	\begin{equation}
		\label{e-w1k}
		\begin{aligned}
			&\Delta w_{1,\xi}^k+|y|^{2\alpha}\tilde h_k(\varepsilon_k y)
			\varepsilon_k^{2+2\alpha}\tilde c_k(\varepsilon_k y)w_{1,\xi}^k \\
			&=|y|^{2\alpha}\tilde h_k(0)\left[\left(\frac{\tilde h_k(\varepsilon_ky)}{\tilde h_k(0)}
			\varepsilon_k^{2+2\alpha}\tilde c_k(\varepsilon_k y)-e^{U_k}\right)w_{1,\xi}^k-e^{U_k}\tilde c_{1,k}w_{0,\xi}^k\right],
		\end{aligned}
	\end{equation}
	and then set
	$$z_k(y)=\tilde w_k(y)-w_{1,\xi}^k(y)=\xi_k(\varepsilon_ky)-\psi_{\xi}^k(\varepsilon_ky)-w_{0,\xi}^k(y)-w_{1,\xi}^k(y).$$
	Therefore, in view of \eqref{crucial-1}, \eqref{impxi:1} and \eqref{e-w1k}, the equation for $z_k$ reads:
	\begin{align*}
		&\Delta z_k+|y|^{2\alpha}\tilde h_k(\varepsilon_k y)\varepsilon_k^{2+2\alpha}\tilde c_k(\varepsilon_k y)z_k\\
		&=\tilde h_k(0)|y|^{2\alpha}\left(e^{U_k}- \frac{\tilde h_k(\varepsilon_ky)}{\tilde h_k(0)}
		\varepsilon_k^{2+2\alpha}\tilde c_k(\varepsilon_k y)\right )(w_{0,\xi}^k+w_{1,\xi}^k)\\
		&\quad+\tilde h_k(0) |y|^{2\alpha}e^{U_k}\tilde c_{1,k} w_{0,\xi}^k+O(\varepsilon_k^{2+\epsilon_0})(1+|y|)^{-2-2\alpha+\epsilon_0}\\
		&=O\left(\varepsilon_k^2\right)(1+|y|)^{-2-2\alpha}.
	\end{align*}
	By using the usual argument based on Green's representation formula, we have
	\begin{equation}\label{est-z}
		|z_k(y)|\le C(\delta)\varepsilon_k^2(1+|y|)^{\delta},
	\end{equation}
	it implies that the oscillation of $z_k$ on
	$\partial B_{\tau/\varepsilon_k}$ is of order $O(\varepsilon_k^{2-\delta})$. In view of \eqref{est-z}, $I_{1a}$ can be improved as follows
	\begin{align*}
		I_{1a}&=G(x,p_1)\int_{B(p_1,\frac{\tau}2)}\rho_k He^{{\rm w}_k}(\xi_k-\psi_{\xi}^k)\\
		&=G(x,p_1)\int_{B_{\frac{\tau}{2 \varepsilon_k}}}\tilde h_k(\varepsilon_ky)e^{U_k}
		(w_{0,\xi}^k+w_{1,\xi}^k)dy+O(\varepsilon_k^2).
	\end{align*}
	Then we make a Taylor expansion for $\tilde h_k(\varepsilon_ky)$, its
	zero-th order term relative to $w_{0,\xi}^k$ has already been shown above
	to be of order $O(\varepsilon_k^{2+\epsilon_0})$ (see \eqref{i1a}), the one relative to $w_{1,\xi}^k$ vanishes since the
	integrand is separable, the term proportional to $\varepsilon_k w_{0,\xi}^k$ vanishes
	again because the integrand is separable, while the one proportional to
	$\varepsilon_k w_{1,\xi}^k$ is already of order $O(\varepsilon_k^2)$.
	{On the other hand, we have already mentioned in Step 1 that all the terms except for $I_{1a}$ are of order either $\varepsilon_k^2$ or $\varepsilon_k^{1+\alpha}$}. Therefore, we derive that
	\begin{equation*}
		|\xi_k(x_1)-\xi_k(x_2)|=O(\varepsilon_k^2)+O(e^{-\lambda^k/2}),
	\end{equation*}
	$\forall x_1,x_2\in M\setminus \{\Omega({p_1},\tau)\cup \Omega({p_2},\tau)\cup \Omega({q},\tau)\}$,
	which is (\ref{osi-vars}). Hence, we finish the proof.
\end{proof}

Based on \eqref{osi-vars}, we can get a better estimate for $\psi_{\xi}^k$:
\begin{equation}\label{im-ph-xi}
	|\psi_{\xi}^k(\varepsilon_ky)|\le C\varepsilon_k^{2+\epsilon_0}(1+|y|).
\end{equation}
We will see that this estimate implies that all the terms in the expansions involving
$\psi_{\xi}^k$ are in fact negligible.


By using \eqref{fullexp:1}, the equation of $z_k$ can be written as follows,
\begin{align*}
	&\Delta z_k+|y|^{2\alpha}\tilde h_k(\varepsilon_k y)\varepsilon_k^{2+2\alpha}\tilde c_k(\varepsilon_k y)z_k\\
	&=-\tilde h_k(0)|y|^{2\alpha}e^{U_k}\bigg (w_{0,\xi}^k\big (c_{0,k}+\tilde c_{2,k}+\frac 14\Delta (\log \tilde h_k)(0)|y|^2\varepsilon_k^2+\frac  12(\tilde c_{1,k})^2\big )\nonumber\\
	&\quad+w_{1,\xi}^k\tilde c_{1,k}+O(\varepsilon_k^{2+\epsilon_0})(1+|y|)^{-2-2\alpha+\epsilon_0} \bigg ). \nonumber
\end{align*}
By using separable functions to remove separable terms, we can write the equation of the radial part of $z_k$, which we denote $\tilde z_k$,
\begin{align*}
	&\Delta \tilde z_k+|y|^{2\alpha}\tilde h_k(0)e^{U_k}\tilde z_k\\
	&=-\tilde h_k(0)|y|^{2\alpha}e^{U_k}\left[w_{0,\xi}^k\left(c_{0,k}+\frac{\Delta (\log \tilde h_k)(0)|y|^2\varepsilon_k^2}{4}+\frac  {(\tilde c_{1,k})^2_r}{2}\right)+(w_{1,\xi}^k\tilde c_{1,k})_r\right],
\end{align*}
where we let $(A)_{\theta},(A)_r$ denote the angular part and radial part of $A$ respectively.
Corresponding to these terms we construct $z_0^k$ to solve
\begin{equation}\label{eq-z0}
	\begin{cases}\dfrac{d^2}{dr^2}z_0^k+\frac 1r\dfrac{d}{dr}z_0^k+r^{2\alpha}\tilde h_k(0)e^{U_k}z_0^k=F_0^k,
		\quad 0<r<\tau/\varepsilon_k,\\
		\\
		z_0^k(0)=\frac{d}{dr}z_0^k(0)=0,
	\end{cases}
\end{equation}
where
\begin{align*}
	F_0^k(r)
	=&-\tilde h_k(0)|y|^{2\alpha}e^{U_k}\bigg (w_{0,\xi}^k\big (c_{0,k}+\frac 14\Delta (\log \tilde h_k)(0)|y|^2\varepsilon_k^2\\
	&+\frac  14\varepsilon_k^2|\nabla \log \tilde h_k(0)|^2(g_k+r)^2\big )
	+(w_{1,\xi}^k\tilde c_{1,k})_r\bigg )\nonumber
\end{align*}
and we used
$$\left(\frac 12(\tilde c_{1,k})^2\right)_r=\frac 14\varepsilon_k^2|\nabla \log \tilde h_k(0)|^2(g_k+r)^2.$$
By standard potential estimates it is not difficult to see that
$$|z_0^k|\le C\varepsilon_k^2(1+r)^{-2\alpha}\log (1+r).$$
Next, we introduce $w_{2,\xi}^k$, which is used for removing the separable terms of the order $O(\varepsilon_k^2)$
in $\tilde w_k$,
\begin{equation}
	\label{first:estimate}
	\begin{aligned}
		&\Delta w_{2,\xi}^k+|y|^{2\alpha}h_k(0)e^{U_k}w_{2,\xi}^k\\
		&=-|y|^{2\alpha} h_k(0)e^{U_k}\bigg (w_{0,\xi}^k\big (\tilde c_{2,k}+\frac 12(\tilde c_{1,k}^2)_{\theta}\big )+(\tilde c_{1,k}w_{1,\xi}^k)_{\theta}\bigg ).
	\end{aligned}
\end{equation}
Using \eqref{4.two-fun1} we see that
\begin{equation}\label{1st-w2k}
	|w_{2,\xi}^k|\le C\varepsilon_k^2(1+|y|)^{-2\alpha} \quad {\rm in} \quad B_{\tau/\varepsilon_k}.
\end{equation}
By standard potential estimates as usual we have
$$
\left|\tilde w_k-\sum_{i=1}^2w_{i,\xi}^k-z_0^k(y)\right|\le C\varepsilon_k^{2+\epsilon_0}(1+|y|)^{-2\alpha+\epsilon_0}.
$$
Next, following a similar argument about $w_{0,\xi}^k$ in \eqref{i1a}, we evaluate the integral around $p_1$ as follows,
\begin{align*}
	&\int_{B_{\tau/\varepsilon_k}}|y|^{2\alpha}\tilde h_k(\varepsilon_ky)\varepsilon_k^{2+2\alpha}
	\tilde c_k(\varepsilon_ky)\xi_k(\varepsilon_k y)dy\\
	&=\tilde h_k(0)\int_{B_{\tau/\varepsilon_k}}|y|^{2\alpha}\exp\{U_k+c_{0,k}+\tilde c_{1,k}+\tilde c_{2,k}+
	\frac 14\Delta (\log \tilde h_k(0))\varepsilon_k^2|y|^2  \\
	&\quad+\frac 12 (\tilde c_{1,k})^2\}
	(w_{0,\xi}^k+w_{1,\xi}^k+z_0^k)+O(\varepsilon_k^{2+\epsilon_0}) \\
	&=\int_{B_{\tau/\varepsilon_k}}\tilde h_k(0)|y|^{2\alpha}e^{U_k}\bigg (\big (1+ c_{0,k}+\frac 12 (\tilde c_{1,k})^2+\frac 14\varepsilon_k^2\Delta (\log \tilde h_k)(0)|y|^2\big )w_{0,\xi}^k \\
	&\quad+(\tilde c_{1,k} w_{1,\xi}^k)_r+z_0^k \bigg )
	+O(\varepsilon_k^{2+\epsilon_0}).
\end{align*}
\noindent
Remark that, by using the equation for $z_0^k$ in \eqref{eq-z0}, we have,
\begin{equation}\label{key-c-2}
	\int_{B_{\tau/\varepsilon_k}}\varepsilon_k^{2+2\alpha}|y|^{2\alpha}
	\tilde h_k(\varepsilon_ky)\tilde c_k(\varepsilon_ky)\xi_k(\varepsilon_ky)dy
	=-\int_{\partial B_{\tau/\varepsilon_k}}\frac{\partial z_0^k}{\partial \nu}+O(\varepsilon_k^{2+\epsilon_0}).
\end{equation}

\noindent We set $A_0^k(x)=b_0^kc_{0,k}(|x|/\varepsilon_k)$, then
\begin{align*}
	&\Delta A_0^k+\tilde h_k(0)\frac{|x|^{2\alpha}e^{-\lambda^k}}{\left(e^{-\lambda^k}+\frac{\tilde{h}_k(0)}{8(1+\alpha)^2}|x|^{2\alpha+2}\right)^2}A_0^k\\
	&=-b_0^k\frac {\tilde h_k(0)}{4}\frac{e^{-\lambda^k}|x|^{2\alpha+2}}{\left(e^{-\lambda^k}+\frac{\tilde{h}_k(0)}{8(1+\alpha)^2}|x|^{2\alpha+2}\right)^2} \\
	&\quad\times\left(\left(1-\frac{2(1+\alpha)}{\alpha}\frac{e^{-\lambda^k}}{e^{-\lambda^k}+\frac{\tilde{h}_k(0)}{8(1+\alpha)^2}|x|^{2\alpha+2}}\right)^2|\nabla \log \tilde h_k(0)|^2+\Delta \log \tilde h_k(0)\right).
\end{align*}
Now we define $A_{\lambda}^k(x)=\frac{d}{d\lambda}A_0^k(x)$ and after a lengthy calculation
we see that
\begin{equation}\label{important-z0}
	z_0^k(y)=A_{\lambda}^k(\varepsilon_k y).
\end{equation}
By using (\ref{important-z0}) we deduce that,
\begin{align*}
	\int_{B_{\tau/\varepsilon_k}}\Delta z_0^k=~&\varepsilon_k^2\int_{B_{\tau/\varepsilon_k}}
	\Delta A_{\lambda}^k(\varepsilon_ky)dy\\
	=~&\int_{B_{\tau}}\Delta A_{\lambda}^k(x)dx=\frac{d}{d\lambda}\int_{B_{\tau}}\Delta A_0^k(x)dx.
\end{align*}
Since
$$\Delta A_0^k(x)=b_0^k\Delta c_{0,k}\left(\frac{|x|}{\varepsilon_k}\right)\varepsilon_k^{-2},$$
we have
\begin{align*}
	\frac{1}{b_0^k}\int_{B_{\tau}}\Delta A_0^k(x)dx=&\int_{B_{\tau/\varepsilon_k}}\Delta c_{0,k}(y)dy\\
	=&
	\int_{\partial B_{\tau/\varepsilon_k}}\frac{\partial c_{0,k}}{\partial \nu}=
	d_{1,k} \Delta \log \tilde h_k(0)\varepsilon_k^2+O(\varepsilon_k^{2+\epsilon_0}),
\end{align*}
and 
\begin{align*}
	\int_{\partial B_{\tau/\varepsilon_k}}\frac{\partial z_0^k}{\partial \nu}=
	\int_{B_{\tau/\varepsilon_k}}\Delta z_0^k=-\frac{1}{1+\alpha}b_0^kd_{1,k}
	\Delta (\log \tilde h_k)(0)\varepsilon_k^2+O(b_0^k\varepsilon_k^{2+\epsilon_0}),
\end{align*}
which, in view of \eqref{key-c-2} eventually implies that 
\begin{equation}\label{p1-leading}
	\int_{\Omega(p_1,\tau)}\rho_kH e^{{\rm w}_k}\xi_k=-\frac{1}{1+\alpha}b_0^kd_{1,k}
	\Delta (\log \tilde h_k)(0)\varepsilon_k^2+O(b_0^k\varepsilon_k^{2+\epsilon_0}).
\end{equation}
\medskip

On the other side, the contribution of the integral around $p_2$ is very small,
\begin{equation}\label{p2-minor}
	\int_{\Omega(p_2,\tau)}\rho_kHe^{{\rm w}_k}\xi_k=O(\varepsilon_k^{2+\epsilon_0}).
\end{equation}
The process is almost the same as what we did for $p_1$. Indeed, we make an expansion of $\xi_k$ around $p_2$, $w_{1,\xi}^k$ is set for $c_{1,k}$ in the expansion of $u_k$, $z_0^k$ is used for treating the second order radial term in the expansion of $\xi_k$ and $c_{0,k}$ of $u_k$. The only difference is that the scaling is now with respect to $\hat \varepsilon_k=\varepsilon_k^{\frac{1+\alpha}{1+\beta}}$, then
$O(\hat \varepsilon_k^2)=O(\varepsilon_k^{2+\epsilon_0})$. We skip these details to avoid repetitions.

%
%

Concerning the integral on $\Omega(q,\tau)$, we first improve the estimate of
$\bar w_k$ (see \eqref{pre-xi}). At first, we can show that 
\begin{align*}
	\mbox{The oscillation of}~\bar\psi_\xi^k(y)~\mbox{on}~{\partial B_{\tau/\bar\varepsilon_k}}~\mbox{is of order}~O(\e_k^2+\bar\varepsilon_k),
\end{align*}
where we used Lemma \ref{osci-xi-better} and the explicit formula of $\bar w_{0,\xi}^k(y)$. As a consequence, by a standard argument we could derive an estimate like
(\ref{im-ph-xi}) holds for $\bar \psi^k_\xi(\varepsilon_ky)$ as well, i.e.,
\begin{equation}
	\label{4.equ-barpsi}
	|\bar\psi_\xi^k(\varepsilon_ky)|\leq C(\varepsilon_k^2\bar\varepsilon_k+\bar\varepsilon_k^2)(1+|y|),\quad y\in B_{\tau/\bar\varepsilon_k}.
\end{equation}
While at $\bar q^k$,  by standard elliptic estimate for harmonic function we have $|\nabla\bar\psi^k_\xi(0)|=O(\varepsilon_k^{2+\epsilon_0})$. Therefore, in local
coordinates such that $0=x(q)$ and after scaling $x=\bar q^k+\bar \varepsilon_k y$,
where $\bar \varepsilon_k=e^{-\bar\lambda^k/2}$,
we can write
the equation for $\bar w_k$ as follows,
$$\Delta \bar w_k+\bar h_k(\bar q^k+\bar \varepsilon_k y) \bar \varepsilon_k^{2}
\bar c_k(\bar q^k+\bar \varepsilon_k y)\bar w_k=
O(\varepsilon_k^{2+\epsilon_0})(1+|y|)^{-3}
\quad {\rm in} \quad B_{\tau/\bar \varepsilon_k}.$$
Here we remark that, compared with \eqref{eqwbar1}, the improvement in the estimates of the
right hand side is obtained because of the vanishing rate of the gradient of the coefficient function
for regular blow-up points (see \eqref{first-deriv-est}) and
the improved estimate \eqref{4.equ-barpsi} for $\bar\psi_{\xi}^k$. On the other hand, we have
$$\bar w_k(0)=0,\quad \nabla\bar w_k(0)=O(\varepsilon_k^{2+\epsilon_0})\quad {\rm and}\quad \bar w_k \mbox{ is a constant on } \partial B_{\tau/\bar \varepsilon_k}.$$
As a consequence, by the usual potential estimates, we conclude that,
\begin{equation*}
	|\bar w_k(y)|\le C(\delta)\varepsilon_k^{2+\epsilon_0}(1+|y|)^{\delta},\quad |y|\le \tau/\bar \varepsilon_k.
\end{equation*}
Next, by using the expansion of $\bar c_k$ in we have,
$$
\int_{\Omega(q,\tau)}\rho_kH (x)e^{{\rm w}_k(x)}\xi_k
=\int_{B_{\tau/\bar \varepsilon_k}}\bar h_k(\bar q^k+\bar \varepsilon_k y)e^{U_k}
\xi_k(\bar q^k+\bar \varepsilon_k y)
+O(\varepsilon_k^{2+\epsilon_0}).
$$
Remark that $e^{-\lambda^k/2}=O(\varepsilon_k^{1+\epsilon_0})$, then by \eqref{bar-w2k}-\eqref{pre-xi} we see that
all the terms including $\bar \psi^k_\xi(\bar \varepsilon_ky)$ are of order
$O(\varepsilon_k^{2+\epsilon_0})$.
Also, neglecting terms which vanish due to the integrand is separable, we have
\begin{align*}
	&\int_{B_{\tau/\bar \varepsilon_k}}\bar h_k(\bar q^k +\bar \varepsilon_k y)e^{U_k}
	\xi_k(\bar q^k+\bar \varepsilon_k y)\\
	&=\int_{B_{\tau/\bar \varepsilon_k}}\bar h_k(\bar q^k)e^{U_k}(\bar w^k_{0,\xi}+\bar w_k)+\bar \varepsilon_k\int_{B_{\tau/\bar \varepsilon_k}}(\nabla \log\bar h_k(\bar q^k) \cdot y)
	\frac{e^{U_k}(\bar b_k \cdot y)}{1+\frac{\bar h_k(\bar q^k)}{8}|y|^2}\\
	&\quad+\bar \varepsilon_k\int_{B_{\tau/\bar \varepsilon_k}}(\nabla \log\bar h_k(\bar q^k) \cdot y)e^{U_k}\bar w_k+O(\varepsilon_k^{2+\epsilon_0})\\
	&=\int_{B_{\tau/\bar \varepsilon_k}}\bar h_k(\bar q^k)e^{U_k}\bar w_k+
	\bar \varepsilon_k\int_{B_{\tau/\bar \varepsilon_k}}(\nabla \log\bar h_k(\bar q^k) \cdot y)
	\frac{e^{U_k}(\bar b_k \cdot y)}{1+\frac{\bar h_k(\bar q^k)}{8}|y|^2}+O(\varepsilon_k^{2+\epsilon_0})\\
	&=O(\varepsilon_k^{2+\epsilon_0}),
\end{align*}
where $\bar b_k=(\bar b^k_1, \bar b^k_2)$ and we used the same argument as in \eqref{i1a} to show that
$\int_{B_{\tau/\bar \varepsilon_k}}\bar h_k(\bar q^k)e^{U_k}\bar w^k_{0,\xi}$ is of order
$O(\varepsilon_k^{2+\epsilon_0})$.
Therefore we eventually deduce that,
\begin{equation}\label{small-q}
	\int_{\Omega(q,\tau)}\rho_kH c_k(x)\xi_k=O(\varepsilon_k^{2+\epsilon_0}).
\end{equation}
In view of \eqref{p1-leading}, \eqref{p2-minor}, \eqref{small-q}, we come up with
a contradiction as follows,
\begin{align*}
	0=\int_M\rho_k H c_k\xi_k
	&=\int_{\Omega(p_1,\tau)}H c_k\xi_k+\int_{\Omega(p_2,\tau)}H c_k\xi_k+\int_{\Omega(q,\tau)}H c_k\xi_k\\
	&\quad+\int_{M\setminus \{\Omega(p_1,\tau)\cup \Omega(p_2,\tau)\cup \Omega(q,\tau)\}}H c_k\xi_k\\
	&=Cb_0\Delta \log h_k(p_1)\varepsilon_k^2+o(\varepsilon_k^2),
\end{align*}
for some constant $C\neq 0$, since in particular the integrals on $\Omega(p_2,\tau)$, $\Omega(q,\tau)$
and $M\setminus \{\Omega(p_1,\tau)\cup \Omega(p_2,\tau)\cup \Omega(q,\tau)\}$ are all of order
$o(\varepsilon_k^2)$.
Since for $k$ large $L(\mathbf{p})\neq 0$ is the same as $\Delta \log h_k(p_1)\neq 0$,
we obtain a contradiction as far as $b_0\neq 0$.

\subsection{Proof of $b_1=b_2=0$.}
In this subsection, we shall prove that $b_1=b_2=0$. At the beginning of this part we recall that in local coordinates around $q,~0=x(q)$, $u_{k}$ satisfy
\begin{equation*}
	\Delta u_{k}+\bar h_k(x)e^{u_{k}}=0\quad \mbox{in}\quad B_\tau,
\end{equation*}
and $\xi_k$ satisfies
\begin{equation*}
	\Delta\xi_k+\bar h_k\bar c_k\xi_k=0\quad \mbox{in}\quad B_\tau.	
\end{equation*}
We shall divide our argument into several steps as below, first, we would like to get some estimation on $\bar v_k:$
\medskip

\noindent{\bf Step one: Intermediate estimates for $\bar v_k$}
\medskip

We work in local coordinates centered at $q$, $0=x(q)$, with $\bar \varepsilon_k=e^{-\bar \lambda^k/2}$. As before, by a slightly abuse of notation we set $\psi_k$ to be the function which encodes the boundary oscillation of $u_k$ in $B_\tau(q)$. We set
\begin{equation}\label{bar-v-2-q}
	\bar v_k(y)=u_{k}(\bar q^k+\bar \varepsilon_ky)+2\log \bar \varepsilon_k-\psi_k(\bar\e_ky),\quad  y\in  \Omega_k:=B_{\tau/\bar\e_k},
\end{equation}
and
\begin{equation}\label{xi-k-q}
	\bar \xi_k(y)=\xi_k(\bar q^k+\bar \varepsilon_ky),\quad y \in \Omega_k.
\end{equation}
Obviously, $\bar v_k$ is constant on $\partial \Omega_k$. By Theorem \ref{reg-int} we have
\begin{equation}\label{v2-exp-new2}
	\bar v_k(y)=U_k+c_{0,k}+c_{2,k}+O(\bar \varepsilon_k^{2+\delta})(1+|y|)^{\delta},
\end{equation}
then we can write
\begin{equation*}
	\bar \varepsilon_k^2 \rho_kH \bar c_k(\bar q^k+\bar \varepsilon_ky)\xi_k=\bar h_{k,0}(\bar q^k+\bar \varepsilon_ky)e^{\bar v_k}\bar\xi_k,
\end{equation*}
where $\bar h_{k,0}=\bar h_ke^{\psi_k}$.  According to our setting, $\bar \lambda^k=u_{k}(\bar q^k)$ and $\bar \varepsilon_k=e^{-\bar \lambda^k/2}$. Now we set
$$\bar w^k(y)=\bar v_k(y)-c_{0,k}(y)-c_{2,k}(y).$$
where $c_{i,k}$ $i=0,1,2$, are defined as in the proof of Theorem \ref{reg-int}. Different from the expansion of $v_k$ in the neighborhood of $p_1$, the term $c_{1,k}$ can be regarded as error term, this is due to that $\nabla \bar h_{k,0}(\bar q^k)=O(\bar \lambda^ke^{-\bar \lambda^k})$. 
After direct calculation one can find that $\bar w^k$ satisfies,
\begin{align*}
	\begin{cases}
		\Delta \bar w^k+\bar h_{k,0}(\bar q^k)e^{\bar v_k}\bar w^k=O(\bar \varepsilon_k^{3-\epsilon_0})(1+|y|)^{-1},\\
		\\
		\bar w^k(0)=|\nabla \bar w^k(0)|=0,
	\end{cases}
\end{align*}
and the oscillation of $\bar w^k$ on $\partial \Omega_k$ is of order $O(\bar \varepsilon_k^2)$. By Theorem \ref{reg-int} we get
$$|\bar w^k(y)|\le C\bar\varepsilon_k^{2+\epsilon_0}(1+|y|)^{2\epsilon_0},$$
which implies
\begin{equation}\label{for-v2-q}
	|\bar v_k-U_k-c_{0,k}-c_{2,k}|\le C\bar \varepsilon_k^{2+\epsilon_0}(1+|y|)^{2\epsilon_0},\quad y\in \Omega_k.
\end{equation}
As a consequence,
\begin{equation}\label{grad-v2-new-2}\nabla \bar v_k=\nabla U_k+\nabla c_{0,k}+\nabla c_{2,k}+O(\tau)\bar\varepsilon_k^3,\quad y\in \Omega_k, \quad |y|\sim \bar\varepsilon_k^{-1}.
\end{equation}
Based on (\ref{v2-exp-new2}), we also have
\begin{equation}
	\label{v2-exp-new}
	\begin{aligned}
		\bar v_k(y)+\log \frac{\bar h_{k,0}(\bar q^k+\bar \varepsilon_ky)}{\bar h_{k,0}(\bar q^k)}
		=~&U_k+c_{0,k}+\tilde c_{2,k}+\frac{\Delta (\log \bar h_{k,0})(\bar q^k)\bar \varepsilon_k^2r^2}{4}\\
		&+O(\bar \varepsilon_k^3)r^3+O(\bar \varepsilon_k^{2+\delta})(1+r)^{\delta}.
	\end{aligned}
\end{equation}
\medskip

\noindent{\bf Step two: Improved estimate on the oscillation of $\xi_k$ away from blow-up points and vanishing rates of $b_0^k$, $\hat b_0^k$ and $\bar b_0^k$.}
From Lemma \ref{osci-xi-better}  we have
\begin{equation}\label{new-osci}
	|\xi_k(x_1)-\xi_k(x_2)|\le C(\varepsilon_k^2+\bar \varepsilon_k).
\end{equation}
The new estimate (\ref{new-osci}) implies that the harmonic function $\psi_{\xi}^k$ that encodes the oscillation of
$\xi_k$ on $\partial \Omega(p_1,\tau)$ satisfies
\begin{equation}\label{new-psi-osc}
	|\psi_{\xi}^k(\varepsilon_ky)|\le C(\varepsilon_k^3+\varepsilon_k\bar \varepsilon_k)|y|,\quad |y|<\tau/\varepsilon_k.
\end{equation}
In the following lemma, we shall prove that the oscillation of $\xi_k$ far away from blow-up points is of order $O(\bar\varepsilon_k)$.

\begin{prop}\label{small-osc-psi}
	For any $x_1,x_2\in M\setminus \{\Omega(p_1,\tau)\cup \Omega(p_2,\tau)\cup\Omega(q,\tau)\}$,
	$$|\xi_k(x_1)-\xi_k(x_2)|\le C\bar \varepsilon_k$$
	for some $C>0$ independent by $k$.
\end{prop}

\begin{rem}
	Since the conclusion holds automatically by Lemma \ref{osci-xi-better}, it suffices to study the case for $\alpha>1$. 
\end{rem}

\begin{proof} As we can see that, Proposition \eqref{small-osc-psi} can be seen as an improvement of Lemma \ref{osci-xi-better}. The crucial point is to get a better estimate for $\psi_\xi^k$ around $p_1$. By \eqref{new-psi-osc} and the fact that $\psi_\xi^k(0)=0$ we have 
	$$\psi_{\xi}^k(\varepsilon_ky)=\sum_ja_j(\varepsilon_k^3+\varepsilon_k\bar\varepsilon_k)y_j+O(\varepsilon_k^4+\varepsilon_k^2\bar\varepsilon_k)|y|^2,\quad |y|<\tau/\varepsilon_k,$$
	for suitable $a_j,j=1,2$.
	In local coordinates around $p_1$, $0=x(p_1)$, the first term of the approximation of $\xi_k$ around $p_1$ is still $w_{0,\xi}^k$ defined as in (\ref{w0kxik}) that satisfies (\ref{e-w0kxi}).
	Using the expansion of $v_k$ around $p_1$ (\ref{p-rough-2}), we write
	\begin{align*}
		&|y|^{2\alpha} \tilde h_k(0)\mbox{exp} \bigg
		(v_k+\log\frac{\tilde h_k(\varepsilon_ky)}{\tilde h_k(0)}\bigg )\\
		&=|y|^{2\alpha}\tilde h_k(0)\mbox{exp}\{U_k+c_{0,k}+\tilde c_{1,k}+\tilde c_{2,k} \\
		&\quad+\frac 14\Delta (\log \tilde h_k)(0)|y|^2\varepsilon_k^2+O(\varepsilon_k^{2+\epsilon_0})(1+|y|)^{2+\epsilon_0})\} \\
		&=|y|^{2\alpha} \tilde h_k(0)e^{U_k}(1+c_{0,k}+\tilde c_{1,k}+\tilde c_{2,k}
		+\frac 14\Delta (\log \tilde h_k)(0)|y|^2\varepsilon_k^2 \\
		&\quad+\frac 12(\tilde c_{1,k})^2)+O(\varepsilon_k^{2+\epsilon_0})(1+|y|)^{-2-\epsilon_0-2\alpha},
	\end{align*}
	for some small $\epsilon_0>0$ depending by $\alpha>0$. Then we can write the equation of $w_{0,\xi}^k$ as follows,
	\begin{equation}
		\label{e-hat-xi-new}
		\begin{aligned}
			&\Delta w_{0,\xi}^k+\tilde h_k(0)|y|^{2\alpha}\exp\left (v_k+\log \frac{\tilde h_k(\varepsilon_ky)}{\tilde h_k(0)}\right )w_{0,\xi}^k\\
			&=|y|^{2\alpha} \tilde h_k(0)e^{U_k}w_{0,\xi}^k
			\left(c_{0,k}+\tilde c_{1,k}+\tilde c_{2,k}+\frac{\Delta (\log  \tilde h_k)(0)|y|^2\varepsilon_k^2}{4}+\frac{\tilde c_{1,k}^2}{2}\right) \\
			&\quad+O(\varepsilon_k^{2+\epsilon_0})b_0^k(1+|y|)^{-2-\epsilon_0-2\alpha}.
		\end{aligned}
	\end{equation}
	The next term in the expansion is $w_{1,\xi}^k$ defined as in (\ref{xi-2nd}). Now we need to include $b_0^k$ in the estimate of $w_{1,\xi}^k$,
	$$|w_{1,\xi}^k(y)|\le Cb_0^k \varepsilon_k (1+|y|)^{-1}, $$
	where we used $w_{0,\xi}^k(0)=b_0^k$.
	At this point we write the equation for $w_{1,\xi}^k$ in the following form:\begin{align*}
		&\Delta w_{1,\xi}^k+\tilde h_k(0)|y|^{2\alpha}\exp \bigg (v_k+\log \frac{\tilde h_k(\varepsilon_ky)}{\tilde h_k(0)}\bigg ) w_{1,\xi}^k \\
		&=- \tilde h_k(0)|y|^{2\alpha}e^{U_k}\tilde c_{1,k}w_{0,\xi}^k+|y|^{2\alpha}\tilde h_k(0)e^{U_k}\tilde c_{1,k}w_{1,\xi}^k
		+\frac{O(b_0^k\varepsilon_k^{3})}{(1+|y|)^{3+2\alpha}}.\nonumber
	\end{align*}
	Writing the last two term as $\frac{O(b_0^k\varepsilon_k^2)}{(1+|y|)^{3+2\alpha}}$, then the above equation can be rewritten as
	\begin{equation}
		\label{crude-w1}
		\begin{aligned}
			&\Delta w_{1,\xi}^k+\tilde h_k(0)|y|^{2\alpha}\exp\left({v_k+\log \frac{\tilde h_k(\varepsilon_ky)}{\tilde h_k(0)}}\right)w_{1,\xi}^k \\
			&=- \tilde h_k(0)|y|^{2\alpha}e^{U_k}\tilde c_{1,k}w_{0,\xi}^k+\frac{O(b_0^k\varepsilon_k^2)}{(1+|y|)^{3+2\alpha}}.
		\end{aligned}
	\end{equation}
	Let
	$$\underline{w}_1^k=\tilde \xi_k(y)-w_{0,\xi}^k-w_{1,\xi}^k-\psi_{\xi}^k(\varepsilon_ky),$$
	then from (\ref{e-hat-xi-new}) and (\ref{crude-w1}) we see that the equation for $\underline{w}_1^k$ takes the form,
	\begin{equation}
		\label{f-W1k}
		\begin{aligned}
			&\left(\Delta+|y|^{2\alpha}\tilde h_k(0)\exp\left(v_k+\log \frac{\tilde h_k(\varepsilon_ky)}{\tilde h_k(0)}\right)\right)\underline{w}_1^k \\
			&=\frac{O(\bar \varepsilon_k)}{(1+|y|)^{4+2\alpha}}+\frac{O(b_0^k\varepsilon_k^2)}{(1+|y|)^{3+2\alpha}}+\frac{O(\varepsilon_k^3+\varepsilon_k\bar\varepsilon_k)}{(1+|y|)^{3+2\alpha}}.
		\end{aligned}
	\end{equation}
	The three terms on the right hand side of \eqref{f-W1k} come from the expansion of $c_k$, $w_{1,\xi}^k$ and $\psi_{\xi}^k$ respectively. On the other hand, $w_{0,\xi}^k$ is constant on $\partial\Omega_k$ and $w_{1,\xi}^k$ has an oscillation of order $O(b_0^k\varepsilon_k^2)$. To eliminate this oscillation we use another harmonic function of order
	$O(b_0^k\varepsilon_k^3|y|)$. For simplicity and with an abuse of notations we include this function in $\psi_{\xi}^k$.
	As a consequence standard potential estimates show that,
	\begin{equation}\label{int-w1}
		|\underline{w}_1^k(y)|\le C(\delta)(b_0^k\varepsilon_k^2+\varepsilon_k^3+\bar \varepsilon_k)(1+|y|)^{\delta}.
	\end{equation}
	In order to further improve the estimate about $\psi_{\xi}^k$ we point out that, in view of (\ref{new-osci}), then around $p_2$ and $q$ we have almost the same improved estimates.
	First of all, around $p_2$, we recall that $\hat w_k$ is defined in (\ref{hat-wk}) and satisfies (\ref{sec-e}). Then we have,
	$$
	\int_{B(p_2,\tau)}\rho_kHe^{{\rm w}_k}\xi_k=O(\bar \varepsilon_k).
	$$
	Moreover, around $q^k$, we recall that $\bar w_k$ is defined in (\ref{bar-w2k}) and satisfies (\ref{pre-xi}). By using (\ref{pre-xi}) and the expansion of $\bar w_k$ we have
	\begin{equation*}
		\int_{B(q^k,\tau)}\rho_k He^{{\rm w}_k}\xi_k=O(\bar \varepsilon_k).
	\end{equation*}
	At this point we move back to the evaluation of the oscillation of $\xi_k$ away from blow-up points. From the Green's representation formula of $\xi_k$, as in the proof of (\ref{osi-vars}), by using (\ref{new-psi-osc}) for $\psi_{\xi}^k$, (\ref{int-w1}) for $\underline{w}_1^k$, (\ref{sec-e}) for $\hat w_k$ and (\ref{pre-xi}) for $\bar w_k$ we can further improve the estimate of $\xi_k$ as follows,
	\begin{equation}\label{more-xi-o}
		|\xi_k(x)-\xi_k(y)|\leq C(b_0^k\varepsilon_k^2+\varepsilon_k^4+\bar \varepsilon_k),
	\end{equation}
	for $x,y\in M\setminus\{\Omega(p_1,\tau)\cup \Omega(p_2,\tau)\cup \Omega(q,\tau)\}$.
	Then (\ref{more-xi-o}) further improves the estimate of $\underline{w}_1^k$, so that, by a bootstrap argument which does not involve the leading term proportional to $b_0^k\varepsilon_k^2$ (which in fact comes just from the integration over $\Omega(p_1,\tau)$), after a finite number of iterations, we deduce that,
	\begin{equation*}
		|\underline{w}_1^k(y)|\le C(\delta)(b_0^k\varepsilon_k^2+\bar \varepsilon_k)(1+|y|)^{\delta},
	\end{equation*}
	and the estimate about the oscillation of $\xi_k$ takes the form,
	\begin{equation}\label{best-osci-2}
		|\xi_k(x_1)-\xi_k(x_2)|\le C(b_0^k\varepsilon_k^2+\bar \varepsilon_k),
	\end{equation}
	for all $x_1,x_2\in M\setminus \{ \Omega(p_1,\tau),\cup  \Omega(p_2,\tau)\cup
	\Omega(q,\tau)\}$.
	
	As an immediately consequence of (\ref{best-osci-2}) for $\psi_{\xi}^k$ we have
	\begin{equation}\label{osc-p-1}
		|\psi_{\xi}^k(\varepsilon_ky)|\leq C(b_0^k\varepsilon_k^3+\varepsilon_k\bar \varepsilon_k)|y|.
	\end{equation}
	In the following lemma, we shall prove that
	$$|b_0^k|\le C e^{-\frac{\lambda^k}2+\frac{\lambda^k}{1+\alpha}}\quad \mbox{if}\quad \alpha>1.$$
	As a consequence, we deduce that Proposition \ref{small-osc-psi} holds. \end{proof}

\begin{lem}\label{b0-va}
	There exists $C>0$ independent of $k$ such that
	\begin{equation}\label{small-d}
		|b_0^k|\le C e^{-\frac{\lambda^k}2+\frac{\lambda^k}{1+\alpha}}\quad \mbox{if}\quad \alpha>1.
	\end{equation}
\end{lem}
\begin{proof}  Recall that $w_{2,\xi}^k$ was defined in \eqref{first:estimate} and the estimate of this term is given in (\ref{1st-w2k}), but now we include the dependence of $b_0^k$. So (\ref{1st-w2k}) takes the form,
	\begin{equation*}
		|w_{2,\xi}^k(y)|\le Cb_0^k\varepsilon_k^2(1+|y|)^{-2\alpha}.
	\end{equation*}
	In order to take care of the radial part of order $O(\varepsilon_k^2)b_0^k$ in the expansion of $\xi_k$, we set $z_{0,k}$ to be the radial function satisfying
	\begin{align*}
		\begin{cases}
			\frac{d^2}{dr^2}z_{0,k}+\frac 1r\frac{d}{dr}z_{0,k}+r^{2\alpha} \tilde h_k(0)e^{U_k}z_{0,k}=F_{0,k}, \quad 0<r<\tau/\varepsilon_k,\\
			\\
			z_{0,k}(0)=\frac{d}{dr}z_{0,k}(0)=0,
		\end{cases}
	\end{align*} where
	\begin{align*}
		F_{0,k}(r)
		=~&- \tilde h_k(0)|y|^{2\alpha}e^{U_k}\bigg (w_{0,\xi}^k\big (c_{0,k}+\frac 14\Delta (\log  \tilde h_k)(0)|y|^2\varepsilon_k^2\\
		&+\frac  14\varepsilon_k^2|\nabla \log  \tilde h_k(0)|^2(g_k+r)^2\big )
		+(w_{\xi,1}^k\tilde c_{1,k})_r\bigg ),
	\end{align*}
	and  
	$$\frac 14\varepsilon_k^2|\nabla  \log \tilde h_k(0)|^2(g_k+r)^2=\left(\frac 12(\tilde c_{1,k})^2\right)_r.$$ Observe carefully that we already considered this function,
	see \eqref{eq-z0}, but in that case the underlying assumption was $b_0^k\to b_0\neq 0$. Concerning $z_{0,k}$ we have,
	$$|z_{0,k}(r)|\le C(\delta)b_0^k \varepsilon_k^2(1+r)^{\delta},$$
	and then, by defining $\underline{z}_k$ as follows,
	$$\underline{z}_k(y)=\xi_k(\varepsilon_ky)-\psi_{\xi}^k(\varepsilon_k y)-w_{0,\xi}^k(y)-w_{1,\xi}^k(y)-w_{2,\xi}^k(y)-z_{0,k}(y),$$
	we have
	$$\Delta \underline{z}_k+|y|^{2\alpha}h_k(\varepsilon_ky)\varepsilon_k^{2+2\alpha}\tilde c_k(\varepsilon_ky)\underline{z}_k=E_{k,1}$$
	with $z_k(0)=0$ and
	$$|E_{k,1}|\le Cb_0^k(\varepsilon_k^{2+\epsilon_0}+\bar \varepsilon_k)(1+|y|)^{-2-\frac{\epsilon_0}2},$$
	for some $\epsilon_0>0$. By standard potential estimates we have,
	\begin{equation*}
		|\underline{z}_k(y)|\le C(\delta)(b_0^k\varepsilon_k^{2+\epsilon_0}+\bar \varepsilon_k)(1+|y|)^{\delta}.
	\end{equation*}
	At this point, by using the estimate about $\underline{z}_k$, we deduce that,
	\begin{align*}
		\int_{\Omega(p_1,\tau)}\rho_kHc_k\xi_k
		&=\int_{\Omega_k}\tilde h_k(\varepsilon_ky)|y|^{2\alpha}e^{v_k}(\xi_k(\varepsilon_ky)-\psi_{\xi}^k(\varepsilon_ky))dy\\
		&\quad+\int_{\Omega_k}\tilde h_k(\varepsilon_ky)|y|^{2\alpha}e^{U_k}\psi_{\xi}^k(\varepsilon_ky)\\
		&=cb_0^k(\Delta \log h_k(0)+O(\varepsilon_k^{\epsilon_0}))\varepsilon_k^2+O(\bar \varepsilon_k),
	\end{align*}
	for some suitable constant $c$. Here the derivation of the above is similar to the derivation of (\ref{p1-leading}) except that we keep track of $b_0^k$ in this estimate.
	Therefore the integral around $p_1$ reads,
	$$\int_{\Omega(p_1,\tau)}\rho_kH c_k\xi_k=c b_0^k\varepsilon_k^2(\Delta (\log \tilde h_k)(0)+
	O(\varepsilon_k^{\epsilon_0}))+O(\bar \varepsilon_k),\quad c\neq 0.$$
	From $\int_M\rho_kHc_k\xi_k=0$ we readily obtain (\ref{small-d}) by splitting the domain of integration in the three regions near the blow-up points
	plus the corresponding complement in $M$. Lemma \ref{b0-va} is established. \end{proof}
\smallskip

With Proposition \ref{small-osc-psi}, we are able to get an estimate on $\bar b_0^k$ and $\hat b_0^k$ as well.
\begin{lem}\label{same-b} There exists $C>0$ independent by $k$, such that
	$$|\bar b_0^k|+|\hat b_0^k|\le C\bar \varepsilon_k\varepsilon_k^{-2}=Ce^{-\frac{\lambda^k}2+\frac{\lambda^k}{1+\alpha}}.$$
\end{lem}

\begin{proof} In the proof of Proposition \ref{small-osc-psi} we already established
	\eqref{small-d} for $b_0^k$. Also the oscillation of $\xi_k$ far away from blow-up points is of order $O(\bar \varepsilon_k)$ (Proposition \ref{small-osc-psi}). In the expansion of $\xi_k$ around $q^k$ we have
	$$\xi_k(x)=-\bar b_0^k+O(\bar \varepsilon_k),\quad x\in \partial B(q^k,\tau).$$
	On $\partial B(p_1^k,\tau)$ we have
	$$\xi_k=-b_0^k+O(b_0^k\varepsilon_k^2)=-b_0^k+O(\bar \varepsilon_k),$$
	and similarly on $\partial B(p_2,\tau)$,
	$$\xi_k(x)=\hat b_0^k+O(\bar \varepsilon_k).$$
	Thus $|b_0^k-\hat b_0^k|+|b_0^k-\bar b_0^k|\le C\bar \varepsilon_k.$
	Lemma \ref{same-b} is established. \end{proof}

\medskip

\noindent{\bf Step three: an improved expansion for $\bar \xi_k$.}

\medskip

As a consequence of Proposition \ref{small-osc-psi}, the fact that  $\nabla \log \bar h_{k,0}(\bar q^k)=O(\bar \lambda^ke^{-\bar\lambda^k})$ and the expansion of $\bar v_k$ in (\ref{v2-exp-new2}), we have
\begin{equation*}
	\Delta \bar \xi_k(y)+\bar h_{k,0}(\bar q^k)e^{U_k}\bar \xi_k=O(\bar \varepsilon_k^{3/2})(1+|y|)^{-5/2}.
\end{equation*}
Then by standard potential estimates, we get
\begin{equation}\label{q-rough}
	|\bar \xi_k(y)-\bar w_{0,\xi}^k|\leq C(\delta)\bar \varepsilon_k (1+|y|)^{\delta}.
\end{equation}
As before, we set $\bar\psi_\xi^k$ is set to be a harmonic function which encodes the bounded oscillation of $\bar\xi_k(y)-\bar w_{0,\xi}^k$ in $B_{\tau/\bar\e_k}$. Using Proposition \ref{small-osc-psi}, we have
$$|\bar\psi_\xi^k(\bar\e_k y)|\leq C\bar\e_k^2|y|.$$
Let
$$z_{0,\xi}^k=\bar \xi_k-\bar w_{0,\xi}^k,$$
we have
\begin{equation}\label{e-z0xi}
	(\Delta +\bar h_{k,0}(\bar q^k)e^{U_k})(z_{0,\xi}^k- \bar\psi^k_{\xi}(\bar\varepsilon_ky))=O(\bar \varepsilon_k^{1+\epsilon_0})(1+|y|)^{-\frac52}\quad\mbox{in}\quad\Omega_k
\end{equation}
for some $\epsilon_0>0$. Note that $z_{0,\xi}^k$ has the same value on the boundary, and
$$(z_{0,\xi}^k(y)-\bar\psi_\xi^k(\bar\e_ky))\mid_{y=0}=0\quad\mbox{and}\quad
\nabla(z_{0,\xi}^k(y)-\bar\psi_\xi^k(\bar\e_ky))\mid_{y=0}=O(\bar\e_k^2).$$
By the standard argument (see \cite[Proposition 3.1]{byz-1}) for equation \eqref{e-z0xi} we deduce that 
\begin{equation}\label{bar-xi-e}
	|z_{0,\xi}^k-\bar \psi_{\xi}^k(\bar \varepsilon_ky)|\leq C(\delta)\bar \varepsilon_k^{1+\epsilon_0}(1+|y|)^{\delta}.
	\quad |y|<\tau/\bar \varepsilon_k,
\end{equation}
for some $\epsilon_0>0$. Then we apply standard elliptic estimates based on (\ref{e-z0xi}) and (\ref{bar-xi-e}) to obtain,
\begin{equation}\label{small-z0xi}
	|\nabla (z_{0,\xi}^k-\bar \psi_{\xi}^k(\bar\varepsilon_k\cdot))|\le C\bar \varepsilon_k^{2+\epsilon_0}\quad \mbox{on}\quad \partial B_{\frac{\tau}{2\bar\varepsilon_k}}
\end{equation}
for some $\varepsilon_0>0$. From (\ref{small-z0xi}) we see that for $r\sim \bar\varepsilon_k^{-1}$, the leading term in $\nabla \xi_k$ is $\bar w_{0,\xi}^k+c_{0,\xi}^k$:
\begin{align}\label{xi-k-large}
	\begin{cases}
		\partial_1\bar \xi_k=\frac{b_1^k(1+a_kr^2)-2a_ky_1(b_1^ky_1+b_2^ky_2)}{(1+a_kr^2)^2}+O(\bar \varepsilon_k^{2+\delta_0}),\\
		\\
		\partial_2\bar \xi_k=\frac{b_2^k(1+a_kr^2)-2a_ky_2(b_2^ky_2+b_1^ky_1)}{(1+a_kr^2)^2}+O(\bar \varepsilon_k^{2+\delta_0}),
	\end{cases}
\end{align}
where $a_k=\frac{\bar h_{k,0}(\bar q^k)}8$ and $\delta_0$ is a small positive constant.

\medskip

\noindent{\bf Step four: Evaluation of a simplified Pohozaev identity}

\medskip

In a local coordinates system around $q$, $0=x(q)$, we have
$$\Delta u_{k}+\bar h_ke^{u_{k}}=0,$$
where
$$\bar h_k=\rho_kh(x)e^{-4\pi\alpha G(x,p_1)-4\pi\beta G(x,p_2)+\chi+f_k}$$
and $f_k$ is any solution satisfying that
$$\Delta f_k=\rho_ke^{\chi}\quad {\rm in}\quad B_\tau,\qquad f_k(0)=0.$$
Here we can choose,
\begin{equation}
	\label{exp-f-3}
	\begin{aligned}
		f_k(x)=~&\frac{\rho_k}{3+\alpha+\beta}\left(R(x,\bar q^k)-R(\bar q^k,\bar q^k)\right)\\
		&+\frac{\rho_k}{3+\alpha+\beta}\left( (1+\alpha)(G(x,p_1)-G(\bar q^k,p_1))\right)\\
		&+\frac{\rho_k}{3+\alpha+\beta}\left((1+\beta)(G(x,p_2)-G(\bar q^k,p_2))\right).
	\end{aligned}
\end{equation}
\begin{rem} {\it
		In the general case the corresponding definitions around a regular blow-up point $q_j$ would be,
		\begin{align*}
			f_{k,j}(x)=~&\frac{\rho_k}{m+\sum\limits_{j=1}^{m_1}\alpha_j}
			\sum\limits_{\ell\neq j, \ell=m_1+1}^m (G(x,\bar q_{\ell}^k)-G(\bar q_{j}^k,\bar q_{\ell}^k))\\
			&+\frac{\rho_k}{m+\sum\limits_{j=1}^{m_1}\alpha_j}\sum\limits_{i=1}^{m_1}(1+\alpha_i)(G(x,p_i)-G(\bar q_{j}^k,p_i))\\
			&+\frac{\rho_k}{m+\sum\limits_{j=1}^{m_1}\alpha_j}\left(R(x,\bar q_{j}^k)-R(\bar q_{j}^k,\bar q_{j}^k)\right)
	\end{align*}}
\end{rem}
In the following we shall apply an easier version of Pohozaev identity to verify that the coefficients $b_1$ and $b_2$ are zero. We first recall that, for the harmonic function $\psi_k$ which was used to cancel out the boundary oscillation of $u_{k}$, by the expression of $f_k$ as in \eqref{exp-f-3}, we have
$$|\psi_k(\bar \varepsilon_ky)|\le C\bar \varepsilon_k\varepsilon_k^2|y|.$$
Then for $j=1,2,$ we use the standard Pohozaev identity for
$$\Delta\bar\xi_k+\bar h_{k,0}e^{\bar v_k}\bar\xi_k=0\quad\mbox{in}\quad \Omega_{k,1}=B_{\frac{\tau}{2\bar\varepsilon_k}}.$$
Multiplying both sides by $\partial_j\bar v_k$ we get
\begin{equation}
	\label{step4.eq-1}
	\begin{aligned}
		&-\int_{\partial\Omega_{k,1}}\partial_j\bar v_k\partial_{\nu}\bar\xi_k+\int_{\Omega_{k,1}}\partial_j(\nabla\bar v_k)\nabla\xi_k\\
		&=\int_{\partial\Omega_{k,1}}\bar h_{k,0}(\bar q^k+\bar \varepsilon_ky)e^{\bar v_k}\nu_jdS-\bar \varepsilon_k\int_{\Omega_{k,1}}\partial_j\bar h_{k,0}(\bar q^k+\bar \varepsilon_ky)e^{\bar v_k}\\
		&\quad-\int_{\Omega_{k,1}}\bar h_{k,0}(\bar q^k+\bar \varepsilon_ky)e^{\bar v_k}\partial_j\bar\xi_k,
	\end{aligned}
\end{equation}
where $\nu_j$ is the $j$-th direction of outer normal vector on $\partial\Omega_{k,1}$. By the equation
$$\Delta\bar v_k+\bar h_{k,0}(\bar q^k+\bar \varepsilon_ky)e^{\bar v_k}=0,$$
we can rewrite the last term on the right hand side of \eqref{step4.eq-1} as $\int_{\Omega_{k,1}}\Delta\bar v_k\partial_j\bar\xi_k$. Together with the integration by parts we get that
\begin{equation}
	\label{impor-2}
	\begin{aligned}
		&\frac{1}{\bar \varepsilon_k}\int_{\partial \Omega_{k,1}}(-\partial_{\nu}\bar v_k\partial_j\bar \xi_k-\partial_{\nu}\bar \xi_k\partial_j\bar v_k+(\nabla \bar v_k\cdot \nabla \bar \xi_k)\nu_j)dS\\
		&=-\int_{\Omega_{k,1}}\partial_j\bar h_{k,0}(\bar q^k+\bar\varepsilon_ky)e^{\bar v_k}\bar \xi_k+\frac{1}{\bar \varepsilon_k}\int_{\partial\Omega_{k,1}}\bar h_{k,0}(\bar q^k+\bar \varepsilon_ky)e^{\bar v_k}\nu_jdS.
	\end{aligned}
\end{equation}
It is not difficult to see that the last term can be bounded by $\bar\varepsilon_k^{1+\delta_0}$. Concerning $\nabla \bar v_k$ we use
\begin{equation*}
	\nabla \bar v_k(y)=-(4+O(\lambda^ke^{-\lambda^k}))\left(\frac{y_1}{r^2},\frac{y_2}{r^2}\right),\quad |y|= \tau/(2\bar \varepsilon_k).
\end{equation*}
Then (\ref{impor-2}) is reduced to
\begin{equation}\label{pi-final}
	\frac{1}{\bar \varepsilon_k}\int_{\partial \Omega_{k,1}}\frac{4}{r}\partial_j\bar \xi_k=-\int_{\Omega_{k,1}}\partial_j\bar h_{k,0}(\bar q^k+\bar\varepsilon_ky)e^{\bar v_k}\bar \xi_k+O(\bar \varepsilon_k^{1+\epsilon_0}).
\end{equation}
To evaluate the left hand side of (\ref{pi-final})  we use
(\ref{xi-k-large}) to deduce that,
$$\frac{1}{\bar \varepsilon_k}\int_{\partial \Omega_{k,1}}\frac{4}{r}\partial_j\bar \xi_k=O(\bar \varepsilon_k^{1+\epsilon_0}).$$
To evaluate the right hand side of (\ref{pi-final}), we use,
$$\partial_j\bar h_{k,0}(\bar q^k+\bar \varepsilon_ky)=\partial_j\bar h_{k,0}(\bar q^k)+\sum_{s=1}^2\partial_{js}\bar h_{k,0}(\bar q^k)\bar \varepsilon_ky_s+O(\bar \varepsilon_k^2)|y|^2,$$
$$e^{\bar v_k}=\frac{1}{(1+a_k|y|^2)^2}(1+O((\log \bar \varepsilon_k)^2\bar \varepsilon_k^2))$$
and (\ref{bar-xi-e}) for $\bar \xi_k$. Then by a straightforward evaluation we have,
$$\bar \varepsilon_k\left(\begin{matrix}
	\partial_{11}\bar h_{k,0}(\bar q_2^k) & \partial_{12}\bar h_{k,0}(\bar q_2^k) \\
	\partial_{12}\bar  h_{k,0}(\bar q_2^k) & \partial_{22}\bar h_{k,0}(\bar q_2^k)
\end{matrix}
\right)\left(\begin{matrix}
	b_1^k\\
	b_2^k
\end{matrix}
\right)=O(\bar \varepsilon_k^{1+\epsilon_0}).$$
Since the underlying assumption is $(b_1^k,b_2^k)\to (b_1,b_2)\neq (0,0)$
and since the non-degeneracy condition is ${\rm det}(D^2\bar h_k(0))\neq 0$, then we readily obtain a contradiction. Up to now, we have proved Theorem \ref{main-theorem-2} under the assumption that $\alpha>1$ and three different type of blow-up points consisting of one positive singular source, one negative singular source and one regular point. When $m_1=m$, the situation is easier and  we omit the details to avoid repetitions. The more general case can be addressed with minor adjustments, primarily involving changes in notation, which completes the proof of Theorem \ref{main-theorem-2} as far as $\alpha>1$.	While for $\alpha\in(0,1)$, we observe that
$$\varepsilon_k^2=e^{-\lambda^k/{(1+\alpha)}}=o(e^{-\lambda^k/2}), $$
then the proof of $b_1=b_2=0$ around each regular blow-up point follows as in \cite{bart-5}. While the general case can be argued similarly. Hence we finish the whole proof. $\hfill\Box$

\begin{proof}[Proof of Theorem \ref{main-theorem-2}.]
	Once we have shown that all the coefficients $b_0$, $b_1$ and $b_2$ are zero, then we can follow the arguments in \cite{bart-5} to show that $\xi_k\equiv0$. Hence we finish the proof.	
\end{proof}

\begin{rem}
	
	The proof of Theorem \ref{mainly-case-2} follows by a similar approach. The key difference is that if $L({\bf p}) = 0$, we use $D({\bf p}) \neq 0 $ to show that the projection on the radial kernel vanishes. If all blow-up points are negative singular points, there is no projection on the translation kernels, allowing us to omit the assumption ${\rm det}(D^2f^*(p_{m_1+1}, \cdots, p_m)) \neq 0 $. The proof of the non-degeneracy result for the Dirichlet problem \eqref{r-equ-flat} (Theorems \ref{main-theorem-4} and \ref{main-theorem-1}) closely mirrors the process used for the case on Riemann surfaces, so we omit the details.
\end{rem}

\medskip

\section{Mean field theory of the vortex model with singular sources}\label{sec:meansink}
Let $\Omega\subset \mathbb{R}^2$ be a bounded domain of class $C^2$, we define the mean field theory (\cite{caglioti-2})
of the vortex model with singular sources. Let $\mathcal{P}(\Omega)$ be the space of vorticity densities,
\begin{align}
	\mathcal{P}(\Omega) =
	\left\{\omega \in L^1(\Omega)\mid \,\omega \geq 0\;  a.e., \;  \int_\Omega \omega {\rm d}x =1, \; \int_{\Omega}\omega \log(\omega)<+\infty\right\},
\end{align}
by definition the energy of a fixed $\omega\in\mathcal{P}(\Omega)$ is,
\begin{align*}
	\mathcal{E}(\omega)=
	\frac{1}{2}\int_\Omega \omega(x) G [\omega](x) {\rm d}x,\quad\mbox{where}\quad
	G[\omega](x) = \int_\Omega G(x,y) \omega(y) {\rm d}y.
\end{align*}
Let us denote by $\psi=G[\omega]$ the stream function, then $\psi\in W^{1,q}_0(\Omega)$, $q\in[1,2)$ is the unique solution (\cite{stam}) of
\begin{align}
	\begin{cases}
		-\Delta \psi = \omega & \mbox{in} \;\  \Omega, \\
		\psi=0 & \mbox{on} \;\  \partial \Omega.
	\end{cases}
\end{align}
Actually, since $\int_{\Omega}\omega \log(\omega)$ is bounded, it is well known (\cite{stam2})
that $\psi\in W^{1,2}_0(\Omega)\cap L^{\infty}(\Omega)$ and the energy of the density $\omega$ is readily seen to be the Dirichlet energy of~$\psi$:
\begin{align*}
	\frac{1}{2}\int_\Omega |\nabla \psi|^2 {\rm d}x
	=\frac{1}{2}\int_{\Omega} (-\Delta \psi) \psi {\rm d}x
	=\frac{1}{2}\int_\Omega \omega G[\omega] {\rm d}x
	=\mathcal{E}(\omega).
\end{align*}
Let us define
$$
H(x)=e^{-4\pi\beta G_{\Omega}(x,0)}\prod_{j=1}^N e^{-4\pi\alpha_j G_{\Omega}(x,q_j)}=|x|^{2\beta}h_0(x),
$$
where,
$$
h_0(x)=e^{-4\pi\beta R_{\Omega}(x,0)}\prod_{j=1}^N e^{-4\pi\alpha_j G_{\Omega}(x,q_j)},
$$
and we assume that,
\begin{equation}\label{sinks}
	x=0\in \Omega,\quad \beta\in (-1,0),\quad \alpha_j\in (0,+\infty),\; j\in \{1,\cdots,N\}.
\end{equation}
Closely related to the mean field theory we have the so called Mean Field Equation,
\begin{equation}\label{r-equ-flat-lm}
	\begin{cases}
		\Delta  {\rm w}+\rho \dfrac{He^{\rm w}}{\int_{\Omega} H e^{\rm w}}=0  &{\rm in} \;\ \Omega
		\\ \\
		{\rm w}=0  &{\rm on} \;\ \partial\Omega,
	\end{cases}
\end{equation}
which is just the Euler-Lagrange equation relative to the natural variational principles for the vortex model, see
\cite{caglioti-1}, \cite{caglioti-2} for the regular case and the discussion below for the case with sinks. Here
$$
\omega_{\rho}=\dfrac{He^{{\rm w}_{\rho}}}{\int_{\Omega} H e^{{\rm w}_{\rho}}},
$$
is the vorticity density corresponding to a solution ${\rm w}_{\rho}=\rho G_{\Omega}[\omega_{\rho}]$. The singularity at the origin describes a vortex of total vorticity $4\pi\rho^{-1}|\beta|$, whence co-rotating with the flow, i.e. of the same sign of $\omega_\rho$. The remaining $N$ singularities describe counter-rotating vortices of total vorticity $-4\pi\rho^{-1}\alpha_j$ each. See Remark \ref{rem.1} below for more details about this point.\\

A subtle problem in the analysis of \eqref{r-equ-flat-lm} is the existence/non existence of solutions for $\rho=8\pi(1+\beta)$, which
is critical with respect to the singular (\cite{ads}) Moser-Trudinger (\cite{moser}) inequality, see Theorem \ref{thm.1} and Section
\ref{section:firstsec} below. Interestingly enough, this problem is strictly related to that of
the description of the thermodynamic equilibrium in the mean field theory of the vortex model and it has been solved in the regular case,
see \cite{caglioti-1, caglioti-2}, \cite{chang-chen-lin},
\cite{BLin3}, \cite{BMal2} showing that the geometry of the domain plays a crucial role. See \cite{BLin2} for the case of positive singular sources.
In the regular case or either if the singular sources are positive,
a crucial point for the characterization of the existence/non existence of solutions in terms of the geometry of $\Omega$
is the fact that the critical threshold is $8\pi$ and that solutions blowing-up as $\rho\to 8\pi$
make 1-point blow-up where the concentration point is a maximum point of the Robin function of $\Omega$,
see \cite{caglioti-1}, \cite{caglioti-2},\cite{chang-chen-lin}, \cite{BLin3}, \cite{BLin2}. Both these features are drastically modified when adding negative singular sources.
In particular, in the model case which we pursue here, where we have only one negative singular source with strength $\beta\in (-1,0)$, the
critical threshold is $8\pi(1+\beta)$ and solutions blowing-up as $\rho\to 8\pi(1+\beta)$
make one point blow-up, but the concentration point is the singular source, {which in principle need not be located at a maximum point of
	the Robin function}. Therefore our results about mean field theory with sinks are two-fold. Indeed, on
one side we extend the results about equivalence of the statistical ensembles and about the non concavity of the entropy
(\cite{caglioti-2}, \cite{bart-5}) on the other side we extend the characterization of the existence/non existence of solutions
(\cite{chang-chen-lin}, \cite{BLin3}, \cite{BLin2})
of \eqref{r-equ-flat-lm} to the case $\rho=8\pi(1+\beta)$. Compared to \cite{chang-chen-lin}, \cite{BLin3}, \cite{BLin2}, among other things we have to prove that entropy maximizers have to concentrate on the singular source in the limit of large energies. We will not go into details of those proof which can be worked out by a standard adaptation of known arguments in \cite{caglioti-2}, \cite{chang-chen-lin} and \cite{bart-5}. On the other side, the characterization (Theorem \ref{thm:5.2}) of the so called pairs $(\Omega, \beta)$ of first/second kind (see Definition \ref{def:first}) is more subtle. Another point which requires some care arises in the proof of the existence of negative specific heat states (Theorem \ref{convex:entropy}) where we also use some recent results in \cite{wyang}.

\subsection{The Canonical Variational Principle}
To simplify the exposition, with an abuse of terminology (see Remark \ref{rem.1}) we define the Free Energy functional as follows,
\begin{align*}
	\mathcal{F}_{\rho}(\omega)=\int_\Omega \omega(x)(\log\omega(x)-\log H(x)){\rm d}x
	-\frac{\rho}{2}\int_\Omega \omega(x)  G [\omega](x){\rm d}x, \quad 
\end{align*}
where $\omega\in  \mathcal{P}(\Omega)$,
and since we are interested in negative temperature states (\cite{caglioti-2,On}), modulo some minor necessary exceptions, we will consider only the case $\rho>0$. Indeed, in the statistical mechanics
formulation (\cite{caglioti-1, caglioti-2}) we would have $-\frac{1}{\kappa T_{\rm stat}}=\rho$ where $T_{\rm stat}$
is the statistical Temperature and
$\kappa$ the Boltzmann constant.\\
Due to the inequality $ab\leq e^a+b(\log(b))$, $a\geq 0,b\geq 1$, it is readily seen that
$\mathcal{F}_{\rho}(\omega)$ is well defined on $\mathcal{P}(\Omega)$. The thermodynamic equilibrium states are by
definition those vorticity densities $\omega$ which
solve the Canonical Variational Principle (({\rm \bf CVP}) for short),
$$
f(\rho)=\inf\{\mathcal{F}_{\rho}(\omega)\mid \omega \in \mathcal{P}(\Omega)\}\eqno {\rm \bf (CVP)}
$$

\begin{rem}\label{rem.1} According to standard conventions {\rm(}\cite{caglioti-2}{\rm)} the physical
	free energy functional should be $-\rho^{-1}\mathcal{F}_{\rho}(\omega)$ and the free energy
	thus would take the form $-\rho^{-1}f(\rho)$. As mentioned above, $\rho=-\frac{1}{\kappa T_{\rm stat}}$ where $\kappa$ is the Boltzmann
	constant and $T_{\rm stat}$
	the statistical temperature.
	
	Actually, in this mean field model, the flow of positive vorticity $\omega$
	is interacting with $N+1$ fixed vortices whose total vorticities are $4\pi|\beta|/\rho$, $-4\pi\alpha_j/\rho$ respectively. This particular
	formulation has the advantage
	that the Euler-Lagrange equations take the form \eqref{r-equ-flat}, whose solutions are widely used to describe
	vortex equilibria of stationary Euler equations in vorticity form, see \cite{CCKW,TuY} and references quoted therein.
	
	However, the more realistic model where the vorticities of the singular sources are \underline{fixed} (i.e. they do not depend by $\rho$) should be defined
	via the standard entropy $-\int_\Omega \omega(x)\log\omega(x){\rm d}x$, while the energy should contain the contribution due to the fixed vortices.
	In other words, we should use the same $\mathcal{F}_{\rho}(\omega)$ as above where $\int_\Omega \omega\log(\omega/ H)$
	should be replaced by $\int_\Omega \omega\log \omega$ while the energy part would be
	$\frac{\rho}{2}\int_{\Omega} \omega G[\omega]-\rho \int_{\Omega}\omega \log(H)$.
	This formulation in turn would yield an Euler-Lagrange equation as
	\eqref{r-equ-flat} where $\beta$ and $\alpha_j$ would be replaced by $\rho \beta$ and $\rho \alpha_j$.
	On the other side, as far as the
	blow-up at singular sources is concerned, some care is needed in the generalization to these equations of known results,
	such as concentration-compactness-quantization {\rm(}\cite{BM3},\cite{BT}{\rm)} and in particular pointwise estimates {\rm(}\cite{BT-2}{\rm)},
	{\rm(}\cite{zhang2}{\rm)}, local uniqueness {\rm(}\cite{bart-4-2},\cite{byz-1},\cite{wu-zhang-ccm}{\rm)} and asymptotic non degeneracy (see Theorem \ref{main-theorem-1} above)
	which will be discussed elsewhere {\rm(}\cite{BCWYZ}{\rm)}.
\end{rem}

We will often use the following consequences of recent results in \cite{bj-lin} and \cite{Bons}.

\begin{lem}\label{lem:ift} 
	Let ${\rm w}_0$ be a solution of \eqref{r-equ-flat-lm} for some $\rho=\rho_0$ and assume that the corresponding linearized equation admits
	only the trivial solution. For any $p\in (1,|\beta|^{-1})$ there exist open neighborhoods $I_0\subset \mathbb{R}$ of $\rho_0$ and
	$B_0 \subset W^{2,p}_0(\Omega)$ of ${\rm w}_0$ such that solutions of \eqref{r-equ-flat-lm} in $I_0\times B_0$ describe a real analytic curve $\mathcal{G}_0$
	of the form $I_0\ni \rho \mapsto {\rm w}_{\rho}\in B_0$.
\end{lem}
\begin{proof}
	Based on the real analytic implicit function theorem (\cite{but}), this is a straightforward generalization of Lemma 2.4 in \cite{Bons}.    
\end{proof}

\begin{lem}\label{lem:spectral}$\,$\\
	$({\rm i})$ {\rm(\cite{bj-lin})} For any $\rho\in [0,8\pi(1+\beta))$ there exists a unique solution ${\rm w}_{\rho}$ of
	\eqref{r-equ-flat-lm};\\
	$({\rm ii})${\rm(\cite{bj-lin})} For any $\rho\in [0,8\pi(1+\beta))$, the linearized operator of
	\eqref{r-equ-flat-lm} evaluated at its unique solution has strictly positive first eigenvalue;\\
	$({\rm iii})$ If a solution of \eqref{r-equ-flat-lm} exists for $\rho=8\pi(1+\beta)$ then it is unique and the linearized operator
	evaluated at this solution has strictly positive first eigenvalue;\\
	$({\rm iv})$ The map $[0,8\pi(1+\beta))\ni \rho\mapsto {\rm w}_{\rho}\in W^{2,p}_0(\Omega)$, $p\in (1,|\beta|^{-1})$ is real analytic;\\
	$({\rm v})$ For any $\rho\in [0,8\pi(1+\beta))$ the map $\rho \mapsto E(\rho)$, where
	$$
	\rho \mapsto E(\rho)=\mathcal{E}(\omega_\rho),\quad \omega_{\rho}=\dfrac{He^{{\rm w}_{\rho}}}{\int_{\Omega} H e^{{\rm w}_{\rho}}},
	$$
	is real analytic and $\frac{dE(\rho)}{d\rho}>0$. In particular $\omega_0(x)=\frac{H(x)}{\int_{\Omega} H}$;\\
	$({\rm vi})$ If a solution of \eqref{r-equ-flat-lm} exists for $\rho=8\pi(1+\beta)$ then the curve of unique solutions for $\rho\in [0,8\pi(1+\beta)]$
	can be continued as a real analytic curve in a right neighborhood of $8\pi(1+\beta)$.
\end{lem}
\begin{proof} By definition the map in $({\rm iv})$ is said to be real analytic at $\rho=0$ if it can be extended to a real
	analytic map in a neighborhood of $\rho=0$.
	The claim in $({\rm i})$, $({\rm ii})$ and $({\rm iii})$ have been recently proved in \cite{bj-lin}.
	Based on $({\rm i})$, $({\rm ii})$, $({\rm iii})$ and Lemma \ref{lem:ift},
	the proof of $({\rm iv})$, $({\rm v})$ and 
	$({\rm vi})$ follow from a rather standard generalization of recent results in \cite{Bons}.
\end{proof}

\begin{rem}\label{rem:spectral}
	Actually, due to the strict convexity of the corresponding variational functional,
	the uniqueness of the solution of \eqref{r-equ-flat-lm} holds as well for fixed $\rho< 0$, in which case the positivity of the first
	eigenvalue of the linearized problem relative to \eqref{r-equ-flat-lm} is trivial.
	As a consequence, in view of Lemma \ref{lem:spectral}, it is not too difficult to see
	that the map $\rho \mapsto E(\rho)$, is well defined and smooth and $\frac{d E(\rho)}{d\rho}>0$ for any
	$\rho \in (-\infty,8\pi(1+\beta))$.
\end{rem}

Before stating the first theorem of this section, we need the following lemma, which is well known in the regular case, see \cite{csw, wolan2}. Let us define,
$$
J_{\rho}({\rm w})=\frac{1}{2\rho}\int_{\Omega}|\nabla {\rm w}|^2{\rm d}x- \log\left(\int_{\Omega} H e^{\rm w}\right),\quad {\rm w}\in W^{1,2}_0(\Omega),
$$
for $\rho \in(0,+\infty)$.

\begin{lem}\label{lem5.1}$\,$
	For any $\rho >0$ we have
	$$
	f(\rho)=\inf\{\mathcal{F}_{\rho}(\omega)\mid\omega \in \mathcal{P}(\Omega)\}=
	\inf\{{J}_{\rho}({\rm w})\mid {\rm w}\in W^{1,2}_0(\Omega)\}.
	$$
	In particular we have:\\
	$(a)$ For any $\omega\in \mathcal{P}(\Omega)$ let ${\rm w}=\rho G[\omega]$. Then ${\rm w}\in W^{1,2}_0(\Omega)$ and
	$$
	J_\rho({\rm w})\leq \mathcal{F}_{\rho}(\omega),
	$$
	where the equality holds if and only if ${\rm w}$ is a solution of \eqref{r-equ-flat-lm};\\
	$(b)$ For any ${\rm w}\in W^{1,2}_0(\Omega)$ let $\omega=\frac{H e^{\rm w}}{\int_{\Omega} H e^{\rm w}}$. Then $\omega\in \mathcal{P}(\Omega)$ and
	$$
	J_\rho({\rm w})\geq \mathcal{F}_{\rho}(\omega),
	$$
	where the equality holds if and only if ${\rm w}$ is a solution of \eqref{r-equ-flat-lm}.
\end{lem}

\begin{proof} We suppress the $x$ dependence for the sake of simplicity. Obviously the first claim follows from $(a)$ and $(b)$.
	
	\noindent $(a)$ Let ${\rm w}=\rho G[\omega]$, we have (\cite{stam2}) ${\rm w}\in W^{1,2}_0(\Omega)\cap L^{\infty}(\Omega)$
	and consequently
	$$
	\int_{\Omega} H e^{\rm w}<+\infty\quad\mbox{and}\quad\omega_{\rm w}=\frac{H e^{\rm w}}{\int_{\Omega} H e^{\rm w}}\in {\mathcal P}(\Omega).
	$$
	Thus it is easy to check that,
	$$
	J_\rho({\rm w})=\int_\Omega \omega\left(\log\omega_{\rm w}-\log H\right)-\frac{1}{2}\int_\Omega \omega{\rm w},
	$$
	implying that,
	$$
	\mathcal{F}_{\rho}(\omega)-J_\rho({\rm w})=\int_\Omega \omega\left(\log\omega-\log \omega_{\rm w}\right)\geq 0
	$$
	where the inequality follows from the Jensen inequality 
	$$\int_{\Omega}s(f(x))\omega_{\rm w}(x){\rm d}x\geq s\left(\int_{\Omega}f(x)\omega_{\rm w}(x){\rm d}x\right),$$
	where 
	$$s(t)=t\log(t),~ t\geq 0\quad\mbox{and}\quad f(x)=\frac{\omega(x)}{\omega_{\rm w}(x)}.$$ Remark that, since $s(t)$ is strictly convex, the equality holds if and only if
	$f(x)$ is constant a.e., that is, since $\{\omega,\omega_{\rm w}\}\subset \mathcal{P}(\Omega)$, if and only if $\omega=\omega_{\rm w}$ a.e. in $\Omega$.
	However this is the same as to say that ${\rm w}=\rho G[\omega_{\rm w}]$, that is, by standard elliptic estimates, that ${\rm w}$ is a solution of
	\eqref{r-equ-flat-lm}.

	\noindent $(b)$ Let $\omega=\frac{H e^{\rm w}}{\int_{\Omega} H e^{\rm w}}$, by the Moser-Trudinger inequality (\cite{moser}) and the H\"older
	inequality
	$H e^{\rm w}{\rm d}x\in L^{p}(\Omega)$ for any $p\in [1,|\beta|^{-1})$ whence $\omega$ is well defined.
	Observe that, putting $c_{\rm w}=\log(\int_{\Omega} H e^{\rm w})$, we have
	$$
	\int_{\Omega} \omega\log\omega=-c_{\rm w}+\int_{\Omega} \omega({\rm w}+\log H ),
	$$
	where by the Sobolev embedding ${\rm w}+\log H\in L^q(\Omega)$ for any $q\in [1,+\infty)$, implying that $\omega\in \mathcal{P}(\Omega)$. Let
	${\rm w}_{\omega}=\rho G[\omega]$, then we have (\cite{stam2})
	$$
	\int_{\Omega}|\nabla{\rm w}_{\omega}|^2=\rho \int_{\Omega} \omega {\rm w}_{\omega}, \quad
	\int_{\Omega} \nabla {\rm w}_{\omega} \cdot \nabla {\rm w}=\rho \int_{\Omega} \omega {\rm w}.
	$$
	Therefore we deduce that,
	$$
	J_\rho({\rm w})-\mathcal{F}_{\rho}(\omega)=\frac{1}{2\rho}\int_{\Omega}|\nabla({\rm w}-{\rm w}_{\omega})|^2\geq 0
	$$
	where, since $\{{\rm w},{\rm w}_{\omega}\}\in W^{1,2}_0(\Omega)$, the equality holds if and only if ${\rm w}={\rm w}_{\omega}$ a.e.,
	that is, by standard elliptic estimates, if and only if ${\rm w}$ is a solution of
	\eqref{r-equ-flat-lm}.
\end{proof}

Now we are able to state the first result about the ({\rm \bf CVP}), which generalizes the corresponding results in \cite{caglioti-2} and \cite{chang-chen-lin}.
\begin{thm}\label{thm.1} 
	For any $\rho \in (0,8\pi(1+\beta))$ the {\rm (}{\rm \bf CVP}{\rm )} admits a unique minimizer $\omega_\rho$.
	In particular ${\rm w}_\rho=\rho G[\omega_\rho]$ is the unique solution of \eqref{r-equ-flat-lm}.
\end{thm}
\begin{proof}
	The uniqueness part is readily seen to be a straightforward consequence of Lemma \ref{lem:spectral}-$({\rm i})$, whence we are just left to prove the existence part of the claim. 
	By the Moser-Trudinger inequality (\cite{moser}) and the H\"older inequality, $J_{\rho}$ is well defined. Furthermore, for any fixed $\rho\in [0,8\pi(1+\beta))$,
	by the singular (\cite{ads}) Moser-Trudinger inequality,  $J_\rho$ is coercive and weakly lower semi-continuous on $W^{1,2}_0(\Omega)$. By the direct method we get that it admits a minimizer.
	The corresponding Euler-Lagrange equation is just \eqref{r-equ-flat-lm} and the conclusion follows by Lemma \ref{lem5.1}.
\end{proof}

\bigskip

To proceed further we restrict our attention to the physically meaningful situation where the singular source $x=0\in\Omega$ satisfies \eqref{neg-crit-2}, that is,  
setting,
$$
\gamma^*(x)=8\pi(1+\beta)R_{\Omega}(x,0) +\log(h_0(x)), \quad x\in \Omega,
$$
we assume that 
\begin{equation}\label{ons-crit}
	\nabla \gamma^*(0)=0.    
\end{equation}
With the same notations as in \eqref{Phi-beta} let us set,
$$
\Phi_{\Omega,0}(x)=8\pi(1+\beta)({G}_{\Omega}(x,0)-{R}_{\Omega}(0,0))+\log(h_0(x)/h_0(0))+2\beta \log|x|,
$$
which satisfies
$$
e^{\Phi_{\Omega,0}(x)}=\frac{|x|^{2\beta}}{|x|^{4+4\beta}}e^{\gamma^*(x)-\gamma^*(0)}.
$$

{Then, in view of \eqref{ons-crit} the following quantity,   is well defined,  
	\begin{equation}\label{D0}
		D_\beta=\lim\limits_{r\to 0}c_*\left(\int_{\Omega\setminus B_{r}}e^{\Phi_{\Omega,0}(x)}-\frac{\pi}{(1+\beta)}\frac{1}{{r}^{2(1+\beta)}}\right),
\end{equation}}
and we remark that $D_\beta$ was already introduced in \eqref{d-beta} above. 
Next, to simplify the evaluations, let us set,
$$c_\beta=\frac{h_0(0)}{8(1+\beta)^2},\quad c_*=e^{\gamma^*(0)}= h_0(0)e^{8\pi(1+\beta){R}_{\Omega}(0,0)},$$
and choose any $\tau>0$ such that $B_\tau\Subset \Omega$. Then we define,
$$
{\rm w}_{\varepsilon}(x)=-4(1+\beta)\log(\tau)+2\log
\left(\dfrac{{\varepsilon}^{2(1+\beta)}+c_\beta}{{\varepsilon}^{2(1+\beta)}+c_\beta(\tau^{-1}|x|)^{2(1+\beta)}}\right), \quad x\in \overline{B_\tau},
$$
and
$$
{\widetilde {\rm w}}_{\varepsilon}(x)=
\begin{cases}
	-4(1+\beta)\log|x| +8\pi(1+\beta){R}_{\Omega}(x,0), \quad  &x\in \Omega\setminus B_\tau, \\ \\
	{\rm w}_{\varepsilon}(x) +8\pi(1+\beta){R}_{\Omega}(x,0), \quad &x\in \overline{B_\tau}, 
\end{cases}
$$
so that in particular ${\widetilde {\rm w}}_{\varepsilon}\equiv 8\pi(1+\beta) G_{\Omega}(x,0)$ in $\Omega\setminus B_\tau$.
At this point, we are ready to deduce a generalization of former results in \cite{caglioti-2}, \cite{chang-chen-lin},
\cite{BLin2}, \cite{BLin3}, which in particular provides an estimate for $f(8\pi(1+\beta))$ in terms of $D_\beta$.  
\begin{lem}\label{Iexp}
	Assume \eqref{ons-crit} and set 
	\[
	\gamma_\Omega(x)=4\pi(1+\beta){R}_{\Omega}(x,x)+\log(h_0(x)),
	\]
	then we have,
	\begin{align*}
		J_{8\pi(1+\beta)}({\widetilde {\rm w}}_{\varepsilon})=&-1-\log\left(\frac{\pi}{1+\beta}\right)-\gamma_\Omega(0)\\
		&-{({\tau} \varepsilon)}^{2(1+\beta)}\frac{1+\beta}{\pi c_*c_\beta}
		\left(D_\beta+o_{{ \varepsilon}}(1)+O(\tau^{-2\beta})\right),
	\end{align*}
	$\mbox{as}\; { \varepsilon}\to 0$, where $o_{{ \varepsilon}}(1)\to 0$ as ${ \varepsilon}\to 0$, uniformly with respect to $\tau$. Then we have
	\begin{equation}\label{Jinf}
		f(8\pi(1+\beta))\leq -1-\log\left(\frac{\pi}{1+\beta}\right)-\gamma_\Omega(0),
	\end{equation}
	where the inequality is strict as far as $D_\beta>0$.
\end{lem}
\begin{proof}
	\begin{align*}
		\frac{1}{2}\int_\Omega |\nabla {\widetilde {\rm w}}_{ \varepsilon}|^2=&~\frac{(1+\beta)^2}{2}\int_{\Omega\setminus
			B_\tau}\nabla\left(\log\frac{1}{|x|^4}+8\pi{R}_{\Omega}(x,0)\right) \cdot \nabla\log\frac{1}{|x|^4}\\
		&+4\pi(1+\beta)^2 \left(\int_{\Omega\setminus B_\tau}\nabla\log{\frac{1}{|x|^4}}\cdot \nabla{R}_{\Omega}(x,0)
		+8\pi\int_{\Omega}|\nabla{R}_{\Omega}(x,0)|^2\right)\\
		&+\frac{1}{2}\int_{B_\tau}\left|\frac{c_\beta\tau^{-2(1+\beta)}4(1+\beta)|x|^{1+2\beta}}{{\varepsilon}^{2(1+\beta)}+c_\beta(\tau^{-1}|x|)^{2(1+\beta)}}\right|^2 +
		8\pi(1+\beta)\int_{B_\tau}\nabla{{\rm w}}_{ \varepsilon} \cdot \nabla{R}_{\Omega}(x,0)\\
		=&\frac{-(1+\beta)^2}{2}\int_{\partial B_\tau}
		\left(\log\frac{1}{|x|^4}+8\pi{R}_{\Omega}(x,0)\right)
		\frac{\partial}{\partial\nu}\left(\log\frac{1}{|x|^4}\right)d\sigma\\
		&-4\pi(1+\beta)^2\int_{\partial B_\tau}\left(\log\frac{1}{|x|^4}\right)\frac{\partial{R}_{\Omega}}
		{\partial\nu}(x,0)d\sigma-16\pi(1+\beta)^2\log( \varepsilon)\\
		&-8\pi(1+\beta)(1-\log(c_\beta))+16\pi(1+\beta)c_\beta^{-1}{\varepsilon}^{2(1+\beta)}
		+O({ \varepsilon}^{4(1+\beta)})\\
		=&-16\pi(1+\beta)^2\log(\tau)+32\pi^2(1+\beta)^2R_{\Omega}(0,0)-16\pi(1+\beta)^2\log( \varepsilon)\\
		&-8\pi(1+\beta)(1-\log(c_\beta))+16\pi(1+\beta)c_\beta^{-1}{ \varepsilon}^{2(1+\beta)}+O({\varepsilon}^{4(1+\beta)}),
	\end{align*}
	as ${ \varepsilon}\to 0$, where we used:
	\begin{align*}
		\left(\int_{\Omega\setminus B_\tau}\nabla\log{\frac{1}{|x|}}\cdot \nabla{R}_{\Omega}(x,0)
		+2\pi\int_{\Omega}|\nabla{R}_{\Omega}(x,0)|^2\right)
		=\int_{\partial B_\tau}\left(\log|x|\right)\frac{\partial{R}_{\Omega}}
		{\partial\nu}(x,0)d\sigma=0,
	\end{align*}
	\begin{align*}
		\frac{1}{2}\int_{B_\tau}\left|\frac{c_\beta\tau^{-2(1+\beta)}4(1+\beta)|x|^{1+2\beta}}{ \varepsilon^{2(1+\beta)}+c_\beta(\tau^{-1}|x|)^{2(1+\beta)}}\right|^2=&
		-16\pi(1+\beta)^2\log(\varepsilon)-8\pi(1+\beta)(1-\log(c_\beta))\\
		&+16\pi(1+\beta)c_\beta^{-1}{ \varepsilon}^{2(1+\beta)}
		+O({ \varepsilon}^{4(1+\beta)}),
	\end{align*}
	as ${ \varepsilon}\to 0$, and
	\begin{align*}
		\int_{B_\tau}\nabla{ {\rm w}}_{\varepsilon} \cdot \nabla{R}_{\Omega}(x,0)=
		\int_{\partial B_\tau}{ {\rm w}}_{\varepsilon}
		\frac{\partial}{\partial\nu}{R}_{\Omega}(x,0)d\sigma=0;
	\end{align*}
	$$
	\int_{\partial B_\tau}8\pi{R}_{\Omega}(x,0)
	\frac{\partial}{\partial\nu}\log\frac{1}{|x|^4}=-\frac{32\pi}{\tau}\int_{\partial B_\tau}{R}_{\Omega}(x,0)d\sigma=-64\pi^2{R}_{\Omega}(0,0).
	$$
	
	\bigskip
	
	\noindent Next observe that, recalling the definitions of $c_*$ and $\Phi_{\Omega,0}$ above, we have 
	{\begin{align*}
			\int_{\Omega\setminus B_\tau}|x|^{2\beta}h_0(x)e^{{\widetilde {\rm w}}_{\varepsilon}}=&
			~c_*\int_{\Omega\setminus B_\tau}e^{\Phi_{\Omega,0}(x)}=
			c_*\int_{\Omega \setminus B_{{ \varepsilon}} }e^{\Phi_{\Omega,0}(x)}-c_*\int_{B_\tau \setminus B_{{ \varepsilon}}}e^{\Phi_{\Omega,0}(x)}\\
			=&~c_*\int_{\Omega \setminus B_{{ \varepsilon}}}e^{\Phi_{\Omega,0}(x)}+
			\frac{c_* \pi}{(1+\beta)}\left(\frac{1}{{\tau}^{2(1+\beta)}}-\frac{1}{{{ \varepsilon}}^{2(1+\beta)}}\right)+O(\tau^{-2\beta})
		\end{align*}
		where we used
		\begin{align*}
			c_*\int_{B_\tau \setminus B_{{ \varepsilon}}}e^{\Phi_{\Omega,0}(x)}=~&
			\int_{B_\tau \setminus B_{{ \varepsilon}}}\dfrac{|x|^{2\beta}}{|x|^{4(1+\beta)}}h_0(x)e^{8\pi(1+\beta)R_{\Omega}(x,0)}\\
			=~&\frac{c_* \pi}{(1+\beta)}\left(\frac{1}{{{ \varepsilon}}^{2(1+\beta)}}-\frac{1}{{\tau}^{2(1+\beta)}}\right)+O(\tau^{-2\beta}).
	\end{align*}}

	{Moreover we have,
		\begin{align*}
			\int_{B_\tau}|x|^{2\beta}h_0(x)e^{{\widetilde {\rm w}}_{ \varepsilon}}=&
			\int_{B_\tau}|x|^{2\beta}e^{\gamma^*(x)}e^{{\rm w}_{ \varepsilon}}=\int_{B_\tau\setminus B_{ \varepsilon}}|x|^{2\beta}e^{\gamma^*(x)}e^{{\rm w}_{ \varepsilon}}+
			\int_{B_{ \varepsilon}}|x|^{2\beta}e^{\gamma^*(x)}e^{{\rm w}_{ \varepsilon}}\\
			=&~\dfrac{\pi c_*}{(1+\beta)}\frac{1}{({ \varepsilon}{\tau})^{2(1+\beta)}}\left(c_\beta+{ \varepsilon}^{2(1+\beta)}+O(\tau^{-2\beta})\right),
		\end{align*}
		where, putting
		$$
		\tau_\beta=(c_\beta)^{-\frac{1}{2(1+\beta)}}\tau,
		$$
		we used
		\begin{align*}
			\int_{B_{\varepsilon}}|x|^{2\beta}e^{\gamma^*(x)}e^{{\rm w}_{\varepsilon}}=~&
			\tau^{-4(1+\beta)}\int_{B_{\varepsilon}}|x|^{2\beta}e^{\gamma^*(x)}
			\left(\dfrac{{ \varepsilon}^{2(1+\beta)}+c_\beta}{{\varepsilon}^{2(1+\beta)}+c_\beta(\tau^{-1}|x|)^{2(1+\beta)}}\right)^2\\
			=~&\dfrac{\left(c_\beta+{ \varepsilon}^{2(1+\beta)}\right)^2}{c_\beta({\varepsilon}\tau)^{2(1+\beta)}}
			\int_{B_{\tau_\beta^{-1}}}\dfrac{|z|^{2\beta}}{(1+|z|^{2(1+\beta)})^2}e^{\gamma^*({\varepsilon}\tau_\beta z)}\\
			=~&\dfrac{\left(c_\beta+{ \varepsilon}^{2(1+\beta)}\right)^2}{c_\beta({ \varepsilon}\tau)^{2(1+\beta)}}
			\dfrac{\pi c_*}{(1+\beta)}\left(\frac{1}{1+\tau_\beta^{2(1+\beta)}}+O(({ \varepsilon}\tau)^{2})\right),
		\end{align*}
		and similarly
		\begin{align*}
			&\int_{B_{\tau}\setminus B_{ \varepsilon}}|x|^{2\beta}e^{\gamma^*(x)}e^{{\rm w}_{ \varepsilon}}\\
			&=\dfrac{\left(c_\beta+{\varepsilon}^{2(1+\beta)}\right)^2}{c_\beta({ \varepsilon}\tau)^{2(1+\beta)}}
			\dfrac{\pi c_*}{(1+\beta)}\left(\frac{c_\beta}{c_\beta+{ \varepsilon}^{2(1+\beta)}}-\frac{1}{1+\tau_\beta^{2(1+\beta)}}+O(\tau^2{ \varepsilon}^{2(1+\beta)})\right).
	\end{align*}}
	
	{Therefore we have
		\begin{align*}
			&\int_\Omega |x|^{2\beta}h_0(x)e^{{\widetilde {\rm w}}_{\varepsilon}}
			=c_*\int_{\Omega \setminus B_{\tau}}e^{\Phi_{\Omega,0}(x)}+\int_{B_\tau}|x|^{2\beta}h_0(x)e^{{{\rm w}}_{\varepsilon}}\\
			&=c_*\left(\int_{\Omega \setminus B_{{ \varepsilon}}}e^{\Phi_{\Omega,0}(x)}-\frac{\pi}{(1+\beta)}\frac{1}{{{ \varepsilon}}^{2(1+\beta)}}\right)
			+\frac{\pi c_*}{(1+\beta)}\frac{1}{{{\tau}}^{2(1+\beta)}}\\
			&\quad+\dfrac{\pi c_*}{(1+\beta)}\frac{1}{{(\tau{ \varepsilon})}^{2(1+\beta)}}
			\left({c_\beta}+{\varepsilon}^{2(1+\beta)}+O(\tau^{-2\beta})\right)
			+O(\tau^{-2\beta})\\
			&=\dfrac{\pi c_*}{(1+\beta)}\frac{1}{{(\tau{\varepsilon})}^{2(1+\beta)}}
			\left(c_\beta+\left(\frac{(1+\beta)}{\pi c_*}\tau^{2(1+\beta)}(D_\beta+o_{{ \varepsilon}}(1))+2+O(\tau^{-2\beta})\right) {{ \varepsilon}^{2(1+\beta)}}\right),
		\end{align*}
		where $o_{{\varepsilon}}(1)\to 0$ as ${\varepsilon}\to 0$ uniformly with respect to $\tau$.}
	
	{At this point, putting together these estimates,
		we see that,
		\begin{align*}
			J_{8\pi(1+\beta)({\tilde {\rm w}}_{\varepsilon})}=&-2(1+\beta)\log(\tau)-2(1+\beta)\log( \varepsilon)+2c_\beta^{-1}{\varepsilon}^{2(1+\beta)}\\
			&+4\pi(1+\beta)R_{\Omega}(0,0)-1+\log(c_\beta)+O({ \varepsilon}^{4(1+\beta)})\\
			&-\log\left(\dfrac{\pi c_*c_\beta}{(1+\beta)}\frac{1}{{(\tau{ \varepsilon})}^{2(1+\beta)}}\right)\\
			&-\log\left(1+\left(\frac{(1+\beta)}{\pi c_*c_\beta}{\tau}^{2(1+\beta)}(D_\beta+o_{{ \varepsilon}}(1))+2c_\beta^{-1}+O(\tau^{-2\beta})\right)
			{{ \varepsilon}^{2(1+\beta)}}\right)\\
			=&-1-4\pi(1+\beta)R_{\Omega}(0,0)-\log\left(\frac{\pi h_0(0)}{1+\beta}\right)\\
			&-{({\tau} \varepsilon)}^{2(1+\beta)}\frac{1+\beta}{\pi c_*c_\beta}
			\left(D_\beta+o_{{\varepsilon}}(1)+O(\tau^{-2\beta})\right),\quad \mbox{as}\; { \varepsilon}\to 0
		\end{align*}
		as claimed.}
	Passing to the limit we obtain the inequality for $f(8\pi(1+\beta))$, where if $D_0>0$ then obviously the inequality must be strict.
\end{proof}

\bigskip

\subsection{Microcanonical Variational Principle}\label{sec5.2}
With the definitions above, for $E>0$ we define the Microcanonical Variational Principle (({\rm \bf MVP}) for short),
$$
S(E)=\sup\{\mathfrak{S}(\omega)\mid \omega \in \mathcal{P}(\Omega),~\mathcal{E}(\omega)=E\}\eqno{\rm \bf (MVP)}
$$
where
$$
\mathfrak{S}(\omega)=-\int_\Omega \omega(x)(\log\omega(x)-\log H(x)){\rm d}x.
$$
As remarked above, due to the inequality $ab\leq e^a+b(\log(b))$, $a\geq 0,b\geq 1$, it is readily seen that
$\mathfrak{S}(\omega)$ is well defined on $\mathcal{P}(\Omega)$. Concerning the ({\rm \bf MVP}) we have
the following generalization of a result in \cite{caglioti-2}.
\begin{thm}\label{thm:5.3} 
	$(j)$ For any $E>0$, $S(E)<+\infty$ and there exists $\omega_{_E}$ which solves the {\rm ({\rm \bf MVP})},
	$S(E)=\mathfrak{S}(\omega_{_E})$. In particular there exists $\rho_{_E}\in \mathbb{R}$ such that ${\rm w}_{_E}=\rho_{_E}G[\omega_{_E}]$ is a solution
	of \eqref{r-equ-flat-lm} for $\rho=\rho_{_E}$ and $\omega_0(x):=\frac{H(x)}{\int_{\Omega} H}$ is the unique entropy maximizer
	for $E=E_0:=\mathcal{E}(\omega_0)$ and $\rho_{E_0}=0$.\\
	$(jj)$ There exists $E_c\in (E_0,+\infty]$ such that for any $E\in [E_0,E_c)$ the map $E\mapsto \rho_{_E}$ is smooth and
	invertible with inverse $E=E(\rho)$, $\rho\in[0, 8\pi(1+\beta))$ and $E(\rho)\nearrow E_c$, as $\rho \nearrow 8\pi(1+\beta)$,
	$\frac{d E(\rho)}{d\rho}>0$ in $[0, 8\pi(1+\beta))$.
\end{thm}
\begin{proof}
	$(j)$ Observe that, putting ${\rm d}\mu_H=\frac{H(x){\rm d}x}{\int_{\Omega}H(x)dx}$ and $f(x)=\frac{\omega(x)}{H(x)}$,
	then by a standard approximation argument and the Jensen inequality for $s(t)=-t\log(t),t\geq 0$, we have,
	$$
	\frac{\mathfrak S(\omega)}{\int_{\Omega}H(x)dx}=\int_{\Omega}s(f(x)){\rm d}\mu_H\leq s\left(\int_{\Omega}f(x){\rm d}\mu_H\right)=\frac{\log\int_{\Omega}H(x)dx}{\int_{\Omega}H(x)dx},
	$$
	where the equality holds if and only if $f$ is constant a.e., that is, if and only if $\omega=\omega_0(x)=\frac{H(x)}{\int_{\Omega} H}$, which
	is therefore the unique maximizer of the entropy at $E=E_0:=\mathcal{E}(\omega_0)$. For $E\neq E_0$,
	let $H_n$ be any sequence of smooth positive functions such that $H_n\to H$ in $L^p(\Omega)$, for some $p\in (1,|\beta|^{-1})$ and define
	$$
	\mathfrak{S}_n(\omega)=-\int_\Omega \omega(x)(\log\omega(x)-\log H_n(x)){\rm d}x, \quad \omega\in \mathcal{P}(\Omega).
	$$
	Modulo minor modifications due to the weight $H_n$, by the same argument adopted in \cite{caglioti-2} we see that
	for any $E>0$, $S_n(E)<+\infty$ and there exists $\omega_{_E,n}\in \mathcal{P}(\Omega)$ which maximizes $\mathfrak{S}_n(\omega)$ at fixed energy
	$\mathcal{E}(\omega_{_E,n})=E$,
	$S_n(E)=\mathfrak{S}(\omega_{_E,n})$. In particular there exists $\rho_{_E,n}\in \mathbb{R}$ such that
	${\rm w}_{_E,n}=\rho_{_E,n}G[\omega_{_E,n}]$
	is a solution of \eqref{r-equ-flat-lm} for $\rho=\rho_{_E,n}$ and $H=H_n$. By standard elliptic regularity theory and the Sobolev embedding, for each fixed $n$, we have
	${\rm w}_{_E,n}\in W^{2,p}_0(\Omega)\cap C^{0,\gamma}(\overline{\Omega})$, for some $\gamma\in (0,1)$.
	Since $\Omega$ is simply connected and of class $C^1$, with the aid of a conformal map, by a well known trick based on the Pohozaev identity
	we can find $\overline{\rho}_{\Omega}$ such that $\rho_{_E,n}\leq \overline{\rho}_{\Omega}$ for any $n$ and $E$. Therefore, for a fixed $E$, by the equation
	\eqref{r-equ-flat-lm} we readily deduce that $\|\nabla {\rm w}_{_E,n}\|_2\leq \overline{\rho}_{\Omega}E$, implying
	by the Moser-Trudinger inequality that $e^{{\rm w}_{_E,n}}$ is uniformly bounded in $L^{q}(\Omega)$ for any $q\geq 1$ and in particular
	that $H_n e^{{\rm w}_{_E,n}}$ is uniformly bounded in $L^{t}(\Omega)$ for some $t\in(1,p)$. Actually by the Sobolev embedding
	${\rm w}_{_E,n}$ is uniformly bounded in $L^{q}(\Omega)$ for any $q\geq 1$, implying by the Jensen inequality
	that $\int_{\Omega} H_n e^{{\rm w}_{_E,n}}$ is bounded below away from $0$.
	Therefore the right hand side of \eqref{r-equ-flat-lm} with $\rho=\rho_{_E,n}$, $H=H_n$ and ${\rm w}={\rm w}_{_E,n}$
	is uniformly bounded in $L^{t}(\Omega)$ and
	by standard elliptic estimates and the Sobolev embedding ${\rm w}_{_E,n}$ is uniformly bounded in
	${\rm w}_{_E,n}\in W^{2,t}_0(\Omega)\cap C^{0,\gamma}(\overline{\Omega})$, for some $\gamma\in (0,1)$.
	At this point it is readily seen that, possibly
	along a subsequence, as $n\to +\infty$, ${\rm w}_{_E,n}$ converges uniformly and in $W^{1,2}_0(\Omega)$ to a solution ${\rm w}_{_E}$
	of \eqref{r-equ-flat-lm} for some $\rho_{_E}\in \mathbb{R}$,
	with $\mathcal{E}(\omega_{_E})=E$, where $\omega_{_E}=\dfrac{He^{{\rm w}_{_E}}}{\int_\Omega He^{{\rm w}_{_E}}}$. Clearly
	$\omega_{_E}$ is a maximizer of the entropy at fixed energy, $S(E)=\mathfrak{S}(\omega_{_E})$, which, excluding $\rho_{E_0}=0$, concludes the proof of $(j)$. 
	\smallskip
	
	\noindent $(jj)$ Remark that if $\rho_{_E}< 0$ then $J_\rho$ is strictly convex and consequently
	\eqref{r-equ-flat-lm} admits a unique solution, which therefore is just ${\rm w}_{_E}$. By Lemma \ref{lem:spectral} and
	Remark \ref{rem:spectral} solutions of \eqref{r-equ-flat-lm} are in fact unique and the first
	eigenvalue of the linearized problem relative to \eqref{r-equ-flat-lm} is strictly positive as far as $\rho<8\pi(1+\beta)$.
	In particular the inverse of $E\mapsto \rho_{_E}$, say $E=E(\rho)$, is well defined and smooth and $\frac{d E(\rho)}{d\rho}>0$ as far as
	$\rho \in (-\infty,8\pi(1+\beta))$. As a consequence $\rho=\rho_{_E}\leq 0$ if and only if $E(\rho)\leq E_0=\mathcal{E}(\omega_0)$,
	where $\omega_0=\frac{H(x)}{\int_{\Omega} H(x)}$. However $\omega_{_E}=\dfrac{He^{{\rm w}_{_E}}}{\int_\Omega He^{{\rm w}_{_E}}}=\omega_0$
	if and only if ${\rm w}_{_E}\equiv{0}$ implying that $\rho_{E_0}=0$. By monotonicity the limit 
	of $E(\rho)$ as $\rho \to 8\pi(1+\beta)^{-}$ is well defined which we take as defining $E_c$.
\end{proof}

\bigskip

\section{As The existence/non existence of solutions for $\rho=8\pi(1+\beta)$. Domains of first/second kind.}\label{section:firstsec}
As mentioned above, the existence/non existence of solutions of \eqref{r-equ-flat-lm} for $\rho=8\pi(1+\beta)$ is a subtle problem, since $J_{8\pi(1+\beta)}$ is bounded below (\cite{ads}) but not
coercive, implying that in general minimizing sequences need not be compact. 
We discuss hereafter the case where a negative singular sink is described in the sense defined above.

First of all, as in \cite{caglioti-2}, the critical value of the energy $E_c$ defined in Theorem \ref{thm:5.3} plays a crucial in the rest of the argument.
\begin{Def}\label{def:first}
	A pair $(\Omega, \beta)$, where $\Omega$ is a simply connected and bounded domain of class $C^2$, $\beta\in (-1,0)$ and $q=0\in \Omega$ satisfies \eqref{ons-crit} is said to be of {\bf first kind} if 
	$E_c=E_c(\Omega,\beta)=+\infty$, of {\bf second kind} otherwise.
\end{Def} 

Any disk centered at the origin, $\Omega=B_r(0)$, with $H(x)=e^{-4\pi\beta G_\Omega(x,0)}=\frac{|x|^{2\beta}}{r^{2\beta}}$ is of first kind for any $\beta\in (-1,0)$. 
In fact, by dilation invariance we are free to take $r=1$ and deduce that 
$$
D_0=-\frac{\pi}{1+\beta},
$$
$$
{\rm w}_{\rho}(x)=2\log\left(\dfrac{1+\gamma_{_\rho}}{1+\gamma_{_\rho}|x|^{2(1+\beta)}}\right),\quad
\gamma_{_\rho}=\frac{\rho}{8\pi(1+\beta)-\rho}, \quad |x|\leq 1.
$$
Then
$$
\omega_{\rho}(x)=\dfrac{|x|^{2\beta}e^{{\rm w}_{\rho}}}{\int_{\Omega} |x|^{2\beta} e^{{\rm w}_{\rho}}}=\frac{(1+\beta)}{\pi}\dfrac{(1+\gamma_{_\rho})|x|^{2\beta}}
{\left(1+\gamma_{_\rho}|x|^{2(1+\beta)}\right)^2},
$$
and
$$\mathcal{E}({\omega_{\rho}})=\frac{1}{\rho}\left(1+\left(1-\frac{1}{\gamma_{\rho}}\right)\log(1+\gamma_\rho)\right).$$
It is readily seen in particular that $E(\rho)=\mathcal{E}(\omega_\rho)\to +\infty$ 
as $\rho\nearrow 8\pi(1+\beta)$.
Our first result concerning domains of first kind is a generalization of analogue results in the regular case, see Theorem 7.1 in \cite{caglioti-1} and Proposition 3.3 in \cite{caglioti-2}.
\begin{thm}\label{thm:6.1}$\,$ Let $(\Omega,\beta)$ be of first kind, then we have:\\	
	$(i)$ $ f(\rho)$ is smooth, strictly convex and increasing for any $\rho\in (0,8\pi(1+\beta))$. In particular
	$$
	E(\rho)=-\frac{df(\rho)}{d\rho}=\frac12 \int_{\Omega}\omega_{\rho}G[\omega_{\rho}],\quad \rho \in (0,8\pi(1+\beta)),
	$$
	where $\omega_{\rho}$ denotes the unique solution of the {\rm (}{\rm \bf CVP}{\rm )} and $\frac{d E(\rho)}{d\rho}>0$;\\
	$(ii)$ for $E\in (E_0,+\infty)$ we have that,
	$$
	S(E)=\inf\limits_{\rho\in \mathbb{R}}\{-f(\rho)-\rho E\}
	$$
	is smooth and concave;\\
	$(iii)$ {\rm [Equivalence of Statistical Ensembles]} If $\omega(E)$ solves the {\rm (}{\rm \bf MVP}{\rm )} then $\omega({E(\rho)})=\omega_{\rho}$,
	that is, $\omega({E})$ coincides with the unique solution of the {\rm (}{\rm \bf CVP}{\rm )} for $\rho =\rho_{E}$, where $\rho_{_E}$ is the inverse of
	$E(\rho)$.\\
	$(iv)$ As $\rho\nearrow 8\pi(1+\beta)$,   
	$$
	f(\rho)\to
	f(8\pi(1+\beta))= -1-\log\left(\frac{\pi}{1+\beta}\right)-\gamma_{\Omega}(0),
	$$
	and 
	$\omega_\rho \rightharpoonup \delta_{q=0}$.
\end{thm}
\begin{proof} The proof of $(i)-(ii)-(iii)$ goes  as in \cite{caglioti-2}, with the exception that the smoothness of $f(\rho)$ as well
	as $E(\rho)=-\frac{df(\rho)}{d\rho}$ and $\frac{d E(\rho)}{d\rho}>0$ both follow from Lemma \ref{lem:spectral}-$({\rm v})$.
	Actually, in view of Lemma \ref{lem:spectral}-$({\rm v})$,  we also deduce that (see section 4 in \cite{Bons} for details)
	$$
	\frac{dS(E)}{dE}=-\rho_{_E},\quad \frac{d^2S(E)}{dE^2}=- \frac{d\rho_{_E}}{dE}=-\left(\frac{dE(\rho)}{d\rho}\right)^{-1}<0, \quad E\in (E_0,+\infty).
	$$
	We are left with the proof of $(iv)$. By definition of domain of first kind, $E(\rho)\to +\infty$ as $\rho\nearrow 8\pi(1+\beta)$. Consequently, according to well known blow-up results (\cite{BM3}), 
	we have that $\omega_\rho \rightharpoonup \delta_{q=0}$. It is not difficult to prove (see for example Proposition 7.3 in \cite{caglioti-1}) that $f(\rho)$ is continuous in $[0,8\pi(1+\beta)]$.  
	Let $\rho_n\nearrow 8\pi(1+\beta)$ and ${\rm w}_n$ be any sequence of minimizers of $J_{\rho_n}$, then we have
	from Lemma \ref{lem5.1},
	$$
	J_{\rho_n}({\rm w}_n)=f(\rho_n)\to f(8\pi(1+\beta)),\quad\mbox{as}\quad\rho_n\nearrow 8\pi(1+\beta).
	$$
	In particular, since ${\rm w}_n$ is a blow-up sequence, by using the pointwise estimates in \cite{BT-2} together with 
	Proposition \ref{est-muk} and \eqref{im-ck}
	we find that,
	$$
	J_{\rho_n}({\rm w}_n)\to f(8\pi(1+\beta))= -1-\log\left(\frac{\pi}{1+\beta}\right)-\gamma_{\Omega}(0).
	$$
	The evaluation is very similar to that worked out in the proof of Lemma \ref{Iexp}
	where one also uses the mean field equation \eqref{r-equ-flat-lm}, and we omit it to here to avoid repetitions. The proof is completed.
\end{proof}

\bigskip

As an immediate consequence of  Theorem \ref{thm:6.1}-$(iv)$ and Lemma \ref{Iexp} we have, 
\begin{cor} Let $(\Omega,\beta)$ be of first kind, then $D_\beta\leq 0$.
\end{cor}
To prove the converse of this fact is more subtle, see Theorem \ref{thm:5.2} below.\\

Theorem \ref{thm:6.1} shows that for pairs $(\Omega,\beta)$ of first kind we have a complete description of the thermodynamic equilibrium,
in the sense that for any $E>0$ there exists a unique $\omega_{E}$ which solves the {\bf (MVP)} and there exists a unique $\rho_{_E}$
such that there exists a unique solution of the {\bf (CVP)} at $\rho=\rho_{_E}$ which satisfies ${\rm w}_{_E}=\rho_{_E}G[\omega_{E}]$.
In other words the two variational principles describe the same thermodynamic equilibrium state.\\ 

For pairs $(\Omega,\beta)$ of second kind the situation is much more involved, since entropy maximizer with $E>E_c$ correspond to values of
$\rho=\rho_{_E}>8\pi(1+\beta)$ where the {\bf (CVP)} has no solution (since it can be shown by standard arguments that $J_{\rho}$ is not bounded below)
and moreover solutions of \eqref{r-equ-flat-lm} are not unique, see Theorem \ref{thm:count1} below. Concerning this point we generalize first some results in \cite{caglioti-2}. We refer to \cite{Bons}, \cite{bart-5}, \cite{BCN} for other recent
progress in this direction in the regular case.

We will need hereafter a generalization of former results in \cite{chang-chen-lin} and \cite{BLin3}.
We omit the proof which follows in part from known blow-up arguments (\cite{BM,BT,BM3}) and otherwise from the
invertibility of the linearized operator at $\rho=8\pi(1+\beta)$, see Lemma \ref{lem:spectral}-$({\rm iii}),({\rm vi})$, and the
implicit function theorem (see for example Proposition 6.1 in \cite{chang-chen-lin} for a proof in the regular case).

\begin{lem}\label{lem:5.1}
	Let $\rho \in [0,8\pi(1+\beta))$ and ${\rm w}_\rho $ be the unique solution of \eqref{r-equ-flat-lm}. Then, the following facts are equivalent:\\
	(-)  $(\Omega,\beta)$ is of second kind;\\
	(-)  For $\rho=8\pi(1+\beta)$, equation \eqref{r-equ-flat-lm} admits a solution denoted by ${\rm w}_{c}$;\\
	(-)  ${\rm w}_\rho \to {\rm w}_{c}$ in $C^2(\Omega)\cap C^{1}(\overline{\Omega})$ as $\rho\nearrow 8\pi(1+\beta)$;\\
	(-)  a subsequence ${\rm w}_{\rho_n}\to {\rm w}_{c}$ in $C^2(\Omega)\cap C^{1}(\overline{\Omega})$ as $\rho_n\nearrow 8\pi(1+\beta)$;\\
	(-)  $J_{8\pi(1+\beta)}$ admits a minimizer.\\
	Moreover, the following facts are equivalent:\\
	(-)  $(\Omega,\beta)$ is of first kind;\\
	(-)  $\|{\rm w}_\rho\|_{\infty}\to +\infty$, as $\rho\nearrow 8\pi(1+\beta)$;\\
	(-)  $\omega_\rho=\frac{He^{{\rm w}_\rho}}{\int_\Omega He^{{\rm w}_\rho}}\rightharpoonup \delta_{p=0}$ weakly in the sense of measures;\\
	(-)  a subsequence $\|{\rm w}_{\rho_n}\|_{\infty}\to +\infty$, as $\rho_n\nearrow 8\pi(1+\beta)$;\\
	(-)  Equation \eqref{r-equ-flat-lm} has no solution at $\rho=8\pi(1+\beta)$;\\
	(-)  $J_{8\pi(1+\beta)}$ has no minimizer. 
\end{lem}

\begin{rem}\label{rem:5.1} Clearly if $(\Omega,\beta)$ is of second kind we deduce from Lemma \ref{lem:5.1}
	and Lemma \ref{lem:spectral}-$({\rm vi})$ that as
	$\rho\nearrow 8\pi(1+\beta)$ the unique solutions
	of \eqref{r-equ-flat-lm} converge to ${\rm w}_c$ in $C^2(\Omega)\cap C^{1}(\overline{\Omega})$
	and then in particular $\mathcal{E}(\omega_\rho)=E(\rho)\to E_c=E(8\pi(1+\beta))=\mathcal{E}(\omega_c)$ where
	$\omega_c=\frac{He^{{\rm w}_c}}{\int_{\Omega}He^{{\rm w}_c}}$ and
	(recall Theorem \ref{thm:6.1}) $\mathfrak{S}(\omega_\rho)=S(E)\to S(E_c)$, implying that
	$f(\rho)\to f(8\pi(1+\beta))=-S(E_c)-8\pi(1+\beta)E_c$.
\end{rem}

We will also need the following generalizations of results in \cite{caglioti-2} about domains of second kind.
\begin{lem}\label{lem:entropyest} Let $(\Omega,\beta)$ be of second kind and $E>E_c$. Then
	$$
	C_1-8\pi(1+\beta)E\leq S(E)\leq C_2-8\pi(1+\beta)E,$$ 
	where
	$$ C_2=S(E_c)+8\pi(1+\beta)E_c=f(8\pi(1+\beta)).
	$$
\end{lem}
\begin{proof} In view of Remark \ref{rem:5.1} the same argument adopted in Proposition 6.1 in \cite{caglioti-2} works fine,
	in this case as well. 
\end{proof}

\begin{lem}\label{lem:entropyasymp} Let $(\Omega,\beta)$ be of second kind and denote by $\omega_{_E}$ any entropy maximizer defined by Theorem \ref{thm:5.3}.
	Then, as $E\to +\infty$ we have $\rho_{_E}\to 8\pi(1+\beta)$ and
	$$
	\omega_{_E}\rightharpoonup \delta_{p=0},\quad \mbox{ weakly in the sense of measures in }\Omega.
	$$
\end{lem}
\begin{proof}
	Let $E_k$, $\omega_k=\omega_{E_{k}}$ be any pair of sequences such that $E_k\to +\infty$ where $\omega_k$ is any entropy maximizer at energy
	$E_k$, put $\rho_k=\rho_{E_k}$, we prove that in fact $\omega_{k}\rightharpoonup \delta_{p=0}$ weakly in the sense of measures and
	$\rho_k\to 8\pi(1+\beta)$. We argue by contradiction and assume that along a subsequence any one of these properties is not satisfied. 
	
	As remarked in the proof of Theorem \ref{thm:5.3}, by the Pohozaev identity we have that
	$\rho_k$ is uniformly bounded, and since $E_k\to +\infty$, then it is not difficult to see that necessarily
	$\|{\rm w}_{k}\|_{\infty}\to +\infty$, where ${\rm w}_{k}=\rho_{k}G[\omega_k]$.
	Therefore ${\rm w}_{k}$ blows up as $k\to +\infty$ and then necessarily along a subsequence
	$\rho_{k}\to \rho_{\infty}$ where (\cite{BM3,BT,li-cmp})
	$$
	\rho_\infty=8\pi(1+\beta) n_0+8\pi\sum\limits_{i=1}^{N}n_i(1+\alpha_i)+8\pi n_r,
	$$
	with $$n_i \in\{0,1\}, i=0,1,\cdots,N,\quad n_r\in \mathbb{N},\quad n_0+\sum\limits_{j=1}^{N}n_j+n_r\geq 1.$$
	Observe that if $\mbox{\rm supp}\{\omega_i\}=\Omega_i$, $i=1,2$ and $\Omega_1\cap \Omega_2=\emptyset$, then
	$\mathfrak{S}(\omega_1+\omega_2)=\mathfrak{S}(\omega_1)+\mathfrak{S}(\omega_2)$ while the energy $\mathcal{E}(\omega)$ is exactly the same
	as the one in \cite{caglioti-2}. Therefore it is readily verified that the argument adopted in Theorem 6.1 in \cite{caglioti-2} works fine,
	showing that, possibly along a subsequence,
	$$
	\omega_{k}\rightharpoonup \delta_{x=p}, \quad\mbox{as}\quad k\to +\infty,
	$$
	weakly in the sense of measures in $\Omega$, for some $p\in \Omega$. Assume first that along a subsequence we reach a blow point $p$ satisfying $p\notin \{0,q_j\}$. Since ${\rm w}_{k}$ is a regular 1-point blow-up sequence, then the subtle balance obtained in Lemma \ref{Iexp} breaks down and, by 
	using the first order pointwise estimates in \cite{li-cmp} together with equation (6.8) in \cite{chen-lin-sharp}, 
	we would find that, 
	\begin{align*}
		J_{8\pi(1+\beta)}({\rm w}_{k})=&~\frac{1}{16\pi(1+\beta)}\int_{\Omega}|\nabla {\rm w}_{k}|^2-\log\left(\int_{\Omega}He^{{\rm w}_{k}}\right)\\
		=&-\frac{2}{1+\beta}\log(\bar \varepsilon_k)+2\log(\bar \varepsilon_k)+O(1)\to +\infty,~ \mbox{as} ~ k\to +\infty,
	\end{align*}
	for a suitably defined $\bar \varepsilon_k\to 0$. 
	Assume next that $p\in \{q_j\}$, since again ${\rm w}_{k}$ is a 1-point blow-up sequence at a positive singular source, say of strength $\alpha_j>0$, then by using the first order pointwise estimates in \cite{BCLT} together with Proposition \ref{est-muk}, we would
	find that, 
	\begin{align*}
		J_{8\pi(1+\beta)}({\rm w}_{k})=~&\frac{1}{16\pi(1+\beta)}\int_{\Omega}|\nabla {\rm w}_{k}|^2-\log\left(\int_{\Omega}He^{{\rm w}_{k}}\right)\\
		=~&-\frac{2(1+\alpha_j)}{1+\beta}\log(\varepsilon_k)+2\log(\varepsilon_k)+O(1)\to +\infty,~ \mbox{as} ~ k\to +\infty,
	\end{align*}
	for a suitably defined
	$\varepsilon_k\to 0$. On the other side, since $\omega_{k}=\omega_{E_k}$ is an entropy maximizer we have, from Lemma \ref{lem:entropyest} and Lemma \ref{lem5.1}
	$$
	J_{8\pi(1+\beta)}({\rm w}_{k})=\mathcal{F}_{8\pi(1+\beta)}(\omega_{k})=-S(E_k)-8\pi(1+\beta)E_k\leq -C_1,
	$$
	which is the desired contradiction. Therefore necessarily $p=0$ in which case $\rho_{\infty}=8\pi(1+\beta)$, as claimed.
\end{proof}

\medskip

Our first result concerning domains of second kind is a generalization of the analogue characterization in the regular case as
first pursued in \cite{caglioti-1} and then completed in \cite{chang-chen-lin}, \cite{BLin3}, \cite{BMal2}, see also \cite{BLin2}.
Here $D_\beta$ as defined in \eqref{D0} plays a crucial role and in particular we will keep the notations of Lemma \ref{Iexp} whenever needed.

\begin{thm}\label{thm:5.2} 
	The following facts are equivalent:\\
	$(a)$ $(\Omega,\beta)$ is of second kind;\\
	$(b)$ $D_\beta>0$;\\
	$(c)$ $f(8\pi(1+\beta))< -1-\log\left(\frac{\pi}{1+\beta}\right)-\gamma_{\Omega}(0)$;\\
	$(d)$ \eqref{r-equ-flat-lm} has a solution for $\rho=8\pi(1+\beta)$.
\end{thm}

\begin{proof}
	If $(b)$ holds, then, in view of Lemma \ref{Iexp}, $(c)$ holds as well.\\
	Assume that $(c)$ holds, we show that a solution exists for $\rho=8\pi(1+\beta)$. If not, by Lemma \ref{lem:5.1}, any sequence of solutions
	${\rm w}_n={\rm w}_{\rho_n}$ blows up as $\rho_n\nearrow 8\pi(1+\beta)$. However, as remarked above, it is not difficult to prove that
	(see for example Proposition 7.3 in \cite{caglioti-1}) $f(\rho)$ is continuous in $[0,8\pi(1+\beta)]$, therefore
	$$
	J_{\rho_n}({\rm w}_n)=f(\rho_n)\to f(8\pi(1+\beta)),\quad\mbox{as}\quad\rho_n\nearrow 8\pi(1+\beta).
	$$
	In particular, since ${\rm w}_n$ is a 1-point blow-up sequence with blow up point $q=0$, by arguing as in the proof of Theorem \ref{thm:6.1}-$(iv)$  we would find that,
	$$
	f(8\pi(1+\beta))= -1-\log\left(\frac{\pi}{1+\beta}\right)-\gamma_{\Omega}(0),
	$$
	which is the desired contradiction to the strict inequality in $(c)$. 
	
	Since by Lemma \ref{lem:5.1} we have that $(a)$ and $(d)$ are equivalent, we are just left to prove that if $(a)$ and $(d)$ hold then $(b)$ holds as well, that is, $D_\beta>0$. 
	
	By contradiction assume that a solution exists at $\rho=8\pi(1+\beta)$, say ${\rm w}_c$, that $E_c<+\infty$ but $D_\beta\leq 0$.
	By Lemma \ref{lem:entropyest}, as $E_n\to +\infty$ and if ${\rm w}_n=\rho_n G_{\Omega}[\omega_n]$ is a sequence
	defined by the entropy maximizers of Theorem \ref{thm:5.3}, then ${\rm w}_n$ makes 1-point blow-up and
	the blow-up point is in fact the singular point $q=0$.
	Since from \eqref{im-ck} we have that $\rho_n-8\pi(1+\beta)$ has the same sign as $D_\beta$, then $D_\beta\geq 0$. In fact, if otherwise $D_\beta<0$, then we would have
	$\rho_n<8\pi(1+\beta)$ implying that the unique solutions for $\rho_n<8\pi(1+\beta)$
	blow-up, which contradicts Lemma \ref{lem:5.1}. Therefore we have that necessarily $D_\beta=0$. At this point observe that, in view of \eqref{ons-crit}, and with the notations of Lemma \ref{Iexp}, $D_\beta$
	can be written equivalently as follows
	\begin{equation}\label{D0alt}
		D_\beta^*:=c_*^{-1}D_\beta=e^{-\gamma^*(0)}\int_{\Omega}\frac{|x|^{2\beta}}{|x|^{4(1+\beta)}}\left(e^{\gamma^*(x)}-e^{\gamma^*(0)}\right)-\int_{\Omega^c}\frac{|x|^{2\beta}}{|x|^{4(1+\beta)}}=
	\end{equation}
	\[
	\int_{\Omega}\frac{|x|^{2\beta}}{|x|^{4(1+\beta)}}\left(e^{\Phi^*_{\Omega}(x)}-1\right)-\int_{\Omega^c}\frac{|x|^{2\beta}}{|x|^{4(1+\beta)}},
	\]
	where $c_*=e^{\gamma^*(0)}$ and 
	$$
	\Phi^*_{\Omega}(x)=\gamma^*(x)-\gamma^*(0).
	$$
	For any $\epsilon>0$ there exists $\delta>0$ such that 
	$$
	\left|\int_{B_r}\frac{|x|^{2\beta}}{|x|^{4(1+\beta)}}\left(e^{\Phi^*_{\Omega}(x)}-1\right)\right|<\epsilon, \;\forall\,r\leq \delta,
	$$
	and we define 
	\begin{align*}
		D^*_{\beta,\epsilon}&=D^*_\beta-\int_{B_\delta}\frac{|x|^{2\beta}}{|x|^{4(1+\beta)}}\left(e^{\Phi^*_{\Omega}(x)}-1\right)\\
		&=\int_{\Omega\setminus B_\delta}\frac{|x|^{2\beta}}{|x|^{4(1+\beta)}}\left(e^{\Phi^*_{\Omega}(x)}-1\right)-\int_{\Omega^c}\frac{|x|^{2\beta}}{|x|^{4(1+\beta)}}.
	\end{align*}
	For any such $\delta$ fixed,
	let $0<\delta<\tau<R$ such that $B_{2\tau}\Subset \Omega\Subset B_{2R}$ and
	$\mathcal{N}\in C^{4}([-1,1]\times \mathbb{R}^2;\mathbb{R}^2)$ be any vector field such that 
	\begin{align*}
		\mathcal{N}(t,x)=\begin{cases}
			(0,0),\quad &\mbox{if}\quad x\notin \overline{B_{2R}\setminus B_{2\tau}},\\
			-g(x)\nu(x),\quad &\mbox{if}\quad x\in \partial \Omega,
		\end{cases}
	\end{align*}
	where $\nu(x)=\nu_{_\Omega}(x)$ is the unit outer normal at $x$ and $g$ a smooth function. Then it is well
	known (\cite{Hen}) that the unique solution of
	$\frac{d}{dt}h(x,t)=\mathcal{N}(x,t)$, $x\in \Omega$, $t\in (-1,1)$, $h(x,t)=x$, defines for $t$ near $t=0$
	a $C^3$ "curve of domains" $\Omega(t):=h(\Omega,t)$ such that
	$\Omega(0)=\Omega$ and in particular the map
	$$
	t\mapsto D^*_{\beta,\epsilon}(t):=\int_{\Omega(t)\setminus B_\delta}\frac{|x|^{2\beta}}{|x|^{4(1+\beta)}}\left(e^{\Phi^*_{\Omega(t)}(x)}-1\right)-\int_{\Omega^c(t)}\frac{|x|^{2\beta}}{|x|^{4(1+\beta)}}
	$$
	is of class $C^{1}$ near $t=0$ and
	\begin{align*}
		I_g(x;\partial\Omega):=&\left.\frac{d}{dt}\left((8\pi+4\pi\beta)G_{\Omega(t)}(x,0)-4\pi\sum\limits_{j=1}^N\alpha_j
		G_{\Omega(t)}(x,q_j)\right)\right|_{t=0}\\
		=&-(8\pi+4\pi\beta)\int\limits_{\partial \Omega}\frac{\partial G_{\Omega}(x,y)}{\partial \nu(y)} 
		\frac{\partial G_{\Omega}(y,0)}{\partial \nu(y)}\mathcal{N}(y,0)\cdot\nu(y)\\
		&+4\pi\sum\limits_{j=1}^N\alpha_j 
		\int\limits_{\partial \Omega}\frac{\partial G_{\Omega}(x,y)}{\partial \nu(y)} 
		\frac{\partial G_{\Omega}(y,q_j)}{\partial \nu(y)}\mathcal{N}(y,0)\cdot\nu(y)\\
		=&\int\limits_{\partial \Omega}\frac{\partial G_{\Omega}(x,y)}{\partial \nu(y)}\left[
		8\pi(1+\beta) 
		\frac{\partial G_{\Omega}(y,0)}{\partial \nu(y)}-
		4\pi\sum\limits_{j=1}^N\alpha_j 
		\frac{\partial G_{\Omega}(y,q_j)}{\partial \nu(y)}\right]g(y),
	\end{align*}
	where we used the following formula, see \cite[Example 3.4]{Hen}
	\begin{align*}
		\left.\frac{d}{dt}G_{\Omega(t)}(x,p)\right|_{t=0}=-\int_{\Omega}\frac{\partial G_{\Omega}(x,y)}{\partial\nu(y)}\frac{\partial G_{\Omega}(y,p)}{\partial\nu(y)}\mathcal{N}(y,0)\cdot\nu(y)dy,\quad\forall x,~p\in\Omega.
	\end{align*}
	Then
	\begin{equation}
		\label{dop}
		\begin{aligned} 
			&\left.\frac{d}{dt}D^*_{\beta,\epsilon}(t)\right|_{t=0}\\
			&=\int_{\partial \Omega}\frac{|x|^{2\beta}}{|x|^{4(1+\beta)}}\left(e^{\Phi^*_{\Omega(t)}}-1\right)\mathcal{N}(x,0)\cdot\nu(x)\\
			&\quad-\int_{\partial \Omega^c}\frac{|x|^{2\beta}}{|x|^{4(1+\beta)}}\mathcal{N}(x,0)\cdot\nu_c(x)\\
			&\quad+\int_{\Omega\setminus B_\delta}\frac{|x|^{2\beta}e^{\Phi^*_{\Omega(t)}(x)}}{|x|^{4(1+\beta)}}
			\left.\frac{d}{dt}\left((8\pi+4\pi\beta)G_{\Omega(t)}(x,0)-4\pi\sum\limits_{j=1}^N\alpha_j
			G_{\Omega(t)}(x,q_j)\right)\right|_{t=0}\\
			&=\int_{\partial \Omega}\frac{|x|^{2\beta}}{|x|^{4(1+\beta)}}e^{\Phi^*_{\Omega(t)} }\mathcal{N}(x,0)\cdot\nu(x)+
			\int_{\Omega\setminus B_\delta}\frac{|x|^{2\beta}}{|x|^{4(1+\beta)}}e^{\Phi^*_{\Omega(t)} } I_g(x;\partial \Omega)\\  
			&=-\int_{\partial \Omega} \frac{|x|^{2\beta}}{|x|^{4(1+\beta)}}e^{\Phi^*_{\Omega(t)} }g(x)+\int_{\Omega\setminus B_\delta}\frac{|x|^{2\beta}}{|x|^{4(1+\beta)}}e^{\Phi^*_{\Omega(t)} } I_g(x;\partial \Omega)\\ 
			&=-\int_{\partial \Omega} g(x)+\int_{\Omega\setminus B_\delta}\frac{|x|^{2\beta}}{|x|^{4(1+\beta)}}e^{\Phi^*_{\Omega(t)}}I_g(x;\partial \Omega).
		\end{aligned}
	\end{equation}
	Let 
	$$
	D_{\delta}(t):=\int_{B_\delta}\frac{|x|^{2\beta}}{|x|^{4(1+\beta)}}\left(e^{\Phi^*_{\Omega(t)}}-1\right)
	$$
	and remark that  
	$$
	D^*_{\beta}(t):=\int_{\Omega(t)}\frac{|x|^{2\beta}}{|x|^{4(1+\beta)}}\left(e^{\Phi^*_{\Omega(t)}}-1\right)-\int_{\Omega^c(t)}\frac{|x|^{2\beta}}{|x|^{4(1+\beta)}}=
	D^*_{\beta,\epsilon}(t)+D_{\delta}(t),
	$$
	and that
	$$
	|D_{\delta}(t)|<2\epsilon,
	$$
	for any $t$ small enough. 
	\medskip
	
	To proceed our discussion, we need the following claim:
	$$ \mbox{{\bf Claim}: There exists function $g$ such that $D^*_{\beta}(t)<0$ for any $t$ small enough. }$$
	
	\noindent We shall prove the claim by dividing our discussion in various cases. Let
	$$
	G^*(y)=\left[
	(8\pi+4\pi\beta)
	\frac{\partial G_{\Omega}(y,0)}{\partial \nu(y)}-
	4\pi\sum\limits_{j=1}^N\alpha_j 
	\frac{\partial G_{\Omega}(y,q_j)}{\partial \nu(y)}\right],\;y\in \partial\Omega.
	$$
	Assume at first that, 
	$$
	\Gamma_+=\{y\in\partial \Omega\,:\, G^*(y)>0\}\neq \emptyset,
	$$
	then for any $\sigma>0$ we can find
	$g\geq 0$ with compact support in $\Gamma_+$, 
	such that 
	$\int_{\Gamma_+}g(x)>\sigma$,
	which is always possible since $\Gamma_+$ is relatively open in $\partial \Omega$. Therefore, since by the strong maximum principle  $\frac{\partial G_{\Omega}(y,0)}{\partial \nu(y)}<0$ for $y\in \partial \Omega$, we deduce that $I_g(x,\partial \Omega)<0$ and consequently that  
	$$
	\left.\frac{d}{dt}D^*_{\beta,\epsilon}(t)\right|_{t=0}<-\int_{\partial \Omega}  g(x)<-\sigma.
	$$
	Therefore, for any $t$ small enough we have,
	$$
	D^*_{\beta,\epsilon}(t)<
	D^*_{\beta,\epsilon}(0)-\sigma t,
	$$
	and then we choose $\sigma$ large enough such that,
	$$
	D^*_{\beta}(t)<D^*_{\beta,\epsilon}(0)-\sigma t+2\epsilon<0. 
	$$
	Therefore we are left with the case where $G^*(y)\leq 0$ and we assume that
	$$
	\Gamma_-=\{y\in\partial \Omega\,:\, G^*(y)<0\}\neq \emptyset,
	$$
	otherwise the proof is easier. Assume that $g\leq 0$ then $I_g(x,\partial \Omega)\leq 0$ and then, as far as $I_g(0,\partial \Omega)<0$, we have,
	$$
	\int_{\Omega\setminus B_\delta}\frac{|x|^{2\beta}}{|x|^{4(1+\beta)}}e^{\Phi^*_{\Omega}} I_g(x;\partial \Omega)\to -\infty \mbox{ as }\delta \to 0^+.
	$$
	In this case we would deduce from \eqref{dop} that, for any $\sigma>0$,   
	possibly taking a smaller $\delta$, 
	$$
	\left.\frac{d}{dt}D_{0,\epsilon}(t)\right|_{t=0}<-\sigma, 
	$$
	implying as above that $D_{0}(t)<0$ for any $t$ small enough. Therefore we are left with showing that there exists $g\leq 0$ such that $I_g(0,\partial \Omega)\neq 0$. However this an easy consequence of
	$$I_g(0,\partial \Omega)=\int\limits_{\partial \Omega}\frac{\partial G_{\Omega}(0,y)}{\partial \nu(y)}G^*(y)g(y)=
	\int\limits_{\Gamma_-}\frac{\partial G_{\Omega}(0,y)}{\partial \nu(y)}G^*(y)g(y)
	$$
	and the fact that $\frac{\partial G_{\Omega}(0,y)}{\partial \nu(y)}>0$ in 
	$\Gamma_-$. This fact completes the proof of the Claim. 
	
	Therefore, since $\Omega(t)$ is simply connected and of class $C^2$ and since $D^*_\beta (t)$ is just proportional via a positive constant to the $D_\beta$ relative to $\Omega(t)$,
	and for $t>0$ small enough $D^*_\beta(t)<0$, we see from Theorem \ref{thm:2.1} that the unique solutions on $\Omega(t)$
	blow-up as $\rho \nearrow 8\pi(1+\beta)$. Let $t_j\searrow 0^+$ and ${\rm w}_{\rho,j}(x)$ be the unique solutions for $\rho<8\pi(1+\beta)$ on
	$\Omega(t_j)$. By assumption there exists a solution ${\rm w}_c$ on $\Omega$ at $\rho=8\pi(1+\beta)$ which is also unique by Lemma \ref{lem:spectral}-$({\rm iii})$
	and we define $M_c=\max\limits_{\Omega}{\rm w}_c$. Remark that for fixed $j$, ${\rm w}_{\rho,j}$ is a smooth function of $\rho$, see
	Lemma \ref{lem:spectral}-$({\rm iv})$. Therefore for any $M>M_c$ and for any $j$ there exists $\rho_{j,M}\leq 8\pi(1+\beta)$ such that
	$\max\limits_{\Omega(t_j)}{\rm w}_{\rho_{j,M},j}=M$. Thus, along a diagonal subsequence, $\rho_{j,M}\to \rho_{M}\leq 8\pi(1+\beta)$ and
	${\rm w}_{\rho_{j,M},j}\to {\rm w}_{\rho_{M}}$ uniformly, where ${\rm w}_{\rho_{M}}$ is a solution of \eqref{r-equ-flat-lm} for $\rho=\rho_{M}$ on $\Omega$ with
	$\max\limits_{\Omega}{\rm w}_{\rho_{M}}=M$. However, since ${\rm w}_c$ is the unique solution for $\rho=8\pi(1+\beta)$, then
	$\rho_M\in (0,8\pi(1+\beta))$. Since $M$ is arbitrary
	we can take $M_k\to +\infty$ and construct a sequence of solutions of \eqref{r-equ-flat-lm} which blows up with $\rho_{M_k}<8\pi(1+\beta)$, implying
	necessarily (\cite{BM3}) that $\rho_{M_k}\to 8\pi(1+\beta)$. This contradicts Lemma \ref{lem:5.1} which would imply that $(\Omega, \beta)$ is of first kind. \end{proof}

\medskip

As an immediate consequence of Theorem \ref{thm:5.2}, Proposition \ref{est-muk} and \eqref{im-ck}, we have the following generalization of a result in \cite{BLin2}.
\begin{cor}\label{last:thm.1}
	Let ${\rm w}_{k}$ be a sequence of solutions of \eqref{r-equ-flat-lm} where $q=0$ satisfies \eqref{ons-crit}, which makes one point blow-up at $q=0$ as $\rho_k\to 8\pi(1+\beta)$.\\ If $|\rho_k-8\pi(1+\beta)|=o(e^{-\lambda_k})$ then 
	$\rho_k<8\pi(1+\beta)$.
\end{cor}

\begin{rem} 
	As far as $\beta=0$, Corollary \ref{last:thm.1} is false on the flat two-torus, see \cite{lin-Lwang}.
\end{rem}

Even with Theorem \ref{thm:5.2} at hand is not easy to find explicit examples of domains of pairs $(\Omega,\beta)$ of second kind. It is an interesting open problem to check whether or not symmetric dumbbell type domains $\Omega$ with a negative pole sitting at the saddle point of the Robin function can be constructed to match the requirement  $D_\beta>0$.
Another possibility could be that of dropping \eqref{ons-crit}, which, although not very meaningful from the physical point of view, seems to be at hand for 
$\beta\in (-1,-\frac12)$. We will discuss these issues elsewhere. In any case, as an immediate consequence of Theorem \ref{thm:5.2} and Lemma \ref{lem:5.1}, we have the carachterization of pairs $(\Omega, \beta)$ of first kind.
\begin{thm}\label{thm:5.2.0} 
	The following facts are equivalent:\\
	$(a)$ $(\Omega,\beta)$ is of first kind;\\
	$(b)$ $D_\beta\leq 0$;\\
	$(c)$ $f(8\pi(1+\beta))= -1-\log\left(\frac{\pi}{1+\beta}\right)-\gamma_{\Omega}(0)$;\\
	$(d)$ \eqref{r-equ-flat-lm} has no solution for $\rho=8\pi(1+\beta)$.
\end{thm}

Among other things, for $(\Omega, \beta)$ of second kind, the following Theorem  shows that, locally near $(8\pi(1+\beta))^+$,
large energy solutions are unique and define a smooth curve.

\begin{thm}\label{lem:convexentropy}
	If $(\Omega, \beta)$ is of second kind, there exists $\varepsilon_* > 0$ such that:\\
	$(i)$ there exists one
	and only one solution ${\rm w}_{\rho}$ of \eqref{r-equ-flat-lm} with $\rho \in (8\pi(1+\beta), 8\pi(1+\beta)+\varepsilon_*)$ and
	$\mathcal{E}\left(\frac{He^{{\rm w}_{\rho}}}{\int_\Omega He^{{\rm w}_{\rho}}dx}\right) >(\varepsilon_*)^{-1}$;\\
	$(ii)$ Let $\mathcal{G}_{\infty}$ denote the set of solutions of \eqref{r-equ-flat-lm} determined in $(i)$. Then, for any fixed
	$p\in (1, |\beta|^{-1})$, $\mathcal{G}_{\infty}$ is a real analytic curve
	$$
	(8\pi(1+\beta), 8\pi(1+\beta)+\varepsilon_*)\ni \rho \mapsto {\rm w}_{\rho}\in W_0^{2,p}(\Omega);\\
	$$
	$(iii)$ along $\mathcal{G}_{\infty}$ we have that the map $E(\rho)=\mathcal{E}\left(\frac{He^{{\rm w}_{\rho}}}{\int_\Omega He^{{\rm w}_{\rho}}dx}\right) $ is real analytic and strictly decreasing with inverse
	$E\mapsto \rho_{E}$, $E\in ((\varepsilon_*)^{-1},+\infty)$ as determined by Theorem \ref{thm:5.3}. In particular $\frac{dE(\rho)}{d\rho}\leq 0$.\\
	$(iv)$ along $\mathcal{G}_{\infty}$ we have that the map $\mu^2(\rho)=\rho\left(\int_{\Omega}H e^{{\rm w}_{\rho}}\right)^{-1}$ is
	real analytic and strictly increasing in $(8\pi(1+\beta),8\pi(1+\beta)+\varepsilon_*)$ with $\mu(\rho)\to 0^+$ as $\rho\to 8\pi(1+\beta)^+$.
	In particular the inverse map
	$\mu\mapsto \rho(\mu)$ is well defined and strictly increasing in $(0,\mu_*)$, where $\mu_*=\lim_{\rho \to 8\pi(1+\beta)+\varepsilon_*}\mu(\rho)$.
\end{thm}
\begin{proof}
	$(i)$ In view of Theorem \ref{thm:5.2} and \eqref{im-ck}, if this was not the case we could find sequences ${\rm w}_{k,1}\neq {\rm w}_{k,2}$ solving \eqref{r-equ-flat-lm} with the same value
	$\rho_k\to 8\pi(1+\beta)^+$ and 
	$$\mathcal{E}\left(\frac{He^{{\rm w}_{k,2}}}{\int_\Omega He^{{\rm w}_{k,2}}dx}\right) =E_{k,2}\geq E_{k,1}=\mathcal{E}\left(\frac{He^{{\rm w}_{k,1}}}{\int_\Omega He^{{\rm w}_{k,1}}dx}\right) \to +\infty.$$ 
	It is readily seen
	that both
	${\rm w}_{k,i}$, $i=1,2$ make $1$-point blow-up at the negative singular source, which for $k$ large contradicts the local uniqueness result in
	\cite{byz-1}. This observation proves $(i)$.\\
	$(ii)$ The claim about $\mathcal{G}_{\infty}$ in $(ii)$ follows immediately from Theorem \ref{main-theorem-1}, Lemma \ref{lem:ift} and $(i)$.\\
	$(iii)$ We first prove that, by taking a smaller $\varepsilon_*$ if needed, the map $E\mapsto \rho_{_E}$ as determined by Theorem \ref{thm:5.3}
	for solutions on $\mathcal{G}_{\infty}$ is injective. If not,  there exist sequences $E_{k,1}\geq E_{k,2}\nearrow +\infty$ such that there exists
	$\rho_{k}\searrow 8\pi(1+\beta)$ such that $\rho_{k}=\rho_{E_{k,1}}=\rho_{E_{k,1}}$. Let ${\rm w}_{k,1},{\rm w}_{k,2}$ be the solutions
	of \eqref{r-equ-flat-lm} with energies $E_{k,1}\geq E_{k,2}$, by the uniqueness result in
	\cite{byz-1} for $k$ large we have that ${\rm w}_{k,1}\equiv{\rm w}_{k,2}$, implying that $E_{k,1}=E_{k,2}$, which is the desired contradiction.\\
	Therefore, since the map $E\mapsto \rho_{_E}$ is continuous and injective in $(\frac{1}{\varepsilon_*},+\infty)$ then it must be strictly monotone
	and since by Lemma \ref{lem:5.1} $\rho_{_E}\to 8\pi(1+\beta)^{+}$ as $E\to +\infty$ then it must be strictly decreasing. Next observe that
	by $(ii)$ the inverse map $\rho\mapsto E(\rho)$ is well defined and real analytic, whence we also have $\frac{dE(\rho)}{d\rho}\leq 0$.\\
	$(iv)$ The fact that $\mu(\rho)$ is well defined and real analytic along $\mathcal{G}_{\infty}$ obviously follows from $(ii)$.
	Let $\mu\mapsto \rho(\mu)$ be defined as follows. Let $u_{\mu}$ be a solution of
	\begin{equation}\label{umu}
		\begin{cases}
			-\Delta u_\mu=\mu^2 H e^{\displaystyle {u}_{\mu}}\quad &\mbox{in}~ \Omega,\\ \\
			u_\mu=0 \quad&\mbox{on}~ \partial \Omega,
		\end{cases}
	\end{equation}
	and $\mathcal{U}_\mu$ be the set of all solutions of \eqref{umu}. We define the multivalued function
	$$
	\Lambda(\mu)=\left\{\mu^2 \int_{\Omega}H e^{\displaystyle {u}_{\mu}}, u_\mu\in \mathcal{U}_\mu\right\}.
	$$
	Let $(\rho, {\rm w}_{\rho})\in \mathcal{G}_{\infty}$, then setting
	$\mu:=\mu(\rho)=\left(\rho\left(\int_{\Omega}H e^{{\rm w}_{\rho}}\right)^{-1}\right)^\frac12$ and $u_{\mu}:={\rm w}_{\rho}$, then $u_{\mu}$ solves \eqref{umu} for
	this $\mu$, whence in particular
	$\rho(\mu)=\mu^2 \int_{\Omega}H e^{{u}_{\mu}}\in \Lambda(\mu)$. We prove that, possibly taking a smaller $\varepsilon_*$,
	this argument defines an injective function $\mu\mapsto \rho(\mu)$, $\mu \in (0,\mu_*)$ with $\mu_*$ as defined in the claim. If not
	there exists a sequence $\mu_k\to 0^+$ such that we could find $\rho_{k,1}\neq \rho_{k,2}$ sharing the same value of $\mu_k$ such that
	$(\rho_{k,i},{\rm w}_{\rho_{k,i}})\in \mathcal{G}_{\infty}$, $i=1,2$  and $8\pi(1+\beta)<\rho_{k,1}<\rho_{k,2}\to 8\pi(1+\beta)^+$.
	Since both solutions belong to $\mathcal{G}_{\infty}$ they both have to blow-up at the negative singular source which for
	$k$ large contradicts the local uniqueness result in \cite{wyang}. This shows that $\mu\mapsto \rho(\mu)$ is a well defined function. However it is also
	injective because otherwise we would have that there exists $0<\mu_{k,1}\leq\mu_{k,2}\to 0^+$ corresponding to the same $\rho_k:=\rho(\mu_{k,1})=\rho(\mu_{k,2})$, while the solutions  $u_{k,1}\neq u_{k,2}$. This is again in contradiction with the
	uniqueness result in \cite{byz-1} which implies ${\rm w}_{k,1}:=u_{k,1}\equiv u_{k,2}=:{\rm w}_{k,2}$.
\end{proof}

With the above Theorem, we are able to deduce a generalization of a recent result about the convexity of the entropy for domains of second kind in the regular case, see Theorem 1.4 in \cite{bart-5}.
\begin{thm}\label{convex:entropy}{\rm [Negative Specific Heat States]}
	If $(\Omega, \beta)$ is of second kind, then there exists $E_*>E_c$ such that $S(E)$ is smooth and strictly convex in $(E_*,+\infty)$.
	In particular
	$$
	\frac{dS(E)}{dE}=-\rho_{E},\quad \frac{d^2 S(E)}{dE^2}=-\frac{d\rho_{E}}{dE}>0,
	$$
	where $\rho_{E}$ is smooth and strictly decreasing in $(E_*,+\infty)$ and $\rho_{E}\searrow 8\pi(1+\beta)$ as $E\nearrow +\infty$.
\end{thm}
\begin{proof}
	By Theorem \ref{lem:convexentropy} we can adapt an argument in \cite{bart-5}.
	Let $\mu\in (0,\mu_*)$, $\rho(\mu)$ as defined in Lemma \ref{lem:convexentropy}-$(iv)$ and $u_{\mu}\equiv {\rm w}_{\rho(\mu)}$
	the corresponding unique solution of \eqref{umu} in $\mathcal{G}_{\infty}$. By the non degeneracy of $u_\mu$ (\cite{wyang})
	we have that $v_{\mu}=\frac{d u_{\mu}}{d\mu}$ is well defined and satisfies,
	\begin{equation}\label{vmu}
		\begin{cases}
			-\Delta v_{\mu}=2\mu H e^{\displaystyle {u}_{\mu}}+\mu^2 He^{\displaystyle {u}_{\mu}}v_{\mu}\quad &\mbox{in}~\Omega\\ \\
			v_{\mu}=0 \quad&\mbox{on}~ \partial \Omega
		\end{cases}
	\end{equation}
	and since by definition $\rho(\mu)=\mu^2 \int_{\Omega} H e^{\displaystyle {u}_{\mu}}$, then
	\begin{equation}\label{lmu}
		\frac{d \rho({\mu})}{d\mu}=2\mu \int_{\Omega} H e^{\displaystyle {u}_{\mu}}+\mu^2\int_{\Omega} H e^{\displaystyle {u}_{\mu}}v_{\mu}.
	\end{equation}
	Remark that by definition $E(\rho)=\frac{1}{2\rho}\int_{\Omega}\frac{He^{{\rm w}_\rho}{\rm w}_\rho}{\int_\Omega He^{{\rm w}_\rho}}$,
	whence in particular
	$$
	E_{\mu}:=E(\rho(\mu))=\frac{\mu^2}{2\rho^2(\mu)}\int_{\Omega} H e^{\displaystyle {u}_{\mu}}u_{\mu}.
	$$
	Multiplying the equation in \eqref{vmu} by $u_\mu$, integrating by parts and using \eqref{umu} we have
	$$
	\mu^2\int_{\Omega} H e^{\displaystyle {u}_{\mu}}v_{\mu}=2\mu \int_{\Omega} H e^{\displaystyle {u}_{\mu}}u_{\mu}+
	\mu^2\int_{\Omega} H e^{\displaystyle {u}_{\mu}}v_{\mu}u_{\mu},
	$$
	which readily implies that
	\begin{equation}\label{emu}
		\frac{d}{d\mu}\left(\mu^2\int_{\Omega} H e^{\displaystyle {u}_{\mu}}u_{\mu}\right)=2\mu^2\int_{\Omega} H e^{\displaystyle {u}_{\mu}}v_{\mu}.
	\end{equation}
	We deduce from \eqref{lmu} and \eqref{emu} that
	\begin{align*}
		\rho^4(\mu)\frac{d E_\mu}{d\mu}=&~\frac{\rho^2(\mu)}2 \frac{d}{d\mu}\left(\mu^2\int_{\Omega} H e^{\displaystyle {u}_{\mu}}u_{\mu}\right)-
		\rho(\mu)\frac{d \rho(\mu)}{d\mu}\mu^2\int_{\Omega} H e^{\displaystyle {u}_{\mu}}u_{\mu}\\
		=&\left(\mu^2\int_{\Omega} H e^{\displaystyle {u}_{\mu}}v_{\mu}\right)\rho^2(\mu)-
		\rho(\mu)\frac{d \rho(\mu)}{d\mu}\mu^2\int_{\Omega} H e^{\displaystyle {u}_{\mu}}u_{\mu}\\
		=&\left(\frac{d\rho(\mu)}{d\mu}-\frac{2\rho(\mu)}{\mu}\right)\rho^2(\mu)-
		\rho(\mu)\frac{d \rho(\mu)}{d\mu}\mu^2\int_{\Omega} H e^{\displaystyle {u}_{\mu}}u_{\mu}\\
		=&~\frac{d\rho(\mu)}{d\mu}\left(\rho^2(\mu)-\rho(\mu)\mu^2\int_{\Omega} H e^{\displaystyle {u}_{\mu}}u_{\mu}\right)-
		2\frac{\rho^3(\mu)}{\mu}.
	\end{align*}
	At this point, by using the fact that $\frac{d \rho(\mu)}{d\mu}\geq 0$,
	$\mu^2\int_{\Omega} H e^{\displaystyle {u}_{\mu}}u_{\mu}=2\rho^2(\mu)E_\mu\to +\infty$, $\rho(\mu)\to 8\pi(1+\beta)$ and
	$\mu\to 0^+$ we deduce that $\frac{d E_\mu}{d\mu}\to -\infty$. Observe moreover that \eqref{lmu} and \eqref{emu} together imply that
	$$
	\frac{d \rho({\mu})}{d\mu}=2\mu \int_{\Omega} H e^{\displaystyle {u}_{\mu}}+\frac12\frac{d}{d\mu}\left(\mu^2\int_{\Omega} H e^{\displaystyle {u}_{\mu}}u_{\mu}\right)=
	2\mu \int_{\Omega} H e^{\displaystyle {u}_{\mu}}+\frac{d}{d\mu}(\rho^2(\mu)E_{\mu}),
	$$
	that is
	$$
	\left(1-\frac{d}{d\rho}(\rho^2 E(\rho))\right)\frac{d\rho(\mu)}{d\mu}=2\mu\int_{\Omega} H e^{\displaystyle {u}_{\mu}},
	$$
	which shows, since $\frac{d\rho(\mu)}{d\mu}\geq 0$ and since from $(iii)$
	$\frac{dE(\rho)}{d\rho}$ is well defined, that in fact $\frac{d\rho(\mu)}{d\mu}>0$. Remark also that, because of the nondegeneracy of $u_\mu$ (\cite{wyang}) we see
	from \eqref{lmu} that $\frac{d\rho(\mu)}{d\mu}$ is bounded, whence we have
	$\frac{dE(\rho)}{d\rho}=\frac{d E_\mu}{d\mu}\frac{d\mu(\rho)}{d\rho}<0$. At this point the conclusion follows
	from the fact that $\frac{dS(E)}{dE}=-\rho_{_E}$ which is done as in (4.1) in \cite{Bons}. \end{proof}

Finally, by Lemma \ref{lem:spectral} and Theorem  \ref{lem:convexentropy}, we provide the exact counting of solutions in a right neighborhood of $8\pi(1+\beta)$ for domains of second kind. 

\begin{thm}\label{thm:count1} If $(\Omega, \beta)$ is of second kind, there exists $\varepsilon_* > 0$ such that there exists exactly two solutions of 
	of \eqref{r-equ-flat-lm} with $\rho \in (8\pi(1+\beta), 8\pi(1+\beta)+\varepsilon_*)$. 
\end{thm}
\begin{proof} By Lemma \ref{lem:spectral}-${\rm (vi)}$ the branch of unique solutions for $\rho=8\pi(1+\beta)$ can be continued in a small enough right neighborhood of $\rho=8\pi(1+\beta)$, say $I_*=(8\pi(1+\beta),8\pi(1+\beta)+\varepsilon_*)$. By Lemma \ref{lem:convexentropy}-$(i)$-$(ii)$, possibly taking a smaller $\varepsilon_*$, there exists a smooth branch of solutions in $I_*$ whose energy is larger than $(\varepsilon_*)^{-1}$. If the claim were false we could find a third distinct sequence of solutions, say ${\rm w}_k$ for some $\rho_k\to 8\pi(1+\beta)^+$. Possibly along a subsequence there are only two possibilities, either ${\rm w}_k\to {\rm w}_\infty$, where ${\rm w}_\infty$ is the unique solution at $8\pi(1+\beta)$ or $\|{\rm w}_k\|_\infty \to + \infty$.
	The former case is ruled out since then $\rho=8\pi(1+\beta)$ would be a bifurcation point, which contradicts Lemma \ref{lem:ift} and Lemma \ref{lem:spectral}-${\rm (iii)}$. 
	The latter case is ruled out by the local uniqueness Theorem in \cite{byz-1}.\end{proof}


\begin{thebibliography}{99}
	
	\bibitem{ads} Adimurthi, K. Sandeep, A singular Moser-Trudinger embedding and its applications. \emph{NoDEA} \textbf{13} (2007), 585-603.
	
	\bibitem{ambjorn} J. Ambjorn, P. Olesen, Anti-screening of large magnetic fields by vector bosons. \emph{Phys. Lett. B}, \textbf{214} (1988), no. 4, 565-569.
	
	
	\bibitem{Bons} D. Bartolucci, {Global bifurcation analysis of mean field equations and
		the Onsager microcanonical description of two-dimensional turbulence}, \emph{Calc. Var. \& P.D.E.}, {\bf 58}:18 (2019).
	
	\bibitem{BCWYZ} D. Bartolucci,
	P. Cosentino, L. Wu, W. Yang, L. Zhang, work in progress 
	
	\bibitem{BCLT} D. Bartolucci, C.C. Chen, C.S. Lin, G. Tarantello, Profile of blow-up solutions to mean field equations with singular data.
	\emph{Comm. P.D.E.} \textbf{29} (2004),  no. 7-8, 1241-1265.
	
	\bibitem{BdMM} D. Bartolucci, F. De Marchis, A. Malchiodi, {Supercritical conformal metrics on
		surfaces with conical singularities}, {\em Int. Math. Res. Not.}, {\bf (24)} (2011), 5625-5643.
	
	\bibitem{bart-5} D. Bartolucci, A. Jevnikar, Y. Lee, W. Yang, Non degeneracy, Mean Field Equations and the
	Onsager theory of 2D turbulence, \emph{Arch. Ration. Mech. An.}  \textbf{230} (2018), 397-426.
	
	\bibitem{bart-4} D. Bartolucci, A. Jevnikar, Y. Lee, W. Yang, Uniqueness of bubbling solutions of mean field equations. \emph{J. Math. Pures Appl.},  \textbf{123} (2019), 78-126.
	
	\bibitem{bart-4-2} D. Bartolucci, A. Jevnikar, Y. Lee, W. Yang, Local Uniqueness of Blow-up solutions of mean field equations with singular data.
	\emph{Jour. Diff. Eqs.}, \textbf{269} (2020), 2057-2090.
	
	
	\bibitem{bj-lin}
	D. Bartolucci, A. Jevnikar, C.S. Lin, Non-degeneracy and uniqueness of solutions to singular mean field equations on
	bounded domains. {\em Jour. Diff. Eqs.} \textbf{266} (2019) 716-741.
	
	\bibitem{BLin2} D. Bartolucci, C.S. Lin, { Sharp existence results for mean field equations with singular data}. \emph{Jour. Diff. Eq.} \textbf{252} (2012), 4115-4137.
	
	\bibitem{BLin3} D. Bartolucci, C.S. Lin, { Existence and uniqueness for
		Mean Field Equations on multiply connected domains at the critical parameter}.
	{\em Math. Ann.}, {\bf 359} (2014), 1-44; DOI 10.1007/s00208-013-0990-6.
	
	\bibitem{BLT} D. Bartolucci, C.S. Lin, G. Tarantello, { Uniqueness and symmetry results for
		solutions of a mean field equation on ${\mathbb{S}}^{2}$ via a new bubbling phenomenon}.
	{\em Comm. Pure Appl. Math.} {\bf 64}(12) (2011), 1677-1730.
	
	\bibitem{BMal} D. Bartolucci, A. Malchiodi, {An improved geometric
		inequality via vanishing moments, with applications to singular
		Liouville equations}, {\em Comm. Math. Phys.} {\bf 322} (2013), 415-452.
	
	\bibitem{BMal2}
	D. Bartolucci, A. Malchiodi, {Mean field equations and domains of first kind}, {\em Rev. Mat. Iberoam.} {\bf 38} (2022), 1067-1086.
	
	\bibitem{BM3} D. Bartolucci, E. Montefusco, { Blow-up analysis,
		existence and qualitative properties of solutions for the two
		dimensional Emden-Fowler equation with singular potential}, {\em M$^{2}$.A.S. }
	{\bf 30}(18) (2007), 2309-2327.
	
	\bibitem{BT} D. Bartolucci, G. Tarantello, Liouville type
	equations with singular data and their applications to periodic
	multivortices for the electroweak theory. {\em Comm. Math. Phys.} 229
	(2002), 3-47.
	
	\bibitem{BT-2}  D. Bartolucci, G. Tarantello, Asymptotic blow-up analysis for singular Liouville type equations with applications. \emph{J. Differential Equations}, \textbf{262} (2017), 3887-3931.
	
	
	\bibitem{byz-1} D. Bartolucci, W. Yang, L. Zhang, Asymptotic Analysis and Uniqueness of blow-up solutions of non-quantized singular mean field equations,
	arXiv:2401.12057.
	
	\bibitem{luca-b} L. Battaglia, M. Grossi, A. Pistoia, Non-uniqueness of blowing-up solutions to the Gel'fand problem. {\em Calc. Var. \& P.D.E.} \textbf{58} (2019), art. 163
	
	\bibitem{BCN} D. Benedetto, E. Caglioti, M. Nolasco, Microcanonical phase transitions for the vortex system. \emph{Math. Mech. Compl. Syst. } \textbf{12}
	(2024), 85-112.
	
	\bibitem{BM} H. Brezis, F. Merle, Uniform estimates and blow-up behavior
	for solutions of $-\Delta u=v(x)e^u$ in two dimensions. {\em Comm. P. D. E.}, 16 (1991)
	1223-1253.
	
	\bibitem{but} B. Buffoni, J. Toland, {Analytic Theory of Global Bifurcation}, (2003) Princeton Univ. Press.
	
	
	\bibitem{caglioti-1} E. Caglioti, P.L. Lions, C. Marchioro , M. Pulvirenti,
	A special class of stationary flows for two-dimensional Euler equations: A statistical mechanics description, \emph{Comm. Math. Phys.}, \textbf{143} (1992), 501--525.
	
	\bibitem{caglioti-2} E. Caglioti, P.L. Lions, C. Marchioro , M. Pulvirenti, A special class of stationary flows for two-dimensional Euler equations: A statistical mechanics description, part II. \emph{Comm. Math. Phys.}, \textbf{174} (1995), 229--260.
	
	\bibitem{chai} C.C. Chai, C.S. Lin, C.L.Wang, Mean field equations, hyperelliptic curves, and modular forms: I, \emph{Camb. J. Math.}. \textbf{3}(1-2) (2015), 127-274.
	
	\bibitem{chan-fu-lin} H. Chan, C.C. Fu, C.S. Lin, Non-topological multi-vortex solutions to the self-dual Chern-Simons-Higgs equation, \emph{Comm. Math. Phys.}, \textbf{231} (2002), no. 2, 189-221.
	
	\bibitem{chang-chen-lin} A. Chang, C.C. Chen, C.S. Lin, Extremal functions for a mean field equation in two dimension. New Stud. Adv. Math., 2
	International Press, Somerville, MA, 2003, 61–93.
	
	\bibitem{CL1} W.X. Chen, C.M. Li, { Classification of solutions of some nonlinear elliptic equations,}
	{\em Duke Math. J.}  {\bf 63}(3) (1991), 615-622.
	
	\bibitem{CL2} W.X. Chen, C.M. Li, { Qualitative properties of solutions of
		some nonlinear elliptic equations in $R^2$}, {\em Duke Math. J.} 71(2) (1993), 427-439.
	
	\bibitem{chen-lin-sharp} C.C. Chen, C.S. Lin, Sharp estimates for solutions of multi-bubbles in compact Riemann surface. \emph{Comm. Pure Appl. Math.}, \textbf{55} (2002), 728-771.
	
	\bibitem{chen-lin-deg} C.C. Chen, C.S. Lin, Topological degree for a mean field equation on Riemann surfaces. \emph{Comm. Pure Appl. Math.}, \textbf{56} (2003), 1667-1727.
	
	\bibitem{chen-lin-wang} C.C. Chen, C.S. Lin, G.Wang, Concentration phenomena of two-vortex solutions in a Chern-Simons model. \emph{Ann. Sc. Norm. Super. Pisa Cl. Sci.}, (5) \textbf{3} (2004), 2, 367397.
	
	\bibitem{chen-lin} C.C. Chen, C.S. Lin, Mean field equation of Liouville type with singularity data: Sharper estimates, \emph{Discrete and Continuous Dynamic Systems-A}, \textbf{28} (2010), 1237-1272
	
	\bibitem{chen-lin-deg-2} C.C. Chen, C.S. Lin, Mean field equation of Liouville type with singular data: topological degree. \emph{Comm. Pure Appl. Math.}, \textbf{68} (2015), 6, 887-947.
	
	\bibitem{Zchen-lin-1}
	Z. Chen, C.S. Lin, Critical points of the classical Eisenstein
	series of weight two. \emph{J. Differential Geom.}
	\textbf{113} (2019), 189-226.
	
	
	\bibitem{csw} {M. Chipot, I. Shafrir, G. Wolansky}, \emph{On the Solutions of Liouville Systems}, {\em J. Differential Equations}, 140 (1997), 59-105.
	
	
	
	\bibitem{CCKW} A. Constantin, D. G. Crowdy,  V. S. Krishnamurthy, M. H. Wheeler, \emph{Liouville chains: new hybrid vortex equilibria of
		the two-dimensional Euler equation}, {\em J. Fluid Mech.} {\bf 921} (2021), A1.
	
	
	\bibitem{dAWZ} T. D'Aprile, J. Wei, L. Zhang,  {On the construction of non-simple blow-up solutions for the singular Liouville equation with a potential}, {\em Calc. Var. \& P.D.E.} to appear.
	
	\bibitem{DKM} M. del Pino, M. Kowalczyk, M. Musso,  {Singular limits in
		Liouville-type equations}, {\em Calc. Var. \& P.D.E.} {\bf 24} (2005), 47-81.
	
	\bibitem{DJ} Z. Djadli, {Existence result for the mean field problem on Riemann surfaces of all genuses}, {\em Comm. Contemp. Math.} \textbf{10} (2008), 205-220.
	
	\bibitem{druet} O. Druet, {Elliptic equations with critical Sobolev exponents in dimension 3}, {\em A. I. H. P.:  Anal. Non Lineaire} \textbf{2} (2002), 125-142.
	
	\bibitem{EGP} P. Esposito, M. Grossi \& A. Pistoia, {On the existence of blowing-up solutions for a mean field equation}, {\em Ann. Inst. H. Poincare Anal. Nonlinear }, {\bf 22} (2005), 227-257.
	
	
	\bibitem{wyang} M. Gao, A. Jevnikar, Y.X. Wang, W. Yang, On the local uniqueness and non-degeneracy of the bubbling solutions to Gel'fand equation with arbitrary singularities, preprint 2024.
	
	\bibitem{gu-zhang-1} Y. Gu, L. Zhang, Degree counting theorems for singular Liouville systems. \emph{ Ann.
		Sc. Norm. Super. Pisa Cl. Sci.} (5) Vol. XXI (2020), 1103-1135.
	
	\bibitem{gu-zhang-2} Y. Gu, L. Zhang, Structure of bubbling solutions of Liouville systems with negative singular
	sources, preprint 2022. https://arxiv.org/abs/2112.10031.
	
	\bibitem{gluck} M. Gluck, Asymptotic behavior of blow-up solutions to a class of prescribing Gauss curvature equations. \emph{Nonlinear Anal.},  {75} (2012), 5787-5796.
	
	\bibitem{GM1} C. Gui, A. Moradifam, { The Sphere Covering Inequality and Its Applications}, {\em Invent. Math.}  {214, (2018) 1169-1204}.
	
	\bibitem{Hen} D. Henry, {Perturbation of the boundary in boundary-value problems of partial differential equations}, London Math. Soc.
	L.N.S. {\bf (318)} Cambridge University Press, Cambridge, (2005).
	
	
	
	
	\bibitem{KW} J. L. Kazdan, F. W. Warner,
	{Curvature functions for compact 2-manifolds}, {\em Ann. Math.} {\bf 99}  (1974), 14-74.
	
	\bibitem{kiessling} Kiessling, Michael K.-H., Statistical mechanics of classical particles with logarithmic interactions. {\em Comm. Pure Appl. Math.} 46(1993), no.1, 27–56.
	
	\bibitem{kuo-lin} T.J. Kuo, C.S. Lin, Estimates of the mean field equations with integer singular sources: non-simple blow-up, \emph{Jour. Diff. Geom.}, \textbf{103} (2016), 377-424.
	
	\bibitem{lee-lin-jfa} Y. Lee, C.S. Lin, Uniqueness of bubbling solutions with collapsing singularities. \emph{J. Funct. Anal.}, 277 (2019), no. 2, 522–557.
	
	\bibitem{li-cmp} Y.Y. Li, Harnack type inequality: the method of moving planes, \emph{Comm. Math. Phys.}, \textbf{200} (1999), 421--444.
	
	\bibitem{ls} Y.Y. Li, I.Shafrir, {Blow-up analysis for Solutions of $-\Delta u = V(x)e^{u}$
		in dimension two}, {\em Ind. Univ. Math. J.}  {\bf 43} (1994), 1255-1270.
	
	
	\bibitem{Lin1} C.S. Lin, {Uniqueness of solutions to the mean field equation for the
		spherical Onsager Vortex}, \emph{Arch. Rat. Mech. An.} {\bf 153} (2000), 153-176.
	
	
	\bibitem{lin22}  C.S. Lin, Spherical metrics with one singularity and odd integer angle on flat tori I, preprint 2022.
	
	\bibitem{lin-yan-uniq}  C.S. Lin, S.S. Yan, On the mean field type bubbling solutions for Chern-Simons-Higgs equation. \emph{Adv. Math.}, \textbf{338} (2018), 1141--1188.
	
	\bibitem{lin-yan-cs} C.S. Lin, S.S. Yan,
	Existence of bubbling solutions for Chern-Simons model on a torus. {\em Arch. Ration. Mech. An.} \textbf{207} (2013), 353–392.
	
	\bibitem{lin-Lwang} C.S. Lin,
	C.L. Wang. Elliptic functions, Green functions
	and the mean field equations on tori.  \emph{Ann. Math.}, \textbf{172} (2010), 911-954.
	
	
	
	\bibitem{ma-wei} L. Ma, J.C. Wei, Convergence for a Liouville equation. \emph{Comment. Math. Helv.}, 76 (2001), no. 3, 506–514.
	
	\bibitem{Mal1} A. Malchiodi, {Topological methods for an elliptic equation with exponential nonlinearities}, \emph{Discr. Cont. Dyn. Syst.}
	{\bf 21} (2008), 277-294.
	
	\bibitem{Mal2} A. Malchiodi, {Morse theory and a scalar field equation on compact
		surfaces}, {\em Adv. Diff. Eq.} {\bf 13} (2008), 1109-1129.
	
	\bibitem{mal-ruiz} A. Malchiodi, D. Ruiz,
	New improved Moser-Trudinger inequalities and singular Liouville equations on compact surfaces.
	\emph{Geom. Funct. Anal.}, 21 (2011), no. 5, 1196–1217.
	
	\bibitem{moser} J. Moser,  A sharp form of an inequality by N.Trudinger, \emph{Indiana Univ. Math. J.} 20 (1971), 1077-1091.
	
	\bibitem{nolasco-taran} M. Nolasco, G. Tarantello, On a sharp Sobolev-type Inequality on two-dimensional compact manifold, \emph{Arch. Rat. Mech. An.}, \textbf{145} (1998), 161-195.
	
	\bibitem{On}  L. Onsager, { Statistical hydrodynamics}, \emph{Nuovo Cimento} {\bf 6}(2) (1949), 279-287.
	
	\bibitem{pot} A. Poliakovsky, G. Tarantello, {On a planar Liouville-type problem in the study of
		selfgravitating strings}, \emph{J. Differential Equations} {\textbf{252}} (2012), 3668-3693.
	
	\bibitem{Parjapat-Tarantello} J. Prajapat, G. Tarantello, On a class of elliptic problems in $\mathbb{R}^2$: Symmetry and uniqueness results, \emph{Proc. Roy. Soc. Edinburgh Sect. A}, \textbf{131} (2001), 967-985.
	
	\bibitem{SO} R.A. Smith, T.M. O'Neil,
	nonaxisymmetric thermal equilibria of cylindrically bounded guiding-center plasmas or discrete vortex system, \emph{Phys. Fluids B}, \textbf{2} (1990), 2961–2975.
	
	\bibitem{spruck-yang} J. Spruck, Y. Yang, On Multivortices in the Electroweak Theory I:Existence of Periodic Solutions, \emph{Comm. Math. Phys.}, \textbf{144} (1992), 1-16.
	
	
	\bibitem{stam} G.  Stampacchia, {\em Le probl{\`e}me de Dirichlet pour les {\'e}quations elliptiques
		du second ordre {\`a} coefficients discontinus},
	{\em  Ann. Ins. Fourier} {\bf 15}(1) (1965), pp. 189-257.
	
	\bibitem{stam2} G.  Stampacchia, {\em Some limit cases of $L^p$ -estimates for solutions of second order
		elliptic equations}, Comm. Pure Appl. Math. {\bf 16} (1963), 505-510.
	
	\bibitem{suzuki} T. Suzuki, Global analysis for a two-dimensional elliptic eiqenvalue problem with the exponential nonlinearly, \emph{Ann. Inst. H. Poincare Anal. Nonlinear }, \textbf{9}(4) (1992), 367-398.
	
	\bibitem{T3} G. Tarantello,
	{Analytical aspects of Liouville type equations with singular sources}, Handbook Diff. Eqs., North Holland,
	Amsterdam, Stationary partial differential equations, {\bf I} (2004), 491-592.
	
	\bibitem{taran-1} G. Tarantello, Multiple condensate solutions for the Chern-Simons-Higgs theory, \emph{J. Math. Phys.}, \textbf{37} (1996), 3769-3796.
	
	\bibitem{taran-2} G. Tarantello, Self-Dual Gauge Field Vortices: An Analytical Approach, PNLDE 72, Birkhauser Boston, Inc., Boston, MA, 2007.
	
	\bibitem{taran-3} G. Tarantello, Blow-up analysis for a cosmic strings equation, \emph{Jour. Funct. Analysis}, \textbf{272} (1) (2017) 255-338.
	
	\bibitem{taran-4} G. Tarantello, Asymptotics for minimizers of a Donaldson functional and mean curvature 1-immersions of surfaces into hyperbolic 3-manifolds, \emph{Adv. Math.}, \textbf{425} (2023) 109090.
	
	\bibitem{troy} M. Troyanov, Prescribing curvature on compact surfaces with conical singularities, \emph{Trans. Amer. Math. Soc.}, \textbf{324} (1991), 793-821.
	
	\bibitem{TuY} A. Tur, V. Yanovsky, Point vortices with a rational necklace: New exact stationary solutions
	of the two-dimensional Euler equation, \emph{Phys. Fluids }, \textbf{16} (2004), 2877–2885.
	
	\bibitem{wei-zhang-19} J.C. Wei, L. Zhang, Estimates for Liouville equation with quantized singularities, {\em Adv. Math.}, \textbf{380} (2021), 107606.
	
	\bibitem{wei-zhang-22} J.C. Wei, L. Zhang, Vanishing estimates for Liouville equation with quantized singularities, {\em Proc. Lond. Math. Soc.}, \textbf{3} (2022), 1-26.
	
	\bibitem{wei-zhang-jems} J.C.Wei, L. Zhang,
	Laplacian Vanishing Theorem for Quantized Singular Liouville Equation. To appear in {\em Journal of European Mathematical Society.}
	
	\bibitem{wolan2} G. Wolansky, { On the evolution of self-interacting clusters and applications to semilinear equations with exponential nonlinearity},
	\emph{J. Anal. Math.} {\bf 59} (1992), 251-272.
	
	\bibitem{wolan} G. Wolansky, On steady distributions of self-attracting clusters under friction and fluctuations, \emph{Arch. Rational Mech. An.}, \textbf{119} (1992), 355--391.
	
	\bibitem{wu-zhang-ccm} L.N. Wu, L. Zhang, Uniqueness of bubbling solutions of mean field equations with non-quantized singularities. {\em  Commun. Contemp. Math.} 23 (2021), no. 4, Paper No. 2050038, 39 pp.
	
	
	
	
	
	\bibitem{y-yang} Y. Yang, Solitons in Field Theory and Nonlinear Analysis, Springer Monographs in Mathematics, Springer, New York, 2001.
	
	\bibitem{zhang1} L. Zhang, Blow-up solutions of some nonlinear elliptic equations involving
	exponential nonlinearities, \emph{Comm. Math. Phys}, \textbf{268} (2006), 105-133.
	
	\bibitem{zhang2} L. Zhang, Asymptotic behavior of blow-up solutions for elliptic equations with exponential nonlinearity and singular data, \emph{Commun. Contemp. Math}, \textbf{11} (2009), 395-411.
	
\end{thebibliography}
\end{document}